\documentclass[11pt]{amsart}
\usepackage{lmodern}
\usepackage{amsmath, amsthm, amssymb, amsfonts}
\usepackage[normalem]{ulem}
\usepackage{mathrsfs}
\usepackage{verbatim} 
\usepackage{longtable}
\usepackage{mathtools}

\usepackage{tikz}
\usetikzlibrary{decorations.pathmorphing}
\tikzset{snake it/.style={decorate, decoration=snake}}
\usepackage{caption}
\usepackage{tikz-cd}
\usetikzlibrary{arrows}
\usepackage{hyperref}

\theoremstyle{plain}
\newtheorem{thm}{Theorem}[section]
\newtheorem{cor}[thm]{Corollary}
\newtheorem{lem}[thm]{Lemma}
\newtheorem{prop}[thm]{Proposition}
\newtheorem{conj}[thm]{Conjecture}

\theoremstyle{definition}
\newtheorem{defn}[thm]{Definition}

\theoremstyle{remark}
\newtheorem{rmk}[thm]{Remark}

\newcommand{\BC}{{\mathbb{C}}}

\newcommand{\BF}{{\mathbb{F}}}
\newcommand{\BG}{{\mathbb{G}}}

\newcommand{\BN}{{\mathbb{N}}}

\newcommand{\BP}{{\mathbb{P}}}
\newcommand{\BQ}{{\mathbb{Q}}}

\newcommand{\BU}{{\mathbb{U}}}

\newcommand{\BZ}{{\mathbb{Z}}}

\newcommand{\CB}{{\mathcal B}}
\newcommand{\CC}{{\mathcal C}}

\newcommand{\CE}{{\mathcal E}}
\newcommand{\CF}{{\mathcal F}}
\newcommand{\CG}{{\mathcal G}}
\newcommand{\CH}{{\mathcal H}}
\newcommand{\CI}{{\mathcal I}}
\newcommand{\CJ}{{\mathcal J}}
\newcommand{\CK}{{\mathcal K}}
\newcommand{\CL}{{\mathcal L}}

\newcommand{\CN}{{\mathcal N}}
\newcommand{\CO}{{\mathcal O}}
\newcommand{\CP}{{\mathcal P}}

\newcommand{\CR}{{\mathcal R}}

\newcommand{\CT}{{\mathcal T}}

\newcommand{\Fp}{{\mathfrak{p}}}
\newcommand{\Fq}{{\mathfrak{q}}}

\newcommand{\FC}{{\mathfrak{C}}}
\newcommand{\FF}{{\mathfrak{F}}}

\newcommand{\ch}{{\mathrm{ch}}}
\newcommand{\td}{{\mathrm{td}}}

\DeclareFontFamily{OT1}{rsfs}{}
\DeclareFontShape{OT1}{rsfs}{n}{it}{<-> rsfs10}{}
\DeclareMathAlphabet{\curly}{OT1}{rsfs}{n}{it}

\newcommand\Hom{\operatorname{Hom}}

\newcommand{\Chow}{\mathrm{CH}}
\newcommand{\Corr}{\mathrm{Corr}}
\newcommand{\Coh}{\mathrm{Coh}}

\usepackage{tikz}
\usepackage{lmodern}
\usetikzlibrary{decorations.pathmorphing}

\makeatletter
\let\@wraptoccontribs\wraptoccontribs
\begin{document}
\title[Perverse filtrations and Fourier transforms]{Perverse filtrations and Fourier transforms}
\date{\today}

\author[D. Maulik]{Davesh Maulik}
\address{Massachusetts Institute of Technology}
\email{maulik@mit.edu}

\author[J. Shen]{Junliang Shen}
%\address{Massachusetts Institute of Technology}
%\email{jlshen@mit.edu}
\address{Yale University}
\email{junliang.shen@yale.edu}

\author[Q. Yin]{Qizheng Yin}
\address{Peking University}
\email{qizheng@math.pku.edu.cn}

\begin{abstract}
We study the interaction between Fourier--Mukai transforms and perverse filtrations for a certain class of dualizable abelian fibrations. Multiplicativity of the perverse filtration and the ``Perverse $\supset$ Chern'' phenomenon for these abelian fibrations are immediate consequences of our theory.  We also show that our class of fibrations include families of compactified Jacobians of integral locally planar curves.  

Applications include the following: (a) we prove the motivic decomposition conjecture for this class (including compactified Jacobian fibrations), which generalizes Deninger--Murre's theorem for abelian schemes; (b) we provide a new proof of the $P=W$ conjecture for $\mathrm{GL}_r$; (c) we prove half of the $P=C$ conjecture concerning refined BPS invariants for the local~$\BP^2$; (d) we show that the perverse filtration for the compactified Jacobian associated with an integral locally planar curve is multiplicative, which generalizes a result of Oblomkov--Yun for homogeneous singularities.

Our techniques combine Arinkin's autoduality for coherent categories, Ng\^o's support theorem for the decomposition theorem, Adams operations in operational $K$-theory, and Corti--Hanamura's theory of relative Chow motives. 

\end{abstract}

\maketitle

\setcounter{tocdepth}{1} 

\tableofcontents
\setcounter{section}{-1}

\section{Introduction}

\subsection{Overview}
Throughout, we work over the complex numbers $\BC$. For a proper morphism $f: X \to Y$ between nonsingular varieties, the perverse truncation functor \cite{BBD} filters the derived direct image $Rf_* \BQ_X \in D^b_c(Y)$, which yields an increasing filtration on the cohomology either for the total space $X$ or for a closed fiber. Such a filtration is called the perverse filtration, which encodes important topological invariants of the map $f: X \to Y$.

In recent years, perverse filtrations have played a crucial role in the study of integrable systems, enumerative geometry, and geometric representation theory. For many interesting abelian fibrations (\emph{e.g.}~Hitchin systems), the associated perverse filtration is discovered to satisfy the following two mysterious properties:
\begin{enumerate}
    \item[(i)] it is \emph{multiplicative} with respect to the cup-product, and
    \item[(ii)] the location of a tautological class in the filtration is determined by its \emph{Chern grading} of the universal family.
\end{enumerate} 
Indeed, the $P=W$ conjecture of de Cataldo--Hausel--Migliorini for $\mathrm{GL}_r$ is equivalent to~(i,~ii) for the $\mathrm{GL}_r$-Hitchin system \cite{dCHM1, dCMS}, and was proven recently using special properties of Hitchin moduli spaces \cite{MS_PW, HMMS}. On the other hand, similar phenomena have also been found or conjectured for Lagrangian fibrations associated with compact hyper-K\"ahler manifolds \cite{SY, dCMS}, compactified Jacobians associated with torus knots \cite{OY, OY2}, and moduli of $1$-dimensional stable sheaves on~$\BP^2$ \cite{KPS}, and one would like to understand this beyond the Hitchin setting. Note that, for general proper morphisms, the perverse filtration will typically not be multiplicative even if the singularities of the fibers are mild \cite[Exercise 5.6.8]{Park}.

The purpose of this paper is to systematically study the perverse filtrations associated with abelian fibrations, and give a uniform explanation for (i, ii) that applies in different geometric settings. We show that both (i) and (ii) are essentially consequences of the \emph{duality} between derived categories of coherent sheaves and the \emph{support theorem} in the decomposition theorem package. Our work also shows the motivic nature of the decomposition theorem associated with abelian fibrations.

\subsection{Dualizable abelian fibrations}\label{Sec0.2}
In Section \ref{Sec1.4}, we introduce a class of proper morphisms, called \emph{dualizable abelian fibrations}. They are modelled on Hitchin systems and compactified Jacobian fibrations. 

Roughly, for nonsingular varieties $M$ and $B$, we say that a proper morphism $\pi: M \to B$ is a dualizable abelian fibration if it is an abelian fibration with a dual fibration $\pi^\vee: M^\vee \to B$, satisfying the following properties:

\begin{enumerate}
    \item[(a)] (Duality) the bounded derived categories of coherent sheaves on $M$ and $M^\vee$ are related by Fourier--Mukai transforms with certain nice properties similar to those of dual abelian schemes;
    \item[(b)] (Support) every simple perverse summand in the decomposition theorem for $R\pi_* \BQ_M$ has full support $B$.
\end{enumerate}
We refer to Section \ref{Sec1.4} for more precise statements concerning (a, b). The Chern character of the Fourier--Mukai kernel of (a) induces a Fourier transform in cohomology:
\[
\mathfrak{F} = \sum_k \mathfrak{F}_k: H^*(M^\vee, \BQ) \to H^*(M, \BQ), \quad \mathfrak{F}_k(H^d(M^\vee, \BQ)) \subset H^{d+2k-2g}(M, \BQ),
\]
where $g$ is the relative dimension $\dim M - \dim B$. The Fourier inverse $\mathfrak{F}^{-1}= \sum_k \FF^{-1}_k$ behaves similarly. In order for the Fourier operators to have interesting homological consequences, we further impose a vanishing condition on the relative product $M^\vee \times_B M^\vee$. We say that a dualizable abelian fibration as above satisfies \emph{the Fourier vanishing (FV)}, if 
\begin{equation}
\tag{FV} \quad \quad \FF^{-1}_i \circ \FF_j = 0,\quad  i+j < 2g.
\end{equation}

\begin{thm}\label{thm0.1}
Let $\pi: M \to B$ be a dualizable abelian fibration satisfying (FV). Then
\begin{enumerate}
    \item[(i)] (Multiplicativity) the perverse filtration associated with $\pi$ is multiplicative, \emph{i.e.},
\[
P_kH^d(M, \BQ) \times P_{k'}H^{d'}(M,\BQ) \xrightarrow{~~\cup~~} P_{k+k'}H^{d+d'}(M, \BQ);
\]
\item[(ii)] (Perverse $\supset$ Chern\footnote{The term ``Perverse $\supset$ Chern'' here means that the perversity (see Section \ref{1.3.3} for the definition of perversity) of a class $\mathfrak{F}_k(\alpha)$ is bounded by the Chern grading~$k$ associated with the Fourier transform.}) for any class $\alpha \in H^*(M^\vee, \BQ)$, we have
\[
\mathfrak{F}_k(\alpha) \in P_kH^*(M,\BQ).
\]
\end{enumerate}
\end{thm}

%\footnote{The term ``Perverse $\supset$ Chern'' here means that the perversity of a class $\mathfrak{F}_k(\alpha)$ is bounded by the Chern grading~$k$ associated with the Fourier transform. Moreover, this perversity bound is optimal in the case of Section \ref{0.3.2.2}, and is expected to be optimal in the case of Section \ref{0.3.2}.}

In fact, we prove a stronger version of Theorem \ref{thm0.1} which is of a motivic nature and specializes to sheaf-theoretic statements; see Theorem \ref{thm:main}.

Before discussing applications, we make some remarks on the condition (FV). For an abelian scheme $\pi: A \to B$ with its dual $\pi^\vee: A^\vee \to B$, the Fourier transform $\FF$ was first considered by Beauville \cite{B}. It induces a canonical (motivic) decomposition, called the \emph{Beauville decomposition}, of the cohomology or the Chow ring of the total space $A$; see Section \ref{Sec1.2} for more discussions on this construction. Once we obtain the Beauville decomposition, the statements (i, ii) above for the abelian scheme $\pi: A \to B$ are immediate consequences. A key step in Beauville's theory is the observation that the operators $\FF_k$ provide \emph{projectors}, which follows from the vanishing
\begin{equation*}\label{too_strong}
\FF^{-1}_i \circ \FF_j = 0, \quad i+j \neq 2g.
\end{equation*}
Clearly, this vanishing is stronger than (FV) in the case of abelian schemes. The reason why we do not pose this stronger condition is that it is hard to verify when the abelian fibration has singular fibers, and furthermore, it is unclear if the vanishing for $i+j>2g$ holds in general. %Therefore, we do not adapt (\ref{too_strong}) in the framework of generalizing Beauville's theory.

%to abelian fibrations with singular fibers.

On the other hand, the weaker condition (FV) is already sufficient to deduce the desired properties (i, ii) of Theorem \ref{thm0.1}. Moreover, the condition (FV) is natural: in view of Arinkin's work \cite{A1, A2}, we will explain in Section \ref{comp_jac} that (FV) holds essentially for the \emph{same reason} as the fact that the normalized Poincar\'e line bundle and the Fourier--Mukai transforms can be extended over the singular fibers for certain abelian fibrations. The following theorem is our main source of applications.

\begin{thm}\label{thm0.2}
Let $C \to B$ be a flat family of integral projective curves which admits a section through its smooth locus. Let $\pi: \overline{J}_C \to B$ be the associated compactified Jacobian fibration. If
\begin{enumerate}
    \item[(i)] every curve in the family $C \to B$ has at worst planar singularities, and
    \item[(ii)] the total space $\overline{J}_C$ is nonsingular,
\end{enumerate}
then $\pi: \overline{J}_C \to B$ is a (self-)dualizable abelian fibration which satisfies (FV).
\end{thm}

By promoting the Fourier theory from schemes to gerbes, we also prove a version of Theorems \ref{thm0.1} and \ref{thm0.2} (as well as Theorem \ref{CoHa} below) for certain twisted compactified Jacobian fibrations which are associated with families of curves admitting no section; see Corollaries~\ref{cor:main} and \ref{cor:motdectwist}.

\subsection{Applications}
We discuss in this section some applications of Theorem \ref{thm0.1}, which explain and unify some common features discovered in non-abelian Hodge theory, enumerative geometry, and representation theory.

\subsubsection{The motivic decomposition conjecture}
We first note that as a by-product of our proof of Theorem \ref{thm0.1}, we verify in Corollary \ref{motivic_decomp} the motivic decomposition conjecture of Corti--Hanamura \cite{CH} for a dualizable abelian fibration $\pi: M\to B$ satisfying (FV). This extends the motivic decomposition of abelian schemes by Deninger--Murre \cite{DM}. Combined with Theorem~\ref{thm0.2}, we obtain the following.

\begin{thm}\label{CoHa}
Let $\pi: \overline{J}_C \to B$ be a compactified Jacobian fibration as in Theorem \ref{thm0.2}. Then the decomposition of $R\pi_* \BQ_{\overline{J}_C}$ into (shifted) semisimple perverse sheaves admits a motivic lifting.
\end{thm}

We refer to Corollary \ref{motivic_decomp} for the precise statement. This verifies the motivic decomposition conjecture \cite{CH} for $\pi: \overline{J}_C\to B$, where the projectors are provided by Arinkin's Fourier--Mukai theory. See \emph{e.g.}~\cite{CDN, ACLS} for other recent developments of the motivic decomposition conjecture, and \cite{dCM?} for an unconditional construction of the (homological) motivic decomposition in the setting of motivated classes.

\subsubsection{The $P=W$ conjecture}\label{0.3.2.2}

One of our motivations and applications comes from the $P=W$ conjecture in non-abelian Hodge theory, which we briefly review.

For a projective curve $\Sigma$ of genus $g\geq 2$ and two coprime integers $r,n$, the non-abelian Hodge correspondence yields a diffeomorphism between the Dolbeault moduli space $M_{\mathrm{Dol}}$ which parameterizes rank $r$ and degree $n$ stable Higgs bundles on~$\Sigma$, and the Betti moduli space $M_B$ which is the corresponding character variety. This further induces
\begin{equation}\label{NAH}
H^*(M_{\mathrm{Dol}}, \BQ) = H^*(M_B, \BQ).
\end{equation}
In 2010, de Cataldo, Hausel, and Migliorini \cite{dCHM1} proposed a relationship between the topology of the Hitchin system $h: M_{\mathrm{Dol}} \to B$ and the Hodge theory of $M_B$ via (\ref{NAH}); more precisely, they conjectured that the perverse filtration associated with the Hitchin system is matched with the double indexed weight filtration for the character variety:
\begin{equation}\label{P=W}
\text{``}P=W\text{''}, \quad P_kH^*(M_{\mathrm{Dol}}, \BQ) = W_{2k} H^*(M_{\mathrm{B}}, \BQ).
\end{equation}
This conjecture, known as the $P=W$ conjecture, has now been proven by the first two authors~\cite{MS_PW} and Hausel--Mellit--Minets--Schiffmann \cite{HMMS} independently via different methods.  

\begin{thm}[\cite{MS_PW, HMMS}]\label{P=W!}
The $P=W$ conjecture (\ref{P=W}) holds.
\end{thm}

As we sketch now, our work in this paper provides a new proof of Theorem \ref{P=W!}.

The $P=W$ conjecture can be decomposed into three identities:
\begin{equation}\label{P=C=C=W}
P_k H^*(M_{\mathrm{Dol}}, \BQ) = C_k H^*(M_{\mathrm{Dol}}, \BQ) = C_kH^*(M_B, \BQ) = W_{2k}H^*(M_B, \BQ);
\end{equation}
%which reflect different geometric structures of the moduli spaces in non-abelian Hodge theory. 
Here $C_\bullet$ stands for the Chern filtration defined via the tautological classes of the moduli space. 

The second and third identities have been established earlier by work of Markman \cite{Markman} Hausel--Thaddeus \cite{HT1} and Shende \cite{Shende} respectively.  As a result, 
the $P=W$ conjecture is reduced to the first identity --- the ``Perverse = Chern'' phenomenon for the Hitchin system. 
Indeed, both existing proofs \cite{MS_PW, HMMS} proceed by establishing the first identity using special features of the moduli space of Higgs bundles.  

%The second equality of (\ref{P=C=C=W}) follows from the fact that the tautological classes on $M_{\mathrm{Dol}}$ and~$M_B$ are matched via non-abelian Hodge theory~\cite{HT1} and are
%are generators of the cohomology \cite{Markman}. The third identity was obtained in \cite{Shende} by calculating the mixed Hodge structure on $H^*(M_B, \BQ)$ with respect to the tautological classes. Therefore, the $P=W$ conjecture is equivalent to the first identity --- This is indeed the approach of both proofs \cite{MS_PW, HMMS}.

Moreover, in \cite{Mellit} Mellit proved the curious Hard Lefschetz for the character variety $M_B$, which further reduces the full $P=W$ (\ref{P=W}) to
\begin{equation}\label{P=W_taut}
    P_k H^*(M_{\mathrm{Dol}}, \BQ) \supset C_k H^*(M_{\mathrm{Dol}}, \BQ). 
\end{equation}

The present work grew out of an attempt to understand the geometric nature of (\ref{P=W_taut}).   The surprising interaction between the perverse filtration and the Chern classes predicts the existence of a good theory of Fourier transforms for certain abelian fibrations with singular fibers. The discovery of an analogous conjecture for the $\BP^2$ geometry (see Section \ref{0.3.2}) further suggests that (\ref{P=W_taut}) should be a phenomenon beyond Hitchin systems. This leads us to Theorem~\ref{thm0.1}. 

Conversely, we will show that (\ref{P=W_taut}) is a consequence of our Fourier theory; in combination with 
an argument from \cite{HMMS} to reduce to integral spectral curves, as we explain in Section \ref{Sec5.4}, this yields Theorem \ref{P=W!}.

\begin{rmk}
The above proof of $P=W$ suggests that, over the locus of integral spectral curves, (\ref{P=W_taut}) is a property of compactified Jacobian fibrations associated with curves with planar singularities and does not rely on the representation theory perspectives of \cite{MS_PW, HMMS}.  This approach gives a possible avenue for studying $P=W$ phenomena beyond the Higgs setting. More precisely, for any dualizable abelian fibration as in Theorem \ref{thm0.1} we may define the Chern filtration $C_k H^*(M, \BQ)$ as the span of the classes
\[
\FF_j(\alpha) \in H^*(M, \BQ), \quad j\leq k, \quad \alpha \in H^*(M^\vee, \BQ).
\]
This generalizes the Chern filtration defined in the $\mathrm{GL}_r$ case, and does not rely on Markman's generation result \cite{Markman}. Therefore, in view of (\ref{P=C=C=W}), the $P=W$ phenomenon is reduced to understanding the interaction between the Fourier transform and the weight filtration on the Betti side. Speculatively, for other Hitchin moduli spaces, this interaction may be related to the Betti geometric Langlands correspondence.
\end{rmk}

\subsubsection{Enumerative geometry of local $\BP^2$}\label{0.3.2}

A phenomenon similar to (\ref{P=W_taut}) was discovered in \cite{KPS} for the refined BPS invariants of local $\BP^2$. It has a different origin from $P=W$, since the non-abelian Hodge correspondence is not relevant here.

Fix coprime integers $r, \chi$ with $r \geq 3$. We consider the moduli $M_{r,\chi}$ of $1$-dimensional stable sheaves $F$ on $\BP^2$ with $[\mathrm{supp}(F)] = rH, \, \chi(F) = \chi$. Here~$H$ is the hyperplane class. The moduli space $M_{r,\chi}$ admits a proper Hilbert--Chow morphism
\[
h: M_{r,\chi} \to \BP H^0(\BP^2, \CO_{\BP^2}(r)), \quad F \mapsto \mathrm{supp}(F)
\]
where $\mathrm{supp}(-)$ stands for the Fitting support. This proper map induces a perverse filtration on the cohomology of $M_{r,\chi}$. Analogously to the tautological classes \cite{HT1} for the moduli of Higgs bundles, Pi and the second author \cite{PS} introduced the tautological class
\[
c_k(j) \in H^{2(k+j-1)}(M_{r,\chi}, \BQ)
\]
given by integrating the normalized $\mathrm{ch}_{k + 1}$ of a universal family over $H^j \in H^{2j}(\BP^2, \BQ)$; see Section \ref{5.3.2}. Moreover, it was shown in \cite{PS} that the first $3r-7$ tautological classes 
\[
c_k(j) \in H^{\leq 2(r - 2)}(M_{r,\chi}, \BQ)
\]
generate the total cohomology as a $\BQ$-algebra, and there is no relation in degrees $\leq 2(r-2)$. This allows us to consider the Chern filtration $C_k H^{\leq 2(r-2)}(M_{r,\chi} ,\BQ)$ spanned by
\[
\prod_{i=1}^s c_{k_i}(j_i) \in H^{\leq 2(r-2)}(M_{r,\chi}, \BQ), \quad \sum_{i=1}^s k_i \leq k.
\]
The following conjecture was proposed in \cite{KPS}, inspired by the original $P=W$ conjecture:
\begin{equation} \label{P=C}
\text{``}P=C\text{''}, \quad   P_k H^{\leq 2(r-2)}(M_{r,\chi}, \BQ) = C_k H^{\leq 2(r-2)}(M_{r,\chi}, \BQ).
\end{equation}
Perverse filtrations associated with moduli of $1$-dimensional sheaves have roots in enumerative geometry \cite{HST,KL,MT,MS}. In particular, the dimensions of the left-hand side of (\ref{P=C}) calculate the refined BPS invariants of the local Calabi--Yau threefold $\mathrm{Tot}(K_{\BP^2})$ for the curve class~$rH$. The $P=C$ conjecture (\ref{P=C}) offers a geometric explanation to an asymptotic product formula for the refined BPS invariants of the local $\BP^2$ calculated via Pandharipande--Thomas theory. We refer to \cite[Section 0]{KPS} for an introduction. 

The following theorem, which will be proven in Section \ref{Sec5.3}, is a consequence of Theorem~\ref{thm0.1}. It verifies half of the conjecture (\ref{P=C}). 

\begin{thm}\label{thm0.4}
We have
\[
P_k H^{\leq 2(r-2)}(M_{r,\chi}, \BQ) \supset C_k H^{\leq 2(r-2)}(M_{r,\chi}, \BQ).
\]
\end{thm}

Consequently, the conjecture (\ref{P=C}) is equivalent to its numerical version \cite[Conjecture 0.1]{KPS}:
\[
\dim \mathrm{Gr}^P_iH^{i+j}(M_{r,\chi}, \BQ) = \dim \mathrm{Gr}^C_iH^{i+j}(M_{r,\chi}, \BQ), \quad i+j \leq 2(r-2). 
\]
We refer to Section \ref{Sec5.3} for more discussions on the $P=C$ conjecture (\ref{P=C}) and its consequences.

\subsubsection{Compactified Jacobians}

For an integral projective curve $C_0$ with planar singularities, the compactified Jacobian $\overline{J}_{C_0}$ is an integral projective variety. These varieties behave like local analogues of Hitchin systems; in particular, the singular cohomology $H^*(\overline{J}_{C_0}, \BQ)$ admits a canonical perverse filtration 
\begin{equation}\label{perverse}
P_0H^*(\overline{J}_{C_0}, \BQ) \subset P_1H^*(\overline{J}_{C_0}, \BQ) \subset  \cdots \subset H^*(\overline{J}_{C_0}, \BQ)
\end{equation}
by \cite{MY, MS1}. When $C_0$ is a rational curve with a unique planar singular point, it was proposed by Shende that $\overline{J}_{C_0}$ should serve as the Dolbeault side of a $P=W$ conjecture with some ``Betti moduli space'' related to the link of the singularity. This proposal was motivated by the connection between algebro-geometric invariants associated with a planar singularity and topological invariants associated with the corresponding link; we refer to \cite{PT_knot, OS, ORS, STZ, Trinh, BAMY} for relevant work and recent developments. Such a~$P=W$ conjecture is wide open to the best of our knowledge; nevertheless, it predicts that the perverse filtration (\ref{perverse}) has to be multiplicative with respect to the cup-product. For special singularities $x^p=y^q$ with $(p,q)=1$, the multiplicativity was proven by Oblomkov--Yun \cite{OY, OY2} using tautological generators and representations of the rational Cherednik algebra. 

We establish the multiplicativity in full generality as a consequence of the sheaf-theoretic version of Theorem \ref{thm0.1} together with Theorem \ref{thm0.2}. This provides evidence for the $P=W$ conjecture for compactified Jacobians.

\begin{thm} \label{thm0.5}
For an integral projective curve $C_0$ with planar singularities, the perverse filtration (\ref{perverse}) is multiplicative with respect to the cup-product, \emph{i.e.},
\[
P_kH^d(\overline{J}_{C_0}, \BQ) \times P_{k'}H^{d'}(\overline{J}_{C_0}, \BQ) \xrightarrow{~~\cup~~} P_{k+k'}H^{d+d'}(\overline{J}_{C_0}, \BQ).
\]
\end{thm}

In \cite{JR}, Rennemo constructed a decomposition  
\[
H^*(\overline{J}_{C_0}, \BQ) = \bigoplus_{k,d} D_k H^{d}(\overline{J}_{C_0}, \BQ)
\]
using natural operators on the cohomology of the Hilbert schemes of points $C^{[k]}_0$, and showed that it splits the perverse filtration. He asked in \cite[Question 1.4]{JR} whether this decomposition is multiplicative with respect to the cup-product. Theorem \ref{thm0.5} implies that $D_{\leq k}$ is multiplicative, which answers affirmatively a weaker version of Rennemo's question. However, so far we do not know if the multiplicativity holds for a natural splitting of the perverse filtration. See also \cite[Conjecture~2.17]{MY} for a natural (conjectural) splitting of the perverse filtration.

\subsection{Outline of paper}
We briefly outline the contents of this paper.  In Section \ref{Sec1}, we review Beauville's work on abelian schemes and define the class of dualizable abelian fibrations which generalize it.  In Section \ref{Sec2}, we prove our first main result Theorem \ref{thm0.1}, giving multiplicativity and Perverse $\supset$ Chern statements for dualizable abelian fibrations.  Here we use the work of Corti--Hanamura on relative Chow motives which gives motivic enhancements of these results.  In Section \ref{comp_jac}, we prove Theorem \ref{thm0.2}; the key input is Arinkin's theory of Fourier--Mukai transforms for compactified Jacobians of integral locally planar curves, combined with an argument using Adams operations. In Section \ref{Sec:twist} we use twisted sheaves to extend this to families of curves without a section.  In Section \ref{applications}, we explain the proof of Theorem \ref{thm0.4} and give the details of the new proof of Theorem \ref{P=W!}.

\subsection{Acknowledgements}
We would like to thank Giuseppe Ancona, Vicky Hoskins, Daniel Huybrechts, Max Lieblich, Simon Pepin Lehalleur, Weite Pi, and Bin Wang for useful discussions. We are especially grateful to Dima Arinkin for conversations on his work, and to the anonymous referees for careful reading and for numerous suggestions that greatly improved the exposition of the paper. J.S.~was supported by the NSF grant DMS-2301474. Q.Y.~was supported by the NSFC grants 11831013 and 11890661.

%\subsection{Dualizable abelian fibrations}

%\subsection{Relations to other work}

\section{Dualizable abelian fibrations} \label{Sec1}

\subsection{Overview}
In this section, we first review the Beauville decomposition for an abelian scheme as a motivating example. Then we introduce the notion of dualizable abelian fibrations which is a suitable framework for generalizing Beauville's theory.

\subsection{The Beauville decomposition}\label{Sec1.2}

To motivate our method, we start with a brief review of the Beauville decomposition \cite{B,DM}. We focus on the homological version for convenience.

Let $\pi: A \to B$ be an abelian scheme of relative dimension $g$ over a nonsingular base variety~$B$. The ``multiplication by $N$'' map $[N]: A \to A$ induces a canonical decomposition of the cohomology
\begin{equation}\label{beau-decomp}
H^*(A, \BQ) = \bigoplus_{i} H^*_{(i)}(A, \BQ),
\end{equation}
which we call the (homological) \emph{Beauville decomposition}. Here 
\[
H^*_{(i)}(A, \BQ) = \{ \alpha \in H^*(A, \BQ) \,|\, [N]^*\alpha = N^i\alpha\}
\]
is an eigenspace of the pullback map $[N]^*:H^*(A, \BQ) \to H^*(A,\BQ)$. The decomposition \eqref{beau-decomp} is motivic (\emph{i.e.}~induced by projectors) and provides a splitting of the Leray filtration\footnote{Here $L_kH^d(A, \BQ) = L^{d-k}H^d(A, \BQ)$ with respect to the usual convention for the Leray filtration.} $L_\bullet H^*(A, \BQ)$ associated with the fibration $\pi: A \to B$, where
\[
L_kH^*(A, \BQ) = \bigoplus_{i \leq k} H^*_{(i)}(A, \BQ).
\]
It is also multiplicative with respect to the cup-product, \emph{i.e.},
\begin{equation} \label{eq:multbeauville}
H^*_{(i)}(A, \BQ) \times H^*_{(i')}(A, \BQ) \xrightarrow{~~\cup~~} H^*_{(i + i')}(A, \BQ).
\end{equation}
In fact, since for $\alpha, \beta \in H^*(A, \BQ)$ we have $[N]^*(\alpha \cup \beta) = [N]^*\alpha \cup [N]^*\beta$, the multiplicativity follows simply by comparing the eigenvalues. For our purpose, we would like to extend (\ref{beau-decomp}) to abelian fibrations with singular fibers. We give another description of (\ref{beau-decomp}) using duality and Fourier transforms.

We consider the dual abelian scheme $\pi^\vee: A^\vee \to B$, and denote by $\CL$ the normalized Poincar\'e line bundle on $A^\vee \times_B A$. Then the correspondence given by the Chern character
\[
\mathrm{ch}(\CL) = \exp( c_1(\CL) )
\]
induces the Fourier transform 
\[
\mathfrak{F}: H^*(A^\vee, \BQ) \to H^*(A, \BQ), \quad \alpha \mapsto p_{2*}(p_1^* \alpha \cup \mathrm{ch}(\CL))
\]
with $p_i$ the natural projections from $A^\vee \times_B A$. The inverse transform $\mathfrak{F}^{-1}$ is induced by $(-1)^g\mathrm{ch}(\CL^\vee)$. Using the Fourier transform $\mathfrak{F}$ and its inverse $\mathfrak{F}^{-1}$, we obtain a canonical decomposition of the cohomology $H^*(A, \BQ)$ which recovers the decomposition (\ref{beau-decomp}). Indeed, for any class $\alpha \in H^d(A, \BQ)$, we have
\[
\mathfrak{F}^{-1}(\alpha) = \sum_{i} \alpha_{(i)}^\vee, \quad  \alpha^\vee_{(i)} \in H^{d + 2g - 2i}(A^\vee, \BQ).
\]
This yields
\begin{equation}\label{FM_decomp}
\alpha = \FF\circ\FF^{-1}(\alpha) = \sum_i \alpha_i, \quad \alpha_i = \mathfrak{F}\left( \alpha^\vee_{(i)}\right) \in H^d_{(i)}(A, \BQ).
\end{equation}
As was explained in \cite{DM}, the decomposition (\ref{FM_decomp}) is motivic with projectors given by the Chern character of the normalized Poincar\'e line bundle $\CL$. Consequently, the Beauville decomposition~(\ref{beau-decomp}) is governed by the normalized Poincar\'e line bundle $\CL$.

Moreover, the multiplicativity \eqref{eq:multbeauville} can be seen from Fourier transforms together with~\eqref{FM_decomp}. For this we use a new product structure on $H^*(A^\vee, \BQ)$:
\[
H^d(A^\vee, \BQ) \times H^{d'}(A^\vee, \BQ) \xrightarrow{~~*~~} H^{d + d' - 2g}(A^\vee, \BQ), \quad (\alpha^\vee, \beta^\vee) \mapsto \mu_*(q_1^*\alpha^\vee \cup q_2^*\beta^\vee),
\]
where $\mu: A^\vee \times_B A^\vee \to A^\vee$ is the addition map and $q_i: A^\vee \times_B A^\vee \to A^\vee$ are the natural projections. We shall call $*$ the \emph{convolution} since it is Fourier-dual to the cup product, \emph{i.e.},
\[
\FF(\alpha^\vee * \beta^\vee) = \FF(\alpha^\vee) \cup \FF(\beta^\vee).
\]
It is also known as the \emph{Pontryagin product} since the definition involves the group structure of~$A^\vee$. Now given two classes $\alpha \in H^d_{(i)}(A, \BQ)$ and $\beta \in H^{d'}_{(i')}(A, \BQ)$, we have by \eqref{FM_decomp}
\[
\FF^{-1}(\alpha) \in H^{d + 2g - 2i}(A^\vee, \BQ), \quad \FF^{-1}(\beta) \in H^{d' + 2g - 2i'}(A^\vee, \BQ),
\]
hence $\FF^{-1}(\alpha) * \FF^{-1}(\beta) \in H^{d + d' + 2g - 2i - 2i'}(A^\vee, \BQ)$. Applying \eqref{FM_decomp} again, we find
\[
\alpha \cup \beta = \FF\left(\FF^{-1}(\alpha) * \FF^{-1}(\beta)\right) \in H^{d + d'}_{(i + i')}(A, \BQ)
\]
which shows the multiplicativity \eqref{eq:multbeauville}. Furthermore, we consider 
\[
\FF_i := \mathrm{ch}_i(\CL): H^d(A^\vee, \BQ) \to H^{d + 2i - 2g}(A, \BQ).
\]
Then \eqref{FM_decomp} also implies that
\[
\FF_i\left(H^*(A^\vee, \BQ)\right) = H^*_{(i)}(A, \BQ),
\]
which serves as a prototype for the ``Perverse $\supset$ Chern'' statement of Theorem \ref{thm0.1}.

\subsection{Terminology and notation}
We fix some terminology and notation in this section to prepare for further discussions. 

\subsubsection{Abelian fibrations}\label{Sec1.3.1} We say that $\pi: M \to B$ is an abelian fibration if both $M$ and $B$ are nonsingular and irreducible, $\pi$ is proper with equidimensional fibers (hence flat), and $M$ contains an open subset $P$ which is a smooth commutative $B$-group scheme $\pi: P \to B$  whose restriction to a certain open subset
\[
\pi_U: P_U \to U \subset B 
\]
is an abelian scheme. 

%An abelian fibration $\pi: M\to B$ is called \emph{regular}, if the codimension of the ``boundary'' is bounded by the relative dimension,
%\[
%\mathrm{codim}_M(M \setminus P) \geq \dim M - \dim B.
%\]
%Although we do not have the ``multiplication by $N$'' map of Section \ref{Sec1.2} for a general (regular) abelian fibration $\pi: M \to B$, such an operation is well-defined on a large open subset:
%\[
%[N]: P \to P.
%\]

We say that $\pi^\vee: M^\vee \to B$ is dual to the abelian fibration $\pi: M\to B$, if $\pi^\vee: M^\vee \to B$ is an abelian fibration and there exists an open subset $U\subset B$ over which $\pi$ and $\pi^\vee$ form dual abelian schemes. Denote by $P^\vee \subset M^\vee$ the open commutative $B$-group scheme associated with~$\pi^\vee$. We say that 
\[
{\CP} \in D^b\mathrm{Coh}(M^\vee\times_B M)
\]
is a \emph{normalized Poincar\'e complex}, if the restriction of $\CP$ to $M_U^\vee \times_U M_U$ recovers the normalized Poincar\'e line bundle $\CL$. Here $M_U, M_U^\vee$ are dual abelian schemes over some open $U\subset B$.

%\begin{enumerate}
%    \item[(i)] the restriction of $\CP$ to the open subset $P^\vee \times_B M$ is a line bundle, which recovers the normalized Poincar\'e line bundle $\CL$ if we further restrict it to $M_U^\vee \times_U M_U$ with $M_U, M_U^\vee$ dual abelian schemes over some open $U\subset B$, and
%    \item[(ii)] the line bundle $\CP|_{P^\vee \times_B M}$ is compatible with the ``multiplication by $N$'' map of the first factor
%    \[
%    [N]: P^\vee \times_B M \to P^\vee \times_B M    \]
%    in the sense that for any $N$, we have
%    \begin{equation} \label{eq:multN}
%    [N]^*(\CP|_{P^\vee \times_B M} ) = (\CP|_{P^\vee \times_B M} )^{\otimes N}.
 %   \end{equation}
%\end{enumerate}

\subsubsection{Coherent derived categories} \label{Sec1.3.2}
Here all morphisms are assumed to be proper. For two morphisms 
\[
f_1: X_1 \to B, \quad f_2: X_2 \to B
\]
with $X_1, X_2, B$ nonsingular, an object $\CK \in D^b\mathrm{Coh}(X_1 \times_B X_2)$ induces a Fourier--Mukai transform
\[
\mathrm{FM}_{\CK}: D^b\mathrm{Coh}(X_1) \to D^b\mathrm{Coh}(X_2),\quad F \mapsto p_{2*}(p_1^*F \otimes \CK).
\]
Here $p_i$ are the natural projections from $X_1\times_BX_2$ and throughout this paper all functors between coherent (resp.~constructible) categories are derived.

For $f_3: X_3 \to B$ with $X_3$ nonsingular, the composition of two objects 
\[
\CK \in D^b\mathrm{Coh}(X_1\times_B X_2), \quad \CK' \in D^b\mathrm{Coh}(X_2\times_BX_3)
\]
is defined under certain flatness assumptions. For example, if $f_1$ and $f_3$ are both flat then
\begin{equation} \label{eq:compFM}
\CK' \circ \CK := p_{13*}\delta^*(\CK \boxtimes \CK' ) \in D^b\mathrm{Coh}(X_1\times_B X_3),
\end{equation}
where $p_{13}$ is the projection to the first and the third factors and
\[
\delta^*: D^b\mathrm{Coh}((X_1\times_BX_2) \times (X_2\times_BX_3)) \to D^b\mathrm{Coh}(X_1\times_B X_2\times_BX_3 )
\]
is the pullback along the regular embedding $X_1\times_B X_2\times_BX_3 \hookrightarrow (X_1\times_BX_2) \times (X_2\times_BX_3)$ base-changed from $\Delta_{X_2}: X_2\hookrightarrow X_2\times X_2$.\footnote{Without flatness (or derived algebraic geometry), the object $\CK' \circ \CK$ defined above may be unbounded or ill-behaved.} In this case the composition of the two Fourier--Mukai transforms
\[
\mathrm{FM}_\CK: D^b\mathrm{Coh}(X_1) \to D^b\mathrm{Coh}(X_2), \quad \mathrm{FM}_{\CK'}: D^b\mathrm{Coh}(X_2) \to D^b\mathrm{Coh}(X_3)
\]
is induced by the Fourier--Mukai kernel $\CK' \circ \CK$.

\begin{rmk}
We also note that when descended to $K$-theory, compositions make sense without flatness and for arbitrary objects
\[
\CK \in K_*(X_1\times_B X_2), \quad \CK' \in K_*(X_2\times_BX_3).
\]
Here $K_*(-)$ stands for the Grothendieck group of coherent sheaves. We define
\begin{equation} \label{eq:compFMK}
\CK' \circ \CK := p_{13*}\delta^!( \CK \boxtimes \CK' ) \in K_*(X_1\times_B X_3),
\end{equation}
where
\[
\delta^!: K_*((X_1\times_BX_2) \times (X_2\times_BX_3)) \to K_*(X_1\times_B X_2\times_BX_3 )
\]
is the refined Gysin pullback with respect to the regular embedding $\Delta_{X_2}: X_2\hookrightarrow X_2\times X_2$; see \cite[Section 3]{AP}. The compatibility between \eqref{eq:compFM} and \eqref{eq:compFMK} (when the former is well-defined) is straightforward.
\end{rmk}

For any nonsingular $B$-scheme $X$, the identity $\mathrm{id}_{D^b\mathrm{Coh}(X)}: D^b\mathrm{Coh}(X) \to D^b\mathrm{Coh}(X)$ is induced by the structure sheaf of the relative diagonal $\CO_{\Delta_{X/B}} \in D^b\mathrm{Coh}(X \times_B X)$. We say that two~objects
\[
\CK \in D^b\mathrm{Coh}(X_1 \times_B X_2), \quad \CK' \in D^b\mathrm{Coh}(X_2 \times_B X_1)
\]
are \emph{inverse}, if 
\[
\CK' \circ \CK \simeq \CO_{\Delta_{X_1/B}}, \quad \CK \circ \CK' \simeq \CO_{\Delta_{X_2/B}}.
\]
Consequently, the Fourier--Mukai transforms $\mathrm{FM}_{\CK}$ and $\mathrm{FM}_{\CK'}$ are both derived equivalences which are inverse to each other.

We say (by abuse of notation) that an object $F \in D^b\mathrm{Coh}(X)$ is of codimension $c$, if the support of every nontrivial cohomology $\CH^i(F)$ is of codimension $\geq c$.

\subsubsection{Constructible derived categories}\label{1.3.3}

Let $f: X \to B$ be a proper morphism between nonsingular varieties with equidimensional fibers. By the decomposition theorem \cite{BBD}, we have
\[
f_* \BQ_X \simeq \bigoplus_{i} {^\mathfrak{p}}\CH^i(f_* \BQ_X)[-i] \in D^b_c(B)
\]
where $D^b_c(-)$ stands for the bounded derived category of constructible sheaves and all functors are derived. Each perverse cohomology on the right-hand side is a semisimple perverse sheaf. We say that $f$ has \emph{full support}, if each simple summand of ${^\mathfrak{p}}\CH^i(f_* \BQ_X)$ has support $B$.

Next, we introduce the main characters of this paper --- perverse filtrations. The sequence of perverse truncation functors yield morphisms
\[
{^\mathfrak{p}}\tau_{\leq k+\dim B} f_* \BQ_X \to f_*\BQ_X
\]
which induces both the \emph{global} and the \emph{local} perverse filtrations.\footnote{Here we are following the indexing convention of \cite{dCHM1} -- the resulting perverse filtrations begin from $P_0$.} More precisely, the global perverse filtration on the total cohomology of $X$ is given by
\[
P_kH^d(X, \BQ) : = \mathrm{Im}\left\{H^{d}(B, {^\mathfrak{p}}\tau_{\leq k+\dim B} f_* \BQ_X) \to H^d(X, \BQ) \right\},
\]
and, for a closed point $b \in B$, the local perverse filtration on the cohomology of the closed fiber $X_b$ is given by
\[
P_{k}H^d(X_b, \BQ)  := \CH^d({^\mathfrak{p}}\tau_{\leq k+\dim B} f_* \BQ_X)_b.
\]
\emph{A priori}, the local perverse filtration for $X_b$ depends on the total family $f: X\to B$.

We say that a class $\alpha \in H^d(X, \BQ)$ has \emph{perversity} $k$ if 
\[
\alpha \in P_kH^d(X, \BQ).
\]

%\[
%\alpha \in \left(P_kH^d(X, \BQ) \setminus P_{k-1}H^d(X, \BQ) \right) \cup \{0\}.
%\]

\subsection{Dualizable abelian fibrations}\label{Sec1.4}

Dualizable abelian fibrations are extensions of abelian schemes involving singular fibers, for which most properties given by the Fourier transform described in Section \ref{Sec1.2} may be generalized.

%In the following, we define dualizable abelian fibrations which will play a key role, Then we show in Proposition \ref{prop1.2} that abelian schemes provide examples of dualizable abelian fibrations.

\begin{defn}[Dualizable abelian fibration]
Let $\pi: M \to B$ be an abelian fibration of relative dimension $g$, \emph{i.e.}~$g = \dim M - \dim B$.
We say that $\pi: M\to B$ is a dualizable abelian fibration, if it has a dual abelian fibration $\pi^\vee: M^\vee \to B$ (see Section \ref{Sec1.3.1})  satisfying the following.

\begin{enumerate}
    \item[(a1)] (Poincar\'e) There is a normalized Poincar\'e complex~${\CP} \in D^b\mathrm{Coh}(M^\vee \times_B M)$ which admits an inverse ${\CP}^{-1} \in D^b\mathrm{Coh}(M \times_B M^\vee)$, \emph{i.e.},
    \begin{equation} \label{eq:dual}
    {\CP}^{-1} \circ {\CP} \simeq \CO_{\Delta_{M^\vee/B}}, \quad   {\CP}\circ \CP^{-1} \simeq \CO_{\Delta_{M/B}}.
    \end{equation}
    \item[(a2)] (Convolution) There is an object $\CK \in D^b\mathrm{Coh}(M^\vee \times_B M^\vee \times_B M^\vee)$ of codimension $g$ (see Section \ref{Sec1.3.2} for the definition of codimension), which satisfies
    \begin{equation}\label{convolution}
    {\CP} \circ \CK \simeq \CO_{\Delta^\mathrm{sm}_{M/B}}\circ ({\CP} \boxtimes {\CP} ) \in D^b\mathrm{Coh}(M^\vee \times_B M^\vee \times_B M).
    \end{equation}
Here $\Delta^\mathrm{sm}_{M/B}$ is the small diagonal of the relative triple product $M\times_B M \times_B M$, and the notation $\CO_{\Delta^\mathrm{sm}_{M/B}}\circ ({\CP} \boxtimes {\CP} )$ stands for composing $\CP$ with $\CO_{\Delta^\mathrm{sm}_{M/B}}$ via the first (resp.~second) factor of $M \times_B M \times_B M$. We shall refer to $\CK$ as the \emph{convolution kernel} as it induces a product structure on $D^b\mathrm{Coh}(M^\vee)$ which is Fourier--Mukai-dual to the tensor product of $D^b\mathrm{Coh}(M)$.
    \item[(b)] (Support) The morphism $\pi$ has full support.
\end{enumerate}
\end{defn}

Here (a1, a2) serve as the detailed version of (a) in Section \ref{Sec0.2}. We list some consequences of (a1, a2) in terms of Fourier--Mukai transforms. Using (a1), we can define the Fourier--Mukai transforms $\mathrm{FM}_{\CP}: D^b\mathrm{Coh}(M^\vee) \to D^b\mathrm{Coh}(M)$ and its inverse
\[
\mathrm{FM}^{-1}_{\CP} = \mathrm{FM}_{{\CP}^{-1}}: D^b\mathrm{Coh}(M) \to D^b\mathrm{Coh}(M^\vee).
\]
Next, we define a convolution product via $\CK$ in (a2):
\[
\ast: D^b\mathrm{Coh}(M^\vee) \times D^b\mathrm{Coh}(M^\vee) \to D^b\mathrm{Coh}(M^\vee), \quad F_1 \ast F_2 = p_{3*}( p_1^*F_1 \otimes p_2^* F_2 \otimes \CK).
\]
Then (\ref{convolution}) implies the interaction between Fourier--Mukai transforms and the convolution:
\begin{equation*}\label{FM_Conv}
\mathrm{FM}_{\CP}(F_1) \otimes\mathrm{FM}_{\CP}(F_2) \simeq \mathrm{FM}_{\CP}(F_1 \ast F_2),\quad F_1,~F_2 \in D^b\mathrm{Coh}(M^\vee).
\end{equation*}

We also note that dualizable abelian fibrations extend abelian schemes.

\begin{prop}\label{prop1.2}
An abelian scheme $\pi: A\to B$ is dualizable in the sense above.
\end{prop}

\begin{proof}
The abelian schemes $\pi: A \to B$ and $\pi^\vee: A^\vee \to B$ form dual abelian fibrations with $\CP$ given by the normalized Poincar\'e line bundle $\CL$. The properties concerning the convolution were proven in \cite[(3.7)]{Mukai}. In particular, restricting over a point $b\in B$ with fiber $A^\vee_b = \mathrm{Pic}^0(A_b)$, the convolution kernel is the structure sheaf of the locus
\[
\{(l_1,l_2,l_3)\in \mathrm{Pic}^0(A_b)^{\times 3}\,|\, l_1 \otimes l_2 \simeq l_3\}.
\]
This is clearly a codimension $g$ locus over $B$. Finally, we note that each simple summand of~$R^i\pi_* \BQ_A$ is a local system with support $B$, which guarantees (b).
\end{proof}

Another class of dualizable fibrations is given by compactified Jacobian fibrations, which we will discuss systematically in Section \ref{comp_jac}. Arinkin's theory of the Poincar\'e sheaf \cite{A1, A2} plays a crucial role. We will also treat twisted compactified Jacobian fibrations in Section \ref{Sec:twist}. 

%In general, we expect that the geometric Langlands correspondence endows structures of twisted dualizable abelian fibrations for Hitchin systems.

\section{Multiplicativity and perversity}\label{Sec2}

\subsection{Overview}
In this section, we apply Fourier transforms to study the perverse filtrations for dualizable abelian fibrations. After the necessary preparations, we state our main result in Theorem~\ref{thm:main} which is an extension of the Beauville decomposition and a motivic enhancement of the multiplicativity and the ``Perverse $\supset$ Chern'' statements of Theorem \ref{thm0.1}. Along the way, we note in Corollary \ref{motivic_decomp} that our approach actually confirms the motivic decomposition conjecture of Corti--Hanamura \cite{CH} for a dualizable abelian fibration satisfying the Fourier vanishing. The proofs are given in Section \ref{Sec2.5}.

\subsection{Relative correspondences}
As our strategy involves passing from coherent derived categories to cycle classes, we find it convenient to use the language of \emph{relative Chow motives} developed by Corti--Hanamura \cite{CH}. This notion has its advantage in the various realizations: Chow groups, mixed Hodge modules, constructible sheaves, and particularly, global cohomology equipped with mixed Hodge structures. Some of these realizations are relevant to our results while others might draw potential interests. We begin with a quick review of the theory.

\subsubsection{Relative Chow motives} \label{Sec2.2.1}
We fix a nonsingular base variety $B$. All Chow groups $\Chow_*(-)$ are taken with $\BQ$-coefficients. Let $f_1: X_1 \to B, f_2: X_2 \to B$ be two proper morphisms with $X_1, X_2$ nonsingular. If $X_2$ is equidimensional,\footnote{If $X_2$ is not equidimensional, we decompose $X_2 = \sqcup_\alpha X_{2, \alpha}$ and define the group of correspondences accordingly.} we define the group of degree $k$ relative correspondences between $X_1$ and $X_2$ to be
\begin{equation*} \label{eq:corr}
\Corr^k_B(X_1, X_2) := \Chow_{\dim X_2 - k} (X_1 \times_B X_2).
\end{equation*}
Compositions of correspondences are similar to the compositions of Fourier--Mukai kernels in~\eqref{eq:compFMK}. For $f_3: X_3 \to B$ proper with $X_3$ nonsingular, and
\[
\Gamma \in \Corr_B^k(X_1, X_2), \quad \Gamma' \in \Corr_B^{k'}(X_2, X_3),
\]
we set
\begin{equation} \label{eq:compCH}
\Gamma'\circ\Gamma := p_{13*}\delta^!(\Gamma \times \Gamma') \in \Corr_B^{k + k'}(X_1, X_3),
\end{equation}
where $p_{13}$ is the projection to the first and the third factors and
\[
\delta^! : \Chow_*((X_1 \times_B X_2) \times (X_2 \times_B X_3)) \to \Chow_{* - \dim X_2} (X_1 \times_B X_2 \times_B X_3)
\]
is the refined Gysin pullback with respect to the regular embedding $\Delta_{X_2} : X_2 \hookrightarrow X_2 \times X_2$. For $X$ nonsingular and proper over $B$, a self-correspondence $\mathfrak{p} \in \Corr_B^0(X, X)$ is a \emph{projector} if $\mathfrak{p} \circ \mathfrak{p} = \mathfrak{p}$.

The category of relative Chow motives over $B$, denoted $\mathrm{CHM}(B)$, consists of triples $(X, \mathfrak{p}, m)$ where $X$ is nonsingular and proper over $B$, $\mathfrak{p} \in \Corr^0_B(X, X)$ is a projector, and $m$ is an integer taking care of Tate twists. Typical examples are given by the motive of $X$, \emph{i.e.},
\[
h(X) := (X, [\Delta_{X/B}], 0).
\]
Morphisms between $(X_1, \mathfrak{p}, m)$ and $(X_2, \mathfrak{q}, n)$ are defined by
\[
\mathrm{Hom}_{\mathrm{CHM}(B)}((X_1, \mathfrak{p}, m), (X_2, \mathfrak{q}, n)) := \mathfrak{q} \circ \Corr^{n - m}_B(X_1, X_2) \circ \mathfrak{p}.
\]
The category $\mathrm{CHM}(B)$ is additive and pseudo-abelian by construction. Note that however, it admits no tensor product as relative products of nonsingular $B$-schemes are often singular.

\subsubsection{Homological realizations}

In this paper we are mostly concerned with homological realizations of $\mathrm{CHM}(B)$.\footnote{We refer to \cite{GHM} for the Chow realization.} A key observation in \cite{CH} is that if $f_1 : X_1 \to B$ and $f_2: X_2 \to B$ are two morphisms with $f_1$ proper and $X_2$ nonsingular, then there is a natural isomophism
\begin{equation} \label{eq:bm}
\phi: \Hom_{D^b_c(B)}(f_{1*}\BQ_{X_1}[i], f_{2*}\BQ_{X_2}[j]) \xrightarrow{~~\simeq~~} H^{\mathrm{BM}}_{2\dim X_2 + i - j}(X_1 \times_B X_2, \BQ). 
\end{equation}

The isomorphism $\phi$ is compatible with compositions on both sides in the following sense. For $i = 1, 2, 3$, let $f_i: X_i \to B$ be morphisms with $f_1, f_2$ proper and $X_2, X_3$ nonsingular. Given two morphisms in $D^b_c(B)$:
\[
u: f_{1*}\BQ_{X_1}[i] \to f_{2*}\BQ_{X_2}[j], \quad v: f_{2*}\BQ_{X_2}[j] \to f_{3*}\BQ_{X_3}[k],
\]
we have
\begin{equation} \label{eq:compcomp}
\phi(v \circ u) = \phi(v) \circ \phi(u) \in H^{\mathrm{BM}}_{2\dim X_3 + i - k}(X_1 \times_B X_3, \BQ)
\end{equation}
where the definition of compositions in Borel--Moore homology is identical to \eqref{eq:compCH}; see \cite[Lemmas~2.21 and 2.23]{CH}.

Consequently, by the cycle class map
\[
\mathrm{cl}: \Chow_*(-) \to H_{2*}^{\mathrm{BM}}(-, \BQ)
\]
and by showing that $D^b_c(B)$ is pseudo-abelian \cite[Lemma 2.24]{CH}, Corti--Hanamura obtained a well-defined realization functor
\begin{equation} \label{eq:real}
\mathrm{CHM}(B) \to D^b_c(B)
\end{equation}
sending $(f: X \to B, \mathfrak{p}, m)$ to $\mathfrak{p}_*(f_*\BQ_X[2m])$. The functor further specializes by taking global cohomology $H^*(-)$.

\subsubsection{Multiplicative structure} \label{Sec2.2.3}
We now discuss the multiplicative structure of the motive $h(X) = (f: X \to B, [\Delta_{X/B}], 0)$ and its realizations. With no tensor product in $\mathrm{CHM}(B)$, we define manually a binary operation of motives as follows. For $i = 1, 2, 3$, let $f_i: X_i \to B$ be proper morphisms with $X_i$ nonsingular. If $X_3$ is equidimensional, we define the group of degree $k$ binary relative correspondences
\[
\Corr_B^k(X_1, X_2; X_3) := \Chow_{\dim X_3 - k} (X_1 \times_B X_2 \times_B X_3).
\]
For $i = 1, 2, 3$, let $g_i: Y_i \to B$ be three more proper morphisms with $Y_i$ nonsingular. Given four correspondences
\[
\Gamma_1 \in \Corr_B^{k_1}(Y_1, X_1), \quad \Gamma_2 \in \Corr_B^{k_2}(Y_2, X_2), \quad \Gamma_3 \in \Corr_B^{k_3}(X_3, Y_3), \quad \Gamma \in \Corr_B^{k}(X_1, X_2; X_3),
\]
there is the composition
\[
\Gamma_3 \circ \Gamma \circ (\Gamma_1 \times \Gamma_2) \in \Corr_B^{k + k_1 + k_2 + k_3}(Y_1, Y_2; Y_3).
\]
Here the notation stands for composing $\Gamma_1, \Gamma_2$ with $\Gamma$ via the $X_1$ and $X_2$ factors, and then with $\Gamma_3$ via the $X_3$ factor, which makes sense since $X_1, X_2, X_3$ are all nonsingular. A binary morphism of motives
\begin{equation} \label{eq:binary}
(X_1, \mathfrak{p}, m) \times (X_2, \mathfrak{q}, n) \to (X_3, \mathfrak{r}, p)
\end{equation}
is then given by a binary relative correspondence
\[
\mathfrak{r} \circ \Gamma \circ (\mathfrak{p} \times \mathfrak{q}), \quad \Gamma \in \Corr_B^{p - m - n}(X_1, X_2; X_3).
\]

\begin{rmk}
One can similarly define multinary morphisms of relative Chow motives, or go even further to equip $\mathrm{CHM}(B)$ with the structure of a \emph{pseudo-tensor category} in the sense of Beilinson--Drinfeld \cite[Section 1.1]{BD}. But for our purpose we shall concentrate on the usual morphisms and the binary ones.
\end{rmk}

To see the relation of \eqref{eq:binary} with the realization functor \eqref{eq:real}, we observe that by \eqref{eq:bm} composed with the K\"unneth formula, there is a natural isomorphism
\[
\Phi: \Hom_{D^b_c(B)}(f_{1*}\BQ_{X_1}[i] \otimes f_{2*}\BQ_{X_2}[j], f_{3*}\BQ_{X_3}[k]) \xrightarrow{~~\sim~~} H^{\mathrm{BM}}_{2\dim X_3 + i + j - k}(X_1 \times_B X_2 \times_B X_3, \BQ).
\]
The compatibility of $\Phi$ with compositions on both sides is summarized in the following lemma. We include a short proof as we could not find it in the literature.

\begin{lem}
For $i = 1, 2, 3$, let $f_i : X_i \to B$ be as above and let $g_i : Y_i \to B$ be morphisms with $g_1, g_2$ proper and $Y_3$ nonsingular. Given four morphisms in $D^b_c(B)$:
\begin{gather*}
u: g_{1*}\BQ_{Y_1}[i'] \to f_{1*}\BQ_{X_1}[i], \quad v: g_{2*}\BQ_{Y_2}[j'] \to f_{2*}\BQ_{X_2}[j], \quad w: f_{3*}\BQ_{X_3}[k] \to g_{3*}\BQ_{Y_3}[k'],\\
T: f_{1*}\BQ_{X_1}[i] \otimes f_{2*}\BQ_{X_2}[j] \to f_{3*}\BQ_{X_3}[k],
\end{gather*}
we have
\[
\Phi(w \circ T \circ(u \otimes v)) = \phi(w) \circ \Phi(T) \circ(\phi(u) \times \phi(v)) \in H^{\mathrm{BM}}_{2\dim Y_3 + i' + j' - k'}(Y_1 \times_B Y_2 \times_B Y_3, \BQ).
\]
\end{lem}

\begin{proof}
We decompose the lemma into three statements:
\begin{gather}
\Phi(T \circ u) = \Phi(T) \circ \phi(u), \label{eq:u} \\
\Phi(T \circ v) = \Phi(T) \circ \phi(v), \nonumber \\
\Phi(w \circ T) = \phi(w) \circ \Phi(T). \label{eq:w}
\end{gather}
The statement \eqref{eq:w} follows from \eqref{eq:compcomp} by viewing $X_1 \times_B X_2 \to B$ as a whole and by K\"unneth. Indeed, the assumptions of \eqref{eq:compcomp} are satisfied since $Y_3$ is nonsingular. For symmetry reasons it remains to prove \eqref{eq:u}. Note that \eqref{eq:compcomp} does not apply directly as $X_2 \times_B X_3$ may be singular.

The strategy is to go through the actual proof of \eqref{eq:compcomp} in \cite[Section 3]{CH} and reduce to the technical Lemma 3.1 therein. We first remark that $u$ corresponds via adjunction and proper base change (since $g_1$ is proper) to a morphism in $D^b_c(X_1)$:
\[
u': p_*\BQ_{Y_1 \times_B X_1}[i'] \to \BQ_{X_1}[i],
\]
where $p: Y_1 \times_B X_1 \to X_1$ is the second projection. Also by adjunction $T$ corresponds to a morphism in $D^b_c(X_1)$:
\begin{equation} \label{eq:T'}
\BQ_{X_1}[i] \to f_1^!R\mathcal{H}\mathrm{om}(f_{2*}\BQ_{X_2}[j], f_{3*}\BQ_{X_3}[k]).
\end{equation}
We have natural isomorphisms
\begin{equation} \label{eq:adj}
\begin{aligned}
\phantom{\simeq{}} & f_1^!R\mathcal{H}\mathrm{om}(f_{2*}\BQ_{X_2}[j], f_{3*}\BQ_{X_3}[k]) \\
\simeq{} & R\mathcal{H}\mathrm{om}\left(f_1^*f_{2*}\BQ_{X_2}[j], f_1^!f_{3*}\BQ_{X_3}[k]\right) \\
\textrm{($f_2$ is proper)} \simeq{} & R\mathcal{H}\mathrm{om}\left(p^{12}_{1*}\BQ_{X_1 \times_B X_2}[j], p^{13}_{1*}p^{13!}_3\BQ_{X_3}[k]\right) \\
\textrm{($X_3$ is nonsingular)} \simeq{} & R\mathcal{H}\mathrm{om}\left(p^{12}_{1*}\BQ_{X_1 \times_B X_2}[j], p^{13}_{1*}\omega_{X_1 \times_B X_3}[k - 2\dim X_3]\right) \\
\textrm{($X_1$ is nonsingular)} \simeq{} & R\mathcal{H}\mathrm{om}\left(p^{12}_{1*}\BQ_{X_1 \times_B X_2}[j], p^{13}_{1*}p^{13!}_1\BQ_{X_1}[k + 2\dim X_1 - 2\dim X_3]\right) \\
\textrm{($p^{12}_{1}$ is proper)} \simeq{} & p^{12}_{1*}R\mathcal{H}\mathrm{om}\left(\BQ_{X_1 \times_B X_2}[j], p^{12!}_{1}p^{13}_{1*}p^{13!}_1\BQ_{X_1}[k + 2\dim X_1 - 2\dim X_3]\right) \\
\simeq{} & p^{12}_{1*}\left(p^{12!}_{1}p^{13}_{1*}\right)p^{13!}_1\BQ_{X_1}[k -j + 2\dim X_1 - 2\dim X_3] \\
\simeq{} & p^{12}_{1*}\left(p^{123}_{12*}p^{123!}_{13}\right)p^{13!}_1\BQ_{X_1}[k - j + 2\dim X_1 - 2\dim X_3] \\
\simeq{} & p^{123}_{1*}p^{123!}_{1}\BQ_{X_1}[k - j + 2\dim X_1 - 2\dim X_3],
\end{aligned}
\end{equation}
where the $p^-_-$ stand for the various projections from relative products of $X_i$. We write
\[
T': \BQ_{X_1}[i] \to q_*q^!\BQ_{X_1}[k - j + 2\dim X_1 - 2\dim X_3]
\]
for the morphism in $D^b_c(X_1)$ resulted from \eqref{eq:T'} and \eqref{eq:adj}, with $q = p^{123}_{1}: X_1 \times_B X_2 \times_B X_3 \to X_1$ the projection to the first factor.

Finally, we obtain \eqref{eq:u} by applying \cite[Lemma 3.1]{CH} to $u'$ and $T'$ together with an identical diagram chase as in \cite[Proof of Lemma 2.23]{CH}.
\end{proof}

In particular, the binary morphism \eqref{eq:binary} specializes to a well-defined morphism in $D^b_c(B)$:
\[
\mathfrak{p}_*(f_{1*}\BQ_{X_1}[2m]) \otimes \mathfrak{q}_*(f_{2*}\BQ_{X_2}[2n]) \to \mathfrak{r}_*(f_{3*}\BQ_{X_3}[2p]).
\]
One primary example of binary morphisms of motives is the cup-product
\[
\cup: h(X) \times h(X) \to h(X)
\]
corresponding to the small relative diagonal $[\Delta^{\mathrm{sm}}_{X/B}] \in \Corr^0_B(X, X; X)$. It specializes to the sheaf-theoretic cup-product in $D^b_c(B)$,
\[
\cup: f_*\BQ_X \otimes f_*\BQ_X \to f_*\BQ_X,
\]
and further to the cup-product
\[
\cup: H^d(X_b, \BQ) \times H^{d'}(X_b, \BQ) \to H^{d + d'}(X_b, \BQ)
\]
by restricting to the closed point $b \in B$.\footnote{In the Chow realization, the cup-product specializes to the intersection product on $\Chow^*(X)$.}

\subsection{Chern characters} \label{Sec2.3}

As we will be working throughout with coherent sheaves on singular varieties, we need a systematic definition of Chern characters taking values in Chow groups. For convenience we use the Riemann--Roch theorem of Baum--Fulton--MacPherson in \cite[Chapter 18]{Ful}.

Let $K_*(-)$ (resp.~$K^*(-)$) denote the Grothendieck group of coherent sheaves (resp.~vector bundles). By \cite[Theorem 18.3]{Ful}, for any finite type scheme $X$ there is a homomorphism
\[
\tau: K_*(X) \to \Chow_*(X)
\]
satisfying the following properties.
\begin{enumerate}

\item[(i)] For $f: X \to Y$ proper and $F \in K_*(X)$, we have $f_*\tau(F) = \tau(f_*F)$.

\item[(ii)] For $E \in K^*(X), F \in K_*(X)$, we have $\tau(E \otimes F) = \ch(E) \cap \tau(F)$.

%\item[(iii)] If $X$ is nonsingular, then for $F \in K_*(X) = K^*(X)$,
%\[
%\tau(F) = \td(T_X) \cap \ch(F) \cap [X].
%\]

\item[(iii)] If $i : X \hookrightarrow Y$ is a closed embedding with $Y$ nonsingular, then for $F \in K_*(X)$,
\begin{equation} \label{eq:locchern}
\tau(F) = \ch_X^Y(F_\bullet) \cap \td(Y) \cap [Y] = \td(i^*T_Y) \cap \ch_X^Y(F_\bullet) \cap [Y].
\end{equation}
Here $\ch_X^Y(-)$ is the localized Chern character in \cite[Section 18.1]{Ful} and $F_\bullet$ is any locally free resolution of $i_*F$ on $Y$.

\item[(iv)] If $f: X \to Y$ is an l.c.i.~morphism and both $X$ and $Y$ admit closed embeddings in nonsingular varieties\footnote{The assumption on the embeddings was recently removed by \cite{AGP}.}, then for $F \in K_*(Y)$,
\[
\tau(f^*F) = \td(T_f) \cap f^*\tau(F).
\]
Here $T_f \in K^*(X)$ is the virtual tangent bundle of $f$.

%\item[(v)] If $Y$ is a closed subvariety of $X$, then
%\[
%\tau_X(\CO_Y) = [Y] + \text{ terms of dimension $< \dim Y$}.
%\]
\end{enumerate}
Intuitively, one can think of $\tau$ as a covariant (for proper morphisms), Todd-twisted Chern character for coherent sheaves on arbitrary varieties.

\subsection{Statement of the main theorem} \label{Sec2.4}

Let $\pi: M \to B$ be a dualizable abelian fibration of relative dimension $g$ in the sense of Section \ref{Sec1.4}, and let $\pi^\vee: M^\vee \to B$ be its dual abelian fibration. We consider the normalized Poincar\'e complex $\CP \in D^b(M^\vee \times_B M)$, and define
\begin{equation} \label{eq:defFF}
\mathfrak{F} := \td(-T_{M^\vee \times_B M}) \cap \tau(\CP) \in \Chow_*(M^\vee \times_B M),
\end{equation}
where $T_{M^\vee \times_B M}$ is the virtual tangent bundle of $M^\vee \times_B M$ (since $M^\vee \times_B M$ is l.c.i.). By \eqref{eq:locchern} applied to the closed embedding $i: M^\vee \times_B M \hookrightarrow M^\vee \times M$, we also have
\begin{equation} \label{eq:FFaut}
\mathfrak{F} = \td((\pi^\vee \times_B \pi)^*T_B) \cap \ch_{M^\vee \times_B M}^{M^\vee \times M}(\CP_\bullet) \cap [M^\vee \times M],
\end{equation}
where $\pi^\vee \times_B \pi: M^\vee \times_B M \to B$ is the natural morphism. On the other hand, for the inverse of $\CP$ we set
\begin{equation} \label{eq:defFF-1}
\begin{aligned}
\mathfrak{F}^{-1} & := \td(-(\pi \times_B \pi^\vee)^*T_B) \cap \tau(\CP^{-1}) \\
& \phantom{:}= \td(T_{M \times_B M^\vee}) \cap \ch_{M \times_B M^\vee}^{M \times M^\vee}(\CP^{-1}_\bullet) \cap [M \times M^\vee] \in \Chow_*(M \times_B M^\vee).
\end{aligned}
\end{equation}
We view both $\mathfrak{F}$ and $\mathfrak{F}^{-1}$ as mixed-degree relative correspondences and call them the Chow-theoretic \emph{Fourier transform} and its inverse.

\begin{lem} \label{lem:ff-1}
We have
\[
\mathfrak{F}^{-1} \circ \mathfrak{F} = [\Delta_{M^\vee/B}] \in \Corr^0_B(M^\vee, M^\vee), \quad \mathfrak{F} \circ \mathfrak{F}^{-1} = [\Delta_{M/B}] \in \Corr_B^0(M, M).
\]
\end{lem}

\begin{proof}
We take ${\CP}^{-1} \circ {\CP} \simeq \CO_{\Delta_{M^\vee/B}}$ from \eqref{eq:dual} and apply $\tau$ on both sides. On one hand we~have
\begin{equation} \label{eq:tauO}
\tau(\CO_{\Delta_{M^\vee/B}}) = \Delta_{M^\vee/B*}\tau(\CO_{M^\vee}) = \Delta_{M^\vee/B*}\left(\td(T_{M^\vee}) \cap [M^\vee]\right) = \td(q_1^{\vee*}T_{M^\vee})\cap [\Delta_{M^\vee/B}]
\end{equation}
where $q^\vee_1: M^\vee \times_B M^\vee \to M^\vee$ is the first projection. On the other hand, we compute
\begin{equation} \label{eq:tauP}
\begin{aligned}
\tau(\CP^{-1} \circ \CP) & = p_{13*}\tau(\delta^*(\CP \boxtimes \CP^{-1})) \\
& = p_{13*}\left(\td(-p_2^*T_M)\cap \delta^!\tau(\CP \boxtimes \CP^{-1})\right) \\
& = p_{13*}\left(\td(-p_2^*T_M)\cap \delta^!(\tau(\CP) \times \tau(\CP^{-1}))\right).
\end{aligned}
\end{equation}
Here we have taken the notation from Sections \ref{Sec1.3.2} and \ref{Sec2.2.1}. The second equality uses the flatness of $\pi^\vee: M^\vee \to B$ so that the base-changed embedding
\[
M^\vee \times_B M \times_B M^\vee \hookrightarrow (M^\vee \times_B M) \times (M \times_B M^\vee)
\]
is regular (hence l.c.i.) with normal bundle identified with $p_2^*T_M$, and pulling back along this embedding coincides with the Gysin pullback $\delta^!$. The third equality uses \cite[Example~18.3.1]{Ful}.

By capping both \eqref{eq:tauO} and \eqref{eq:tauP} with $\td(-q_1^{\vee*}T_{M^\vee})$, we find
\begin{align*}
[\Delta_{M^\vee/B}] & = \td(-q_1^{\vee*}T_{M^\vee}) \cap p_{13*}\left(\td(-p_2^*T_M)\cap \delta^!(\tau(\CP) \times \tau(\CP^{-1}))\right) \\
& = p_{13*}\left(\td(-p_1^*T_{M^\vee}) \cap \td(-p_2^*T_M)\cap \delta^!(\tau(\CP) \times \tau(\CP^{-1}))\right) \\
& = p_{13*}\delta^!\left(\left(\td(-T_{M^\vee \times_B M}) \cap \tau(\CP)\right) \times \left(\td(-(\pi \times_B \pi^\vee)^*T_B) \cap \tau(\CP^{-1})\right)\right) \\
& = p_{13*}\delta^!(\FF \times \FF^{-1}) = \FF^{-1} \circ \FF,
\end{align*}
which proves the first statement of the lemma. The second statement is almost identical.
\end{proof}

\begin{rmk}
The seemingly asymmetric assignments for $\mathfrak{F}$ and $\mathfrak{F}^{-1}$ are deliberate. For one thing $\mathfrak{F}$ restricts to the Fourier transform in \cite{DM} over the open $U \subset B$ where $\pi: M_U \to U$ is an abelian scheme,
\[
\mathfrak{F}|_U = \ch(\CL) \cap [M^\vee_U \times_U M_U] = \exp(c_1(\CL)) \cap [M^\vee_U \times_U M_U],
\]
with no Todd contribution whatsoever. Another reason is the compatibility with tautological classes; see Section \ref{Sec5.2}.
\end{rmk}

We decompose both $\FF$ and $\FF^{-1}$ into
\begin{gather*}
\mathfrak{F} = \sum_{i \geq 0} \mathfrak{F}_i, \quad \mathfrak{F}_i \in \Corr_B^{i - g}(M^\vee, M), \\
\mathfrak{F}^{-1} = \sum_{i \geq 0} \mathfrak{F}^{-1}_i, \quad \mathfrak{F}^{-1}_i \in \Corr_B^{i - g}(M, M^\vee).
\end{gather*}

Now we can formulate the Fourier vanishing mentioned in Section \ref{Sec0.2} in terms of correspondences. 

\begin{defn}[Fourier vanishing]
With the notation as above, we say that the dualizable abelian fibration $\pi: M \to B$ of relative dimension $g$ satisfies the Fourier vanishing (FV), if 
\begin{equation}
\tag{FV} \FF^{-1}_i \circ \FF_j = 0 \in \Corr_B^{i + j - 2g}(M^\vee, M^\vee), \quad i+j <2g.
\end{equation}
\end{defn}

The main theorem of this section is a motivic version of Theorem \ref{thm0.1}.

\begin{thm} \label{thm:main}
Let $\pi: M \to B$ be a dualizable abelian fibration of relative dimension $g$ which satisfies (FV).
\begin{enumerate}
\item[(i)] (Decomposition) For each $k$, the pair of self-correspondences
\[
\mathfrak{p}_k := \sum_{i \leq k} \FF_i \circ \FF^{-1}_{2g - i} \in \Corr^0_B(M, M), \quad \mathfrak{q}_{k + 1}:= \sum_{i \geq k + 1} \FF_i \circ \FF^{-1}_{2g - i} \in \Corr^0_B(M, M)
\]
are orthogonal projectors, and induce a decomposition of motives
\begin{equation}\label{eq:motdec}
h(M) = P_kh(M) \oplus Q_{k + 1}h(M) \in \mathrm{CHM}(B)
\end{equation}
with $P_kh(M) = (M, \Fp_k, 0),\, Q_{k + 1}h(M) = (M, \Fq_{k + 1}, 0)$.

\item[(ii)] (Realization) For each $k$, the homological realization of $P_kh(M)$ together with the inclusion $P_kh(M) \to h(M)$ is the natural morphism
\[
{^\mathfrak{p}}\tau_{\leq k+\dim B} \pi_* \BQ_M \to \pi_*\BQ_M.
\]
Similarly, the homological realization of $h(M) \to Q_{k + 1}h(M)$ is the natural morphism
\[
\pi_*\BQ_M \to {^\mathfrak{p}}\tau_{\geq k + 1 + \dim B} \pi_* \BQ_M.
\]

\item[(iii)] (Multiplicativity) For each pair of $k, l$, the cup-product
\[
\cup: h(M) \times h(M) \to h(M)
\]
projects to zero on
\[
\cup: P_kh(M) \times P_lh(M) \to Q_{k + l + 1}h(M).
\]

\item[(iv)] (Perverse $\supset$ Chern) For each $k$, the morphism $\FF_k: h(M^\vee)(g - k) \to h(M)$ projects to zero on
\[
\FF_k: h(M^\vee)(g - k) \to Q_{k + 1}h(M).
\]
\end{enumerate}
\end{thm}

In view of the discussion in Section \ref{Sec2.2.3}, Theorem \ref{thm:main}(iii) specializes to the sheaf-theoretic multiplicativity
\[
\cup: {^\mathfrak{p}}\tau_{\leq k+\dim B} \pi_* \BQ_M \otimes {^\mathfrak{p}}\tau_{\leq l+\dim B} \pi_* \BQ_M \to {^\mathfrak{p}}\tau_{\leq k + l +\dim B} \pi_* \BQ_M,
\]
which further specializes to Theorem \ref{thm0.1}(i). Theorem \ref{thm:main}(iv) specializes immediately to Theorem \ref{thm0.1}(ii) by taking global cohomology. Assuming Theorem \ref{thm0.2} (which will be proven in Section \ref{comp_jac}) and Theorem \ref{thm1.4} (the perverse filtration (\ref{perverse}) being intrinsic), Theorem \ref{thm:main}(iii) also specializes to Theorem \ref{thm0.5} concerning the cohomology of a single curve with planar singularities.

Moreover, by tweaking the projectors $\mathfrak{p}_k$ of Theorem \ref{thm:main}(i) we obtain a motivic decomposition for $\pi: M \to B$, thus confirming Corti--Hanamura's conjecture in this case.

\begin{cor}\label{motivic_decomp}
Let $\pi: M \to B$ be a dualizable abelian fibration of relative dimension $g$ which satisfies (FV). There exists a decomposition of motives
\begin{equation*}
h(M) = \bigoplus_{i = 0}^{2g} h_i(M) \in \mathrm{CHM}(B)
\end{equation*}
whose homological realization recovers the decomposition theorem for $\pi: M\to B$.
\end{cor}

\subsection{Proofs} \label{Sec2.5}

We now prove the four statements of Theorem \ref{thm:main}. Afterwards, we note that Corollary \ref{motivic_decomp} follows immediately.

\subsubsection{Decomposition} \label{Sec2.5.1}

By Lemma \ref{lem:ff-1}, we have
\[
[\Delta_{M/B}] = \FF \circ \FF^{-1} = \sum_{i = 0}^{2g}\FF_i \circ \FF^{-1}_{2g - i} = \Fp_k + \Fq_{k + 1} \in \Corr_B^0(M, M),
\]
where the second equality holds for degree reasons. Applying the condition (FV), we find
\begin{gather*}
\Fp_k = [\Delta_{M/B}] \circ \Fp_k = \left(\sum_{i = 0}^{2g}\FF_i \circ \FF^{-1}_{2g - i} \right) \circ \left(\sum_{j \leq k}\FF_j \circ \FF^{-1}_{2g - j} \right) \\ = \left(\sum_{i \leq k}\FF_i \circ \FF^{-1}_{2g - i} \right) \circ \left(\sum_{j \leq k}\FF_j \circ \FF^{-1}_{2g - j} \right) = \Fp_k \circ \Fp_k, \\
\Fq_{k + 1} = [\Delta_{M/B}] - \Fp_k = ([\Delta_{M/B}] - \Fp_k) \circ ([\Delta_{M/B}] - \Fp_k) = \Fq_{k + 1} \circ \Fq_{k + 1}, \\
\Fp_k \circ \Fq_{k + 1} = \Fp_k \circ ([\Delta_{M/B}] - \Fp_k) = 0, \\
\Fq_{k + 1} \circ \Fp_k = ([\Delta_{M/B}] - \Fp_k) \circ \Fp_k = 0.
\end{gather*}
This shows that $\Fp_k$ and $\Fq_{k + 1}$ are orthogonal projectors adding up to $[\Delta_{M/B}]$, which proves the motivic decomposition \eqref{eq:motdec}.

\subsubsection{Realization} \label{Sec2.5.2}

The realization statement is a consequence of the support condition (b) of Section \ref{Sec1.4}. We first remark that by restricting to the open subset $U \subset B$ where $\pi_U: M_U \to U$ is an abelian scheme, the self-correspondences
\[
\overline{\Fp}_k|_U := \FF_k|_U \circ \FF^{-1}_{2g - k}|_U \in \Corr^0_U(M_U, M_U)
\]
are projectors giving rise to the motivic decomposition in \cite{DM}:
\begin{equation} \label{eq:motdecsm}
h(M_U) = \bigoplus_{i = 0}^{2g}h_i(M_U) \in \mathrm{CHM}(U), \quad h_i(M_U) = (M_U, \overline{\Fp}_i|_U, 0).
\end{equation}
The above formula for the projectors is written down explicitly in \cite[Section 2]{Kun}. Moreover, it was shown in \cite[Remarks after Corollary 3.2]{DM} that the homological realization of \eqref{eq:motdecsm} yields a canonical decomposition into shifted local systems
\[
\pi_{U*}\BQ_{M_U} = \bigoplus_{i = 0}^{2g} R^i\pi_{U*}\BQ_{M_U}[-i].
\]
In particular, for the realization of $P_kh(M)|_U = (M_U, \Fp_k|_U, 0)$ we have
\[
\Fp_k|_{U*}(\pi_{U*}\BQ_{M_U}) = \bigoplus_{i = 0}^{k} R^i\pi_{U*}\BQ_{M_U}[-i]
\]
with $\Fp_k|_{U*}(\pi_{U*}\BQ_{M_U}) \to \pi_{U*}\BQ_{M_U}$ given by the natural inclusion.

Consequently, for the realization of the inclusion $P_kh(M) \to h(M)$, \emph{i.e.},
\begin{equation} \label{eq:realinc}
\Fp_{k*}(\pi_*\BQ_M) \to \pi_*\BQ_M,
\end{equation}
we know that the restriction to $U$ of the perverse cohomology
\begin{equation} \label{eq:i+dimB}
{^\Fp}\CH^{i + \dim B}(\Fp_{k*}(\pi_*\BQ_M))
\end{equation}
is zero for $i > k$, and the restriction to $U$ of the induced morphism
\begin{equation} \label{eq:i+dimB2}
{^\Fp}\CH^{i + \dim B}(-): {^\Fp}\CH^{i + \dim B}(\Fp_{k*}(\pi_*\BQ_M)) \to {^\Fp}\CH^{i + \dim B}(\pi_*\BQ_M)
\end{equation}
is an isomorphism for $i \leq k$. The support condition (b) then implies that \eqref{eq:i+dimB} and~\eqref{eq:i+dimB2} must stay so over the entire base $B$; otherwise the perverse sheaf ${^\Fp}\CH^{i + \dim B}(\pi_*\BQ_M)$ would admit a nontrivial subobject (either ${^\Fp}\CH^{i + \dim B}(\Fp_{k*}(\pi_*\BQ_M))$ or ${^\Fp}\CH^{i + \dim B}(\Fq_{{k + 1}*}(\pi_*\BQ_M))$) supported away from $U$, contradicting with the full support of ${^\Fp}\CH^{i + \dim B}(\pi_*\BQ_M)$. Therefore we have
\begin{equation}\label{eqn34}
\Fp_{k*}(\pi_*\BQ_M) \simeq {^\mathfrak{p}}\tau_{\leq k+\dim B} \pi_* \BQ_M
\end{equation}
with \eqref{eq:realinc} given by the natural morphism ${^\mathfrak{p}}\tau_{\leq k+\dim B} \pi_* \BQ_M \to \pi_* \BQ_M$. The realization statement for $Q_{k + 1}h(M)$ follows from its $P_kh(M)$ counterpart together with the motivic decomposition~\eqref{eq:motdec}.

\subsubsection{Multiplicativity} \label{Sec2.5.3}

The motivic multiplicativity is proven via the convolution product. Recall the convolution kernel $\CK \in D^b\Coh(M^\vee \times_B M^\vee \times_B M^\vee)$ from the condition~(a2) of Section \ref{Sec1.4}. We define the Chow-theoretic \emph{convolution} (or the \emph{Pontryagin product}) to be
\begin{equation} \label{eq:chowconv}
\FC := \td(-q^{\vee*}_{12}T_{M^\vee \times_B M^\vee}) \cap \tau(\CK) \in \Chow_*(M^\vee \times_B M^\vee \times_B M^\vee)
\end{equation}
where $q^\vee_{12}: M^\vee \times_B M^\vee \times_B M^\vee \to M^\vee \times_B M^\vee$ is the projection to the first two factors and~$T_{M^\vee \times_B M^\vee}$ is the virtual tangent bundle of $M^\vee \times_B M^\vee$ (since $M^\vee \times_B M^\vee$ is l.c.i.). After Section \ref{Sec2.2.3}, we view $\FC$ as a mixed-degree binary relative correspondence.

\begin{lem} \label{lem:conv}
We have
\[
\FF \circ \FC \circ (\FF^{-1} \times \FF^{-1}) = [\Delta^{\mathrm{sm}}_{M/B}] \in \Corr_B^0(M, M; M).
\]
\end{lem}

\begin{proof}
The proof is similar to that of Lemma \ref{lem:ff-1}. We take the equation ${\CP} \circ \CK \simeq \CO_{\Delta^\mathrm{sm}_{M/B}}\circ ({\CP} \boxtimes {\CP} )$ from \eqref{convolution} and rewrite it as
\[
{\CP} \circ \CK \circ ({\CP}^{-1} \boxtimes {\CP}^{-1} ) \simeq \CO_{\Delta^\mathrm{sm}_{M/B}} \in D^b\Coh(M \times_B M \times_B M).
\]
Then we apply $\tau$ on both sides. On one hand we have
\begin{equation} \label{eq:tauOsm}
\tau(\CO_{\Delta^{\mathrm{sm}}_{M/B}}) = \Delta^{\mathrm{sm}}_{M/B*}\tau(\CO_{M}) = \Delta^{\mathrm{sm}}_{M/B*}\left(\td(T_{M}) \cap [M]\right) = \td(q_3^*T_{M})\cap [\Delta^{\mathrm{sm}}_{M/B}]
\end{equation}
where $q_3: M \times_B M \times_B M \to M$ is the third projection.
On the other hand, we first compute~$\tau(\CK \circ \CP^{-1})$ where the composition is on the first factor of $M^\vee \times_B M^\vee \times_B M^\vee$. We~have
\begin{align*}
\tau(\CK \circ \CP^{-1}) & = p_{134*}\tau(\delta^*(\CP^{-1} \boxtimes \CK)) \\
& = p_{134*}\left(\td(-p_2^*T_{M^\vee}) \cap \delta^!\tau(\CP^{-1} \boxtimes \CK)\right) \\
& = p_{134*}\left(\td(-p_2^*T_{M^\vee}) \cap \delta^!(\tau(\CP^{-1}) \times \tau(\CK))\right) \\
& = p_{134*}\delta^!\left(\left(\td(-(\pi \times_B \pi^\vee)^*T_B) \cap \tau(\CP^{-1})\right) \times \left(\td(-q_1^{\vee*}T_{\pi^\vee}) \cap \tau(\CK)\right) \right) \\
& = (\td(-q_1^{\vee*}T_{\pi^\vee}) \cap \tau(\CK)) \circ \FF^{-1},
\end{align*}
where $q_i^\vee: M^\vee \times_B M^\vee \times_B M^\vee \to M^\vee$ is the $i$-th projection. Similarly, we have
\[
\tau(\CK \circ (\CP^{-1} \boxtimes \CP^{-1})) = (\td(-q_1^{\vee*}T_{\pi^\vee} -q_2^{\vee*}T_{\pi^\vee}) \cap \tau(\CK)) \circ (\FF^{-1} \times \FF^{-1}).
\]
Finally, comparing with \eqref{eq:tauOsm} we find
\begin{align*}
[\Delta^{\mathrm{sm}}_{M/B}] & = \td(-q_3^*T_M) \cap \tau(\CO_{\Delta^\mathrm{sm}_{M/B}}) \\
& = \td(-q_3^*T_M) \cap \tau({\CP} \circ \CK \circ ({\CP}^{-1} \boxtimes {\CP}^{-1} )) \\
& = \td(-q_3^*T_M) \cap p_{124*} \tau(\delta^*((\CK \circ({\CP}^{-1} \boxtimes {\CP}^{-1} )) \boxtimes \CP)) \\
& = \td(-q_3^*T_M) \cap p_{124*}\left(\td(-p_3^*T_{M^\vee})\cap \delta^!\left(\tau(\CK \circ (\CP^{-1} \boxtimes \CP^{-1})) \times \tau(\CP)\right)\right) \\
& = \td(-q_3^*T_M) \cap p_{124*}\left(\td(-p_3^*T_{M^\vee}) \phantom{\FF^{-1}}\right.\\
& \quad\quad\quad \cap\left.\delta^!\left(\left((\td(-q_1^{\vee*}T_{\pi^\vee} - q_2^{\vee*}T_{\pi^\vee}) \cap \tau(\CK)) \circ (\FF^{-1} \times \FF^{-1})\right) \times \tau(\CP)\right)\right) \\
& = p_{124*} \delta^!\left(\left((\td(-q_{12}^{\vee*}T_{M^\vee \times_B M^\vee}) \cap \tau(\CK)) \circ (\FF^{-1} \times \FF^{-1})\right) \times \left(\td(-T_{M^\vee \times_B M}) \cap \tau(\CP)\right)\right) \\
& = p_{124*} \delta^!\left((\FC \circ (\FF^{-1} \times \FF^{-1})) \times \FF\right) = \FF \circ \FC \circ (\FF^{-1} \times \FF^{-1}),
\end{align*}
which proves the lemma.
\end{proof}

%\begin{rmk} \label{altconv}
%Alternatively, we can rewrite \eqref{convolution} as $\CK \simeq \CP^{-1} \circ \CO_{\Delta^\mathrm{sm}_{M/B}} \circ (\CP \boxtimes \CP)$ and prove instead the identity
%\[
%\FC = \FF^{-1} \circ [\Delta^\mathrm{sm}_{M/B}] \circ (\FF \times \FF) \in \mathrm{CH}_*(M^\vee \times_B M^\vee \times_B M^\vee).
%\]
%\end{rmk}

The key input from the condition (a2) of Section \ref{Sec1.4} is a degree estimate
\begin{equation} \label{eq:convdim}
\FC \in \Corr^{\geq -g}_B(M^\vee, M^\vee; M^\vee)
\end{equation}
since the convolution kernel $\CK \in D^b\Coh(M^\vee \times_B M^\vee \times_B M^\vee)$ is of codimension $g$. Now the restriction of the cup-product $\cup: h(M) \times h(M) \to h(M)$ to
\[
\cup: P_kh(M) \times P_lh(M) \to Q_{k + l + 1}h(M)
\]
is given by the composition
\begin{equation} \label{eq:convcomp}
\Fq_{k + l + 1} \circ [\Delta^{\mathrm{sm}}_{M/B}] \circ (\Fp_k \times \Fp_l) \in \Corr^0_B(M, M; M).
\end{equation}
By Lemma \ref{lem:conv} we have
\begin{equation} \label{eq:cupconv}
\Fq_{k + l + 1} \circ [\Delta^{\mathrm{sm}}_{M/B}] \circ (\Fp_k \times \Fp_l) = \Fq_{k + l + 1} \circ \FF \circ \FC \circ \left((\FF^{-1} \circ \Fp_k) \times (\FF^{-1} \circ \Fp_l)\right).
\end{equation}
Expanding the right-hand side of \eqref{eq:cupconv} and applying (FV), we find
\begin{equation} \label{eq:estimate}
\begin{aligned}
& \left(\sum_{i_3 \geq k + l + 1}\FF_{i_3}\circ\FF^{-1}_{2g - i_3}\right) \circ \left(\sum_{j_3}\FF_{j_3}\right) \circ \FC \\
& \quad \circ \left(\left(\left(\sum_{j_1}\FF^{-1}_{j_1}\right) \circ \left(\sum_{i_1 \leq k}\FF_{i_1}\circ\FF^{-1}_{2g - i_1}\right)\right) \times \left(\left(\sum_{j_2}\FF^{-1}_{j_2}\right) \circ \left(\sum_{i_2 \leq l}\FF_{i_2}\circ\FF^{-1}_{2g - i_2}\right)\right)\right) \\
{}={} & \left(\sum_{i_3 \geq k + l + 1}\FF_{i_3}\circ\FF^{-1}_{2g - i_3}\right) \circ \left(\sum_{j_3 \geq k + l + 1}\FF_{j_3}\right) \circ \FC \\
& \quad \circ \left(\left(\left(\sum_{j_1 \geq 2g - k}\FF^{-1}_{j_1}\right) \circ \left(\sum_{i_1 \leq k}\FF_{i_1}\circ\FF^{-1}_{2g - i_1}\right)\right) \times \left(\left(\sum_{j_2 \geq 2g - l}\FF^{-1}_{j_2}\right) \circ \left(\sum_{i_2 \leq l}\FF_{i_2}\circ\FF^{-1}_{2g - i_2}\right)\right)\right) \\
{}={} & \Fq_{k + l + 1} \circ \left(\sum_{j_3 \geq k + l + 1}\FF_{j_3}\right) \circ \FC \circ \left(\left(\left(\sum_{j_1 \geq 2g - k}\FF^{-1}_{j_1}\right) \circ \Fp_k\right) \times \left(\left(\sum_{j_2 \geq 2g - l}\FF^{-1}_{j_2}\right) \circ \Fp_l\right)\right).
\end{aligned}
\end{equation}
We observe that
\begin{gather*}
\sum_{j_1 \geq 2g - k}\FF^{-1}_{j_1} \in \Corr_B^{\geq g - k}(M, M^\vee), \quad \sum_{j_2 \geq 2g - l}\FF^{-1}_{j_2} \in \Corr_B^{\geq g - l}(M, M^\vee), \\
\sum_{j_3 \geq k + l + 1}\FF_{j_3} \in \Corr_B^{\geq k + l + 1 - g}(M^\vee, M).
\end{gather*}
This, together with the degree estimate \eqref{eq:convdim}, yields for \eqref{eq:estimate} a total degree of
\[
\geq -g  + (g - k) + (g - l) + (k + l + 1 - g)  = 1.
\]
In other words, we find $\Fq_{k + l + 1} \circ [\Delta^{\mathrm{sm}}_{M/B}] \circ (\Fp_k \times \Fp_l)$ to lie strictly in $\Corr^{\geq 1}_B(M, M; M)$ which, compared with \eqref{eq:convcomp}, implies the vanishing
\[
\Fq_{k + l + 1} \circ [\Delta^{\mathrm{sm}}_{M/B}] \circ (\Fp_k \times \Fp_l) = 0.
\]
This proves the multiplicativity statement.

\subsubsection{Perverse $\supset$ Chern} \label{Sec2.5.4}
The final statement amounts to showing that
\[
\Fq_{k + 1} \circ \FF_k = 0 \in \Corr_B^{k - g}(M^\vee, M),
\]
which follows immediately from (FV) since
\[
\Fq_{k + 1} \circ \FF_k = \left(\sum_{i \geq k + 1}\FF_i \circ \FF^{-1}_{2g - i}\right) \circ \FF_k = \sum_{i \geq k + 1}\FF_i \circ \left( \FF^{-1}_{2g - i} \circ \FF_k\right) = 0.
\]
The proof of Theorem \ref{thm:main} is now complete. \qed

\subsubsection{Motivic decomposition} \label{Sec2.5.5}
Finally, we note that our construction yields a motivic decomposition of $\pi: M \to B$ which proves Corollary \ref{motivic_decomp}.

One subtlety of the sequence of motives $P_kh(M)$ is that we do not know if $P_kh(M)$ is a summand of $P_lh(M)$ for $k < l$. While the condition (FV) immediately implies
\begin{equation} \label{eq:submot}
\Fp_l \circ \Fp_k = \Fp_k, \quad k < l,
\end{equation}
it is unclear if $\Fp_k \circ \Fp_l = \Fp_k$. However, we can fix this issue by introducing a minor modification
\[
\widetilde{\Fp}_k := \Fp_k \circ \cdots \circ \Fp_{2g} \in \Corr^0_B(M, M), \quad \overline{\Fp}_k := \widetilde{\Fp}_k - \widetilde{\Fp}_{k - 1} \in \Corr^0_B(M, M).
\]
Now it is easy to verify that both $\widetilde{\Fp}_k$ and $\overline{\Fp}_k$ are projectors, and
\[
\widetilde{\Fp}_k \circ \widetilde{\Fp}_l = \widetilde{\Fp}_l \circ \widetilde{\Fp}_k = \widetilde{\Fp}_k, \quad k < l.
\]
For instance, we have by \eqref{eq:submot}
\[
\widetilde{\Fp}_k \circ \widetilde{\Fp}_l = \Fp_k \circ \cdots \circ \Fp_{2g} \circ \Fp_l \circ \cdots \circ \Fp_{2g} = \Fp_k \circ \cdots \circ \Fp_{2g} = \widetilde{\Fp}_k, \quad k \leq l.
\]
The resulting motives $\widetilde{P}_kh(M) = (M, \widetilde{\Fp}_k, 0)$ and $h_k(M) = (M, \overline{\Fp}_k, 0)$ satisfy
\[
\widetilde{P}_kh(M) = \widetilde{P}_{k - 1}h(M) \oplus h_k(M) \in \mathrm{CHM}(B).
\]
By induction, we also find that the $\overline{\Fp}_k$ are mutually orthogonal projectors giving rise to a motivic decomposition
\begin{equation} \label{eq:motbbd}
h(M) = \bigoplus_{i = 0}^{2g} h_i(M) \in \mathrm{CHM}(B), \quad h_i(M) = (M, \overline{\Fp}_k, 0).
\end{equation}
The support argument of Section \ref{Sec2.5.2} then shows that the homological realization of \eqref{eq:motbbd} provides a decomposition
\[
\pi_*\BQ_M = \bigoplus_{i = 0}^{2g}{^\mathfrak{p}}\CH^{i + \dim B}(\pi_* \BQ_M)[-i - \dim B] \in D^b_c(B).
\]
This finishes the proof of Corollary \ref{motivic_decomp}. \qed

\section{Compactified Jacobians}\label{comp_jac}

\subsection{Overview} \label{Sec3.1}
In this section, we treat the geometry of (degree $0$) compactified Jacobian fibrations; our main purpose is to prove Theorem \ref{thm0.2}.

Throughout this section, we let $C \to B$ be a flat family of integral projective curves of arithmetic genus $g$ with planar singularities over an irreducible base $B$. We further assume that $C \to B$ admits a section with image contained in its smooth locus. We denote by $\pi: \overline{J}_C \to B$ the compactified Jacobian fibration, and assume that the total space~$\overline{J}_C$ is nonsingular. In particular, $B$ is nonsingular and~$\pi: \overline{J}_C \to B$ is flat with integral fibers of dimension $g$.

At the end of this section, we discuss compactified Jacobian fibrations of arbitrary degree still in the presence of a section to $C \to B$, in which case changes to the arguments are minimal. The assumption on the existence of a section will be removed in Section \ref{Sec:twist} using a stack version of our theory.

\subsection{Compactified Jacobians}\label{Sec3.2}

Let $J_C \subset \overline{J}_C$ be the Jacobian fibration parameterizing line bundles. It is an open dense subset of $\overline{J}_C$, which forms a nonsingular commutative group scheme over $B$.

The Fourier--Mukai theory of compactified Jacobian fibrations was developed by Arinkin~\cite{A1, A2}. More precisely, Arinkin constructed in \cite{A1} the normalized Poincar\'e line bundle $\CP$ over the product $J_C \times_B \overline{J}_C$ extending the standard normalized Poincar\'e line bundle for Jacobians associated with nonsingular curves. Then in \cite{A2} Arinkin showed that the sheaf $j_*\CP$, where $j: J_C \times_B \overline{J}_C \hookrightarrow \overline{J}_C \times_B \overline{J}_C$ is the natural open embedding, is a Cohen--Macaulay coherent sheaf (which we still denote by $\CP$ for notational convenience) on $\overline{J}_C \times_B \overline{J}_C$; he further proved that $\CP$ induces a derived auto-equivalence
\[
\mathrm{FM}_{\CP}: D^b\mathrm{Coh}(\overline{J}_C) \xrightarrow{~~\simeq~~} D^b\mathrm{Coh}(\overline{J}_C).
\]
The inverse of this transform is induced by
\[
\CP^{-1} : = \CP^\vee \otimes \pi_2^* \omega_{\pi}[g], \quad \CP^\vee: = {\CH}om(\CP, \CO_{\overline{J}_C\times_B \overline{J}_C}).
\]
Here $\omega_\pi$ is the relative canonical bundle with respect to $\pi: \overline{J}_C \to B$, and $\pi_2: \overline{J}_C \times_B \overline{J}_C \to B$ is the natural morphism.

Furthermore, the support theorem of Ng\^o \cite{Ngo} and the Severi inequality \cite[Lemma 4.1]{MS} ensure that the decomposition theorem associated with $\pi: \overline{J}_C \to B$ has full support. Therefore, in order to prove Theorem \ref{thm0.2}, it remains to verify:
\begin{enumerate}
    \item[(i)] the convolution kernel
    \begin{equation}\label{K_kernel}
    \CK \in D^b\mathrm{Coh}(\overline{J}_C \times_B \overline{J}_C \times_B \overline{J}_C)
    \end{equation}
    is of codimension $g$;
    \item[(ii)] the condition (FV) holds for $\pi: \overline{J}_C \to B$.
\end{enumerate}
We will discuss (i), which is a result of Arinkin, in Section \ref{AV}, and will discuss (ii) in Section~\ref{Sec3.5}.

We conclude this section by recalling the \emph{canonical} perverse filtration constructed in \cite{MY} for the compactified Jacobian~$\overline{J}_{C_0}$ associated with \emph{any} integral projective curve $C_0$ with planar singularities; this is the filtration used in Theorem \ref{thm0.5}. For such a $C_0$, we include it in a family of curves $C \to B$ over an irreducible $B$ such that the associated compactified Jacobian fibration $\pi: \overline{J}_C \to B$ is nonsingular and $\overline{J}_{C_0} \subset \overline{J}_C$ is the closed fiber over $0 \in B$. Then the construction of Section \ref{1.3.3} yields a perverse filtration 
\[
P_0H^*(\overline{J}_{C_0}, \BQ) \subset P_1H^*(\overline{J}_{C_0}, \BQ) \subset  \cdots \subset H^*(\overline{J}_{C_0}, \BQ)
\]
via the morphism $\pi: \overline{J}_C \to B$.

\begin{thm}[{Maulik--Yun \cite[Theorem 1.1]{MY}}]\label{thm1.4}
The perverse filtration defined above is independent of the choice of the family $C \to B$ of curves.
\end{thm}

\begin{proof}
In \cite[Theorem 1.1]{MY}, the family $C \to B$ is chosen with certain assumptions. However, as pointed out in \cite[Proposition 2.15]{MY}, the smoothness of $\overline{J}_C$ suffices.
\end{proof}

In particular, Theorem \ref{thm0.5} follows from Theorem \ref{thm:main}(iii) and Theorem \ref{thm0.2}, where we include~$C_0$ in a family $\pi: C \to B$ as a closed fiber.

\subsection{Arinkin's dimension bound and the convolution kernel} \label{AV}

We first give an expression of the convolution kernel (\ref{K_kernel}). Denote by $\overline{J}_C^{n+1}$ the $(n+1)$-th relative product. It carries natural projections
\[
\pi_{n + 1}: \overline{J}_C^{n+1} \to B, \quad p_i: \overline{J}_C^{n + 1} \to \overline{J}_C, \quad p_{i,j}: \overline{J}_C^{n + 1} \to \overline{J}^2_C, \quad \cdots .
\]
We consider the object
\[
\CK_n: = p_{1,\dots,n, *}\left( p_{1,n+1}^* \CP \otimes p_{2,n+1}^* \CP \otimes \cdots \otimes p_{n,n+1}^* \CP \right) \otimes \pi_{n}^* \omega_\pi[g]   \in D^b\mathrm{Coh} ( \overline{J}_C^n ).
\]
Here the boundedness of $\CK_n$ is a consequence of the flatness of $\CP$ with respect to both factors \cite[Theorem A and Lemma 6.1]{A2}. Let $\nu: \overline{J}_C \to \overline{J}_C$ be the involution sending a sheaf to its dual, so that we can write
\[
\CP^\vee = (\mathrm{id}_{\overline{J}_C} \times_B \nu)^* \CP.
\]
Then by the defining equation
\[
\CK \simeq \CP^{-1} \circ \CO_{\Delta_{\overline{J}_C/B}^{\mathrm{sm}}} \circ (\CP \boxtimes \CP)
\]
obtained from \eqref{convolution}, and again by the flatness of $\CP$ with respect to both factors, we may express the convolution kernel (\ref{K_kernel}) as
\[
\CK = \left(\mathrm{id}_{\overline{J}_{C}\times_B \overline{J}_C}  \times_B \nu \right)^* \CK_3.
\]

Therefore, the condition (i) of Section \ref{Sec3.2} is a special case of the following result of Arinkin~\cite[Section 7.1]{A2}.

\begin{prop}[Arinkin]\label{Arinkin_vanishing}
    For any $n$, we have
    \[
    \mathrm{codim}_{\overline{J}^n_C} (\mathrm{supp}( \CK_n ) ) \geq g.
    \]

\end{prop}

Note that $\CK_n$ is a classical object for low values of $n$. When $n=1$ the proposition was a step towards establishing \cite[Theorem 1.1]{A1} concerning the Fourier--Mukai transform of the structure sheaf. When $n=2$, the proposition was used in \cite{A2} to show that $\mathrm{FM}_\CP$ is a derived equivalence. The same argument works for the general case, which we review as follows.

\begin{proof}[Proof of Proposition \ref{Arinkin_vanishing}]
We first fix some notation. For an integral projective curve $C_b$, we denote by $g_a(C_b)$ its arithmetic genus, and $\widetilde{C}_b$ its normalization. We denote by $\delta(C_b) \in \BZ_{\geq 0}$ the dimension of the maximal affine group of $\mathrm{Pic}(C_b)$. Clearly, when $C_b$ has planar singularities, we have
\[
\delta(C_b) = g_a(C_b) - g_a(\widetilde{C}_b). 
\]
This yields a constructible function 
\[
\delta: B \to \BZ_{\geq 0}, \quad b \mapsto \delta(C_b)
\]
associated with $C \to B$. For an irreducible subvariety $Z \subset B$, we define $\delta(Z)$ to be the value of $\delta(b)$ at a general point $b\in Z$.

\medskip
\noindent {\bf Step 1.} Recall that the smoothness of $\overline{J}_C$ yields the Severi inequality \cite[Lemma 4.1]{MS}
\[
\mathrm{codim}_B (Z) \geq \delta_Z
\]
for any irreducible $Z\subset B$. Therefore, it suffices to show that for any curve $C_b$ in the family $C \to B$, we have
\[
\mathrm{codim}_{\overline{J}_{C_b}}\left(\mathrm{supp}( \CK_n )\cap \overline{J}_{C_b}\right) \geq g_a(\widetilde{C}_b).
\]

\medskip
\noindent {\bf Step 2.} Now we focus on a fixed curve $C_b$. For a point $F \in \overline{J}_{C_b}$, we denote by $\CP_{F}$ the restriction of $\CP$ to $\overline{J}_{C_b}=\overline{J}_{C_b} \times \{F\}$. If we further restrict this sheaf to the open subset $J_{C_b} \subset \overline{J}_{C_b}$ parameterizing line bundles on $C_b$, it is a line bundle.

If $(F_1, F_2,\cdots, F_n) \in \overline{J}^n_{C_b}$ lies in the support of $\CK_n$, then there is a nontrivial cohomology
\begin{equation}\label{non-van}
H^*(\overline{J}_{C_b}, \CP_{F_1} \otimes \CP_{F_2} \otimes  \cdots \otimes \CP_{F_n} ) \neq 0
\end{equation}
By the same argument as in \cite[Proposition 7.2]{A2}, (\ref{non-van}) ensures that 
\begin{equation}\label{condi}
\bigotimes_{i=1}^n (\CP_{F_i}|_{J_{C_b}} ) \simeq \CO_{J_{C_b}}.
\end{equation}

\medskip
\noindent {\bf Step 3.} Finally, we claim that (\ref{condi}) expresses a codimension $g_a(\widetilde{C}_b)$ condition for 
\[
(F_1, F_2,\cdots, F_n) \in \overline{J}^n_{C_b}.
\]
This was explained in \cite[Corollary 7.3 and Proposition 7.4]{A2}. Indeed, pulling back the isomorphism (\ref{condi}) via the Abel--Jacobi map $C_b \hookrightarrow \overline{J}_{C_b}$, we obtain that 
\begin{equation}\label{condi2}
\bigotimes_{i=1}^n ({F_i}|_{C^{\mathrm{reg}}_b} ) \simeq \CO_{C^{\mathrm{reg}}_b}.
\end{equation}
Then the same argument as in \cite[Proposition 7.4]{A2} applied to the action morphism
\[
\mu \times \mathrm{id}_{\overline{J}_{C_b}}: J_{C_b} \times \overline{J}^n_{C_b} \to  \overline{J}^n_{C_b}
\]
shows that (\ref{condi2}) is a codimension $g_a(\widetilde{C}_b)$ condition. This proves Proposition~\ref{Arinkin_vanishing}.
\end{proof}

We have thus verified (i) of Section \ref{Sec3.2}.

\subsection{Remarks on the Fourier vanishing for abelian schemes}

We first discuss two proofs of (FV) for abelian schemes.

As in Section \ref{Sec1.2}, we let $\pi: A \to B$ be an abelian scheme of relative dimension $g$ with $\pi^\vee: A^\vee \to B$ its dual. Then by Lemma \ref{lem:ff-1}, for degree reasons we have for any $k \neq 2g$,
\begin{equation}\label{eqn41}
\sum_{i + j = k}\FF^{-1}_i \circ \FF_j = 0 \in \Corr_B^{k - 2g}(A^\vee, A^\vee).
\end{equation}
The idea of \cite{B,DM} is that, if we apply the pullback associated with the multiplication by $N$ map on the first factor 
\[
[N]: A^\vee \times_B A^\vee \to A^\vee \times_B A^\vee,
\]
we may scale each individual term of (\ref{eqn41}) by a different constant $N^j$:
\begin{equation}\label{eqn42}
\sum_{i + j = k}N^j\FF^{-1}_i \circ \FF_j = 0.
\end{equation}
Since this holds for every $N$, we obtain the desired vanishing 
\[
\FF^{-1}_i \circ \FF_j = 0, \quad i+j \neq 2g.
\]

However, this argument does not work for abelian fibrations with singular fibers, since it relies on the facts that there is a multiplication by $N$ map on the total space $A^\vee$ and the normalized Poincar\'e complex is a line bundle. Both fail when there are singular fibers.

We now present another proof of the weaker version (FV) whose idea can be generalized. For convenience, we focus on the Jacobian fibration $\pi: J_C \to B$ associated with a family of nonsingular projective curves $C \to B$ of genus $g$. Recall that we have the normalized Poincar\'e line bundle $\CL$. The discussion in Section \ref{AV} actually yields the dimension bound
\[
\mathrm{codim}_{J^2_C} \left(\mathrm{supp}( \CL^{-1} \circ \CL^{\otimes N}   ) \right) \geq g
\]
for any $N \in \BZ$. Therefore, if we take the Chern character of $\CL^{-1} \circ \CL^{\otimes N}$, we obtain the equation~(\ref{eqn42}) only for $i+j < 2g$. This yields (FV) for $\pi: J_C \to B$. 

The second proof above relies on the following two ingredients:
\begin{enumerate}
    \item[(i)] the Arinkin-type dimension bound of Section \ref{AV}, and
    \item[(ii)] the $K$-theoretic operations $(-)^{\otimes N}$ which help us to scale the Chern character of different degrees distinctly.
\end{enumerate}
In the next section, we will generalize them to treat $\pi: \overline{J}_C \to B$ where the second ingredient is replaced by the Adams operations from $K$-theory.

\subsection{Adams operations and Fourier vanishing} \label{Sec3.5}

We now verify the condition (FV) for $\pi: \overline{J}_C \to B$, which proves (ii) of Section \ref{Sec3.2}. Consequently, this completes the proofs of Theorems \ref{thm0.2}, \ref{CoHa}, and \ref{thm0.5}.

\subsubsection{Dimension bound}
We consider the closed embedding
\begin{equation} \label{eq:regemb}
i: \overline{J}_C \times_B \overline{J}_C  \hookrightarrow  \overline{J}_C \times \overline{J}_C.
\end{equation}
For any $N \in \BZ_{>0}$, the object 
\[
(i_*\CP)^{\otimes N} \in D^b\mathrm{Coh}(\overline{J}_C \times \overline{J}_C)
\]
is bounded, and is exact off $\overline{J}_C \times_B \overline{J}_C \subset \overline{J}_C \times \overline{J}_C$. %Therefore its localized Chern characters produce classes on $\overline{J}_C \times_B \overline{J}_C$.
%\[
%i: \overline{J}_C \times_B \overline{J}_C \times_B \overline{J}_C \hookrightarrow \overline{J}_C \times_B \overline{J}_C \times \overline{J}_C
%\]
Then we consider 
\[
\widetilde{\CK}(N): = \CP^{-1} \circ (i_*\CP)^{\otimes N} \in D^b\mathrm{Coh}(\overline{J}_C \times \overline{J}_C).
\]
By definition (and by the flatness of $\CP$ with respect to both factors), it is constructed by pulling back the two objects $(i_*\CP)^{\otimes N}, \CP^{-1}$ to the triple product 
\[
\overline{J}_C \times \overline{J}_C \times_B \overline{J}_C,
\]
taking their tensor product, and pushing forward to the first and the third factors.

\begin{prop}
For any $N$, the object $\widetilde{\CK}(N) \in D^b\mathrm{Coh}(\overline{J}_C \times \overline{J}_C)$ is supported on a codimension $g$ subset of $\overline{J}_C \times_B \overline{J}_C$.
\end{prop}

\begin{proof}
The proof relies on Arinkin's dimension estimate which goes through the following steps. It is essentially parallel to the proof of Proposition \ref{Arinkin_vanishing}.

\medskip
\noindent {\bf Step 1.} We first show that the set-theoretic support of $\widetilde{\CK}(N)$ is contained in $\overline{J}_C \times_B \overline{J}_C$. By definition and the projection formula, the object $\widetilde{\CK}(N)$ is obtained via the pushforward of an object supported on
\[
\overline{J}_C \times_B \overline{J}_C \times_B  \overline{J}_C \hookrightarrow \overline{J}_C \times \overline{J}_C \times_B \overline{J}_C
\]
Therefore, it can be written as the pushforward through the following chain of maps
\[
\overline{J}_C \times_B \overline{J}_C \times_B  \overline{J}_C \rightarrow  
 \overline{J}_C \times_B \overline{J}_C \hookrightarrow 
\overline{J}_C \times \overline{J}_C
\]
where the first map is the projection to the first and the third factors, and the second map is the natural inclusion $i$. This proves the claim.

\medskip
\noindent {\bf Step 2.} Now any closed point in the support of $\widetilde{K}(N)$ can be represented by a pair of sheaves on $C_b \subset C$:
\begin{equation}\label{pair1}
(F_1, F_2) \in \mathrm{supp}( \widetilde{K}(N) ) \cap \overline{J}_{C_b}.
\end{equation}
As in Step 2 of the proof of Proposition \ref{Arinkin_vanishing}, it suffices to show that
\begin{equation}\label{Step2_prop3.3}
\mathrm{codim}_{\overline{J}_{C_b}}\left(\mathrm{supp}( \widetilde{\CK}(N) )\cap \overline{J}_{C_b}\right) \geq g_a(\widetilde{C}_b)
\end{equation}
where $\widetilde{C}_b$ is the normalization of $C_b$.

\medskip
\noindent {\bf Step 3.} For a pair $(F_1, F_2)$ as in (\ref{pair1}), by base change we obtain a nontrivial cohomology 
\[
H^*\left( \overline{J}_C, (\iota_*\CP_{F_1} )^{\otimes N} \otimes \iota_*\CP^\vee_{F_2} \right) \neq 0.
\]
Here $\iota: \overline{J}_{C_b} \hookrightarrow \overline{J}_C$ is the closed embedding of a closed fiber and the tensor product takes place on the ambient variety $\overline{J}_C$. Equivalently, we may write this nontrivial cohomology on the fiber~$\overline{J}_{C_b}$ as
\begin{equation}\label{coh1}
H^*\left( \overline{J}_{C_b}, ( \iota^*\iota_*\CP_{F_1} )^{\otimes N} \otimes \CP^\vee_{F_2} \right) \neq 0.
\end{equation}
Since $\overline{J}_C$ is nonsingular, the complex $\iota^*\iota_*\CP_{F_1}$ is \emph{perfect} on the possibly singular fiber $\overline{J}_{C_b}$. The following claim describes all the cohomology sheaves of this complex. 
\medskip

\noindent {\bf Claim.} Each cohomology sheaf $\CH^k(\iota^*\iota_*\CP_{F_1}) \in \mathrm{Coh}(\overline{J}_{C_b})$ is a direct sum of finite copies of~$\CP_{F_1}$ on $\overline{J}_{C_b}$.

\begin{proof}[Proof of the claim] We first note the decomposition
\[
\iota^*\iota_* \CO_{\overline{J}_{C_b}} \simeq \bigoplus_k \CH^k(\iota^*\iota_*\CO_{\overline{J}_{C_b}})[-k]
\]
where each term in the right-hand side is a free $\CO_{\overline{J}_{C_b}}$-module. This can be seen from base change and the corresponding statement for the embedding $\{b\} \hookrightarrow B$.

Next, we relate the two objects\footnote{The two objects (\ref{eq47}) may not be isomorphic in $D^b \mathrm{Coh}(\overline{J}_{C_b})$, since the second may not be perfect while the first is always perfect.}
\begin{equation}\label{eq47}
\iota^*\iota_*\CP_{F_1}, \quad  \iota^*\iota_* \CO_{\overline{J}_{C_b}} \otimes \CP_{F_1} \in D^b\mathrm{Coh}(\overline{J}_{C_b}).
\end{equation}
By the projection formula, we see that they are isomorphic after pushing forward to $\overline{J}_C$,
\[
\iota_*(\iota^*\iota_*\CP_{F_1} ) \simeq \iota_* \left( \iota^*\iota_* \CO_{\overline{J}_{C_b}} \otimes \CP_{F_1}\right).
\]
Therefore, the cohomology sheaves of the two objects (\ref{eq47}), which are $\CO_{\overline{J}_{C_b}}$-modules themselves, are isomorphic as $\CO_{\overline{J}_C}$-modules. This forces them to be isomorphic as $\CO_{\overline{J}_{C_b}}$-modules:
\[
\CH^k(\iota^*\iota_*\CP_{F_1} ) \simeq  \CH^k\left(\iota^*\iota_* \CO_{\overline{J}_{C_b}} \otimes \CP_{F_1}\right)  \in \mathrm{Coh}(\overline{J}_{C_b}).
\]
The claim then follows from the decomposition of $\iota^*\iota_* \CO_{\overline{J}_{C_b}}$. 
\end{proof}
%\appendix

%\begin{rmk}
%    We note that the two objects (\ref{eq47}) may not be isomorphic in $D^b \mathrm{Coh}(\overline{J}_C)$, since the second may not be perfect while the first is always perfect. 
%\end{rmk}

\noindent {\bf Step 4.} By the claim of Step 3, we know that the object $\iota^*\iota_*\CP_{F_1}$ admits an increasing filtration induced by the standard truncation functors, whose graded pieces are direct sums of finite copies of $\CP_{F_1}$. Hence the cohomology (\ref{coh1}) admits a filtration whose graded pieces are direct sums of finite copies of
\begin{equation}\label{non-van1}
H^*\left(\overline{J}_{C_b}, \CP_{F_1}^{\otimes N} \otimes \CP^{-1}_{F_2} \right).
\end{equation}
In particular, we obtain from (\ref{coh1}) that there exists a nontrivial cohomology of the type (\ref{non-van1}). 

By the same argument as in \cite[Section 7.1]{A2} or Steps 2 and 3 in the proof of Proposition~\ref{Arinkin_vanishing}, the non-vanishing cohomology (\ref{non-van1}) expresses a codimension $g_a(\widetilde{C}_b)$ condition for the pair $(F_1, F_2) \in \overline{J}^2_{C_b}$: the constraint is given by
\[
(\CP_{F_1}|_{J_{C_b}} )^{\otimes N} \otimes (\CP_{F_2}|_{J_{C_b}} )^\vee \simeq \CO_{J_{C_b}},
\]
which further yields
\[
(F_1\big{|}_{C^{\mathrm{reg}}_b})^{\otimes N} \otimes (F_2\big{|}_{C^{\mathrm{reg}}_b})^\vee \simeq \CO_{C^{\mathrm{reg}}_b}.
\]
This is a codimension $g_a(\widetilde{C}_b)$ condition, and the proof of (\ref{Step2_prop3.3}) is complete.
\end{proof}

For our purpose we also consider variants of $\widetilde{\CK}(N)$, replacing the (derived) tensor product partially by the (derived) exterior product:
\[
\widetilde{\CK}^\lambda(N_1, N_2, \dots, N_k) : = 
 \CP^{-1} \circ \bigotimes_{i=1}^{k}\left(\wedge^{N_i}(i_*\CP) \right)\in D^b\mathrm{Coh}(\overline{J}_C \times \overline{J}_C).
\]
Since $\wedge^{N_i}(i_*\CP)$ is an isotypic component of the natural $\mathfrak{S}_n$-action on $(i_*\CP)^{\otimes N_i}$, the object above is a direct summand of $\widetilde{\CK}(\sum_i N_i)$. This gives the following.

\begin{cor} \label{cor:wedge}
For any $k$-tuple $(N_1, N_2, \dots, N_k)$, the object $\widetilde{\CK}^\lambda(N_1, N_2, \dots, N_k)$ is supported on a codimension $g$ subset of $\overline{J}_C \times_B \overline{J}_C$.
\end{cor}

\subsubsection{Adams operations}
With the Arinkin-type dimension bound in hand, we proceed to $K$-theoretic operations following Gillet--Soul\'e \cite{GS}.

Let $i: X \hookrightarrow Y$ be a closed embedding of finite type schemes. We denote by $K_{X}(Y)$ the Grothendieck group of bounded complexes of locally free sheaves on $Y$ which are exact off $X$. In \cite[Section 4]{GS}, a sequence of \emph{Adams operations}
\[
\psi^N: K_{X}(Y) \to K_{X}(Y), \quad N \geq 1
\]
were constructed using the $\lambda$-ring structure on $K_{X}(Y)$ given by the (derived) exterior product, together with the induction formula
\begin{equation} \label{eq:adamsdef}
\psi^N - \psi^{N - 1} \otimes \lambda^1 + \cdots +(-1)^{N - 1}\psi^1 \otimes \lambda^{N - 1} + (-1)^NN\lambda^N = 0.
\end{equation}

The compatibility of the Adams operations with the localized Chern character $\ch_X^Y(-)$ in \cite[Section~18.1]{Ful} is described in \emph{e.g.}~\cite[Theorem 3.1]{KR}. Namely for $F_\bullet \in K_X(Y)$, we~have
\begin{equation} \label{eq:scale}
\ch_{X, k}^Y(\psi^N(F_\bullet)) = N^k \ch_{X, k}^Y(F_\bullet) \in \Chow^k(X \to Y).
\end{equation}
Here $\Chow^k(X \to Y)$ is the degree $k$ bivariant Chow group of $i: X \hookrightarrow Y$. From \eqref{eq:scale} we see how localized Chern characters in different degrees are scaled differently by the Adams operations.

We now return to the closed embedding \eqref{eq:regemb} and fix it for the rest of this section.\footnote{Alternatively, one could perhaps use the Adams operations $\psi_N$ in \cite{AGP, HM} which are canonical and independent of the closed embedding.} Since $\overline{J}_C \times \overline{J}_C$ is nonsingular, capping with $\CO_{\overline{J}_C \times \overline{J}_C} \in K_*(\overline{J}_C \times \overline{J}_C)$ induces a canonical isomorphism
\begin{equation} \label{eq:kisom}
\cap\CO_{\overline{J}_C \times \overline{J}_C}: K_{\overline{J}_C \times_B \overline{J}_C}(\overline{J}_C \times \overline{J}_C) \xrightarrow{~~\simeq~~} K_*(\overline{J}_C \times_B \overline{J}_C)
\end{equation}
whose inverse is the map that sends $F \in K_*(\overline{J}_C \times_B \overline{J}_C)$ to any locally free resolution $F_\bullet$ of~$i_*F$ on $\overline{J}_C \times \overline{J}_C$; see \cite[Lemma 1.9]{GS}. Similarly, by \cite[Propositions 17.3.1 and 17.4.2]{Ful} there are canonical isomorphisms
\begin{equation} \label{eq:chowisom}
\cap[\overline{J}_C \times \overline{J}_C]: \Chow^k(\overline{J}_C \times_B \overline{J}_C \to \overline{J}_C \times \overline{J}_C) \xrightarrow{~~\simeq~~} \Chow_{2\dim\overline{J}_C - k}(\overline{J}_C \times_B \overline{J}_C).
\end{equation}

Under \eqref{eq:kisom} and \eqref{eq:chowisom} the Adams operations (with respect to the closed embedding \eqref{eq:regemb}) take the form
\[
\psi^N: K_*(\overline{J}_C \times_B \overline{J}_C) \to K_*(\overline{J}_C \times_B \overline{J}_C), \quad N \geq 1
\]
and the localized Chern character $\ch_{\overline{J}_C \times_B \overline{J}_C}^{\overline{J}_C \times \overline{J}_C}(-)$ yields a morphism (with simplified notation)
\[
\widetilde{\ch}: K_*(\overline{J}_C \times_B \overline{J}_C) \to \Chow_*(\overline{J}_C \times_B \overline{J}_C),
\]
such that for $F \in K_*(\overline{J}_C \times_B \overline{J}_C)$ we have
\begin{equation} \label{eq:compch}
\widetilde{\ch}_k(\psi^N(F)) = N^k\widetilde{\ch}_k(F) \in \Chow_{2\dim\overline{J}_C - k}(\overline{J}_C \times_B \overline{J}_C).
\end{equation}

Moreover, Corollary \ref{cor:wedge} together with the expression \eqref{eq:adamsdef} immediately implies the following.

\begin{cor} \label{cor3.5}
For any $N$, the class $\CP^{-1} \circ \psi^N(\CP) \in K_*(\overline{J}_C \times_B \overline{J}_C)$ is supported on a codimension $g$ subset of $\overline{J}_C \times_B \overline{J}_C$.
\end{cor}

\subsubsection{Proof of (FV) for $\pi: \overline{J}_C \to B$} \label{Sec3.5.3}
We apply $\tau(-)$ to the $K$-theory class
\[
\CP^{-1} \circ \psi^N(\CP) \in K_*(\overline{J}_C \times_B \overline{J}_C).
\]
Following the calculations in \eqref{eq:tauP}, we find
\begin{equation} \label{eq:taucomp}
\tau(\CP^{-1} \circ \psi^N(\CP)) = p_{13*}\left(\td(-p_2^*T_{\overline{J}_C})\cap \delta^!(\tau(\psi^N(\CP)) \times \tau(\CP^{-1}))\right).
\end{equation}
Here we have taken the notation from Sections \ref{Sec1.3.2} and \ref{Sec2.2.1}. Then we cap both sides of \eqref{eq:taucomp} with $\td(-q_1^*T_{\overline{J}_C} - \pi_2^*T_B)$, where $q_i: \overline{J}_C \times_B \overline{J}_C \to \overline{J}_C$ is the $i$-th projection and $\pi_2: \overline{J}_C \times_B \overline{J}_C \to B$ is the natural map. We have by \eqref{eq:locchern}
\begin{equation} \label{eq:f-1ch}
\begin{aligned}
& \td(-q_1^*T_{\overline{J}_C} - \pi_2^*T_B) \cap \tau(\CP^{-1} \circ \psi^N(\CP)) \\
={} & \td(-q_1^*T_{\overline{J}_C} - \pi_2^*T_B) \cap p_{13*}\left(\td(-p_2^*T_{\overline{J}_C})\cap \delta^!(\tau(\psi^N(\CP)) \times \tau(\CP^{-1}))\right)\\
={} & p_{13*}\delta^!\left( \left(\td(-q_1^*T_{\overline{J}_C} -q_2^*T_{\overline{J}_C}) \cap \tau(\psi^N(\CP))\right) \times (\td(- \pi_2^*T_B) \cap \tau(\CP^{-1})) \right) \\
={} & \FF^{-1} \circ \widetilde{\ch}(\psi^N(\CP)).
\end{aligned}
\end{equation}

The dimension bound of Corollary \ref{cor3.5} now implies that both ends of \eqref{eq:f-1ch} vanish in codimensions $< g$. In other words, we have for $k < 2g$,
\[
\sum_{i + j = k}\FF^{-1}_i \circ \widetilde{\ch}_{j + \dim B}(\psi^N(\CP)) = \sum_{i + j = k}N^{j + \dim B}\FF^{-1}_i \circ \widetilde{\ch}_{j + \dim B}(\CP) = 0 \in \Corr_B^{k - 2g}(\overline{J}_C, \overline{J}_C).
\]
Note that the first equality uses the scaling \eqref{eq:compch}. Since the second equality holds for every $N$, we obtain the vanishing of each individual term
\begin{equation} \label{eq:FVch}
\FF^{-1}_i \circ \widetilde{\ch}_{j + \dim B}(\CP) = 0, \quad i + j < 2g.
\end{equation}

Finally, we notice that by \eqref{eq:FFaut} we have
\[
\FF = \td(\pi_2^*T_B) \cap \widetilde{\ch}(\CP).
\]
Since the Todd class $\td(\pi_2^*T_B)$ is pulled back from the base $B$, we conclude from \eqref{eq:FVch} that
\begin{align*}
\FF^{-1}_i \circ \FF_j & = \FF^{-1}_i \circ \left(\sum_{j' + j'' = j} \td_{j'}(\pi_2^*T_B) \cap \widetilde{\ch}_{j'' + \dim B}(\CP)\right) \\
& = \sum_{j' + j'' = j} \td_{j'}(\pi_2^*T_B) \cap (\FF^{-1}_i \circ  \widetilde{\ch}_{j'' + \dim B}(\CP)) = 0
\end{align*}
for $i + j < 2g$. This completes the proof of (FV) for $\pi: \overline{J}_C \to B$. \qed

\subsection{Variants}\label{Sec3.6}

Recall from \cite{A1, A2} that Arinkin's normalized Poincar\'e sheaf $\CP$ on $\overline{J}_C \times_B \overline{J}_C$ is constructed from the normalized universal family $\CF$ on $C \times_B \overline{J}_C$ of degree $0$ rank $1$ torsion-free sheaves. Here the normalization of $\CF$ is characterized by the trivializations along the sections 
\[
B \subset C, \quad  0_{J} \subset  \overline{J}_C. 
\]
In this section, we write
\[
\CP  = \CP(\CF, \CF)
\]
to indicate the dependence of the normalized Poincar\'e sheaf on the universal sheaf $\CF$, where the two factors correspond to the two factors in Arinkin's formula \cite[(1.1)]{A2} respectively; see also the formula after (3.2) in \cite{ADM}.

We have shown that this normalized Poincar\'e sheaf $\CP$ yields a motivic lifting of the perverse filtration associated with $\pi: \overline{J}_C \to B$. Now we discuss variants of the construction above, which will be used in Section \ref{Sec:twist}.

Using the section of $C \to B$, we can identify the degree $d$ compactified Jacobian $\overline{J}^d_C$ with the (degree $0$) compactified Jacobian $\overline{J}_C$. A universal family of rank $1$ degree $d$ torsion-free sheaves on $C \times_B \overline{J}^d_C = C \times_B \overline{J}_C$ is given by
\[
\CF \otimes p_C^*O_C(dB) \otimes p_J^*\CL,
\]
where $p_C, p_J$ are the natural projections from $C \times_B \overline{J}_C$, and $\CL$ is a line bundle on $\overline{J}_C$. Now for two integers~$d,e$, we consider the twisted Poincar\'e sheaf
\begin{equation*}
\CP':=\CP(\CF \otimes p_C^*O_C(eB) \otimes p_J^*\CL_e, \,\CF \otimes p_C^*O_C(dB) \otimes p_J^*\CL_d) \in \mathrm{Coh}( \overline{J}_C \times_B \overline{J}_C ),
\end{equation*}
where we impose the condition that $\CL_d, \CL_e$ are line bundles on $\overline{J}_C$ whose restrictions to a nonsingular fiber ${J}_{C_b}$ lie in $\mathrm{Pic}^0({J}_{C_b})$. By the proof of \cite[Proposition 3.1(i)]{ADM}, we have
\[
\CP' \simeq  \CP \otimes \left( \CL_e^{\otimes d} \boxtimes \CL_d^{\otimes e} \right), \quad \CP'^{-1} \simeq  \CP^{-1} \otimes \left( \CL_d^{\vee \otimes e} \boxtimes \CL_e^{\vee \otimes d} \right).
\]
Parallel to the construction of the Fourier theory using $\CP$, we can also use $\CP'$ to define the Fourier transforms $\mathfrak{F}', \mathfrak{F}'^{-1}$, the projectors $\mathfrak{p}'_k, \mathfrak{q}'_{k + 1}$, and the motives $P'_kh(\overline{J}_C)$ with orthogonal complements $Q'_{k+1}h(\overline{J}_C)$. These operators and submotives \emph{a priori} depend on the choices of~$d, e, \CL_d, \CL_e$.

The following proposition shows that, changing the degree, or twisting the normalized Poincar\'e sheaf by fiberwise homologically trivial line bundles does not influence the realization.

\begin{prop}\label{prop3.6}
For $\CP'$ as above, the homological realization of $P'_kh(\overline{J}_C)$ together with the inclusion $P'_kh(\overline{J}_C) \to h(\overline{J}_C)$ is also given by the natural morphism
\[
{^\mathfrak{p}}\tau_{\leq k+\dim B} \pi_* \BQ_{\overline{J}_C} \to \pi_*\BQ_{\overline{J}_C}.
\]
\end{prop}

\begin{proof}
By the support argument of Section \ref{Sec2.5.2}, it suffices to work with the Jacobian fibration $\pi_U: J_U \to U$ over an (arbitrary) open subset $U \subset B$ parameterizing nonsingular curves, and show that the composition
\begin{equation}\label{eqn58}
\Fp'_k|_{U*}(\pi_{U*}\BQ_{J_U}) \rightarrow \pi_{U*}\BQ_{J_U} \rightarrow \Fp_k|_{U*}(\pi_{U*}\BQ_{J_U})
\end{equation}
is an isomorphism.

Note that both sides of (\ref{eqn58}) are direct sums of (shifted) local systems. In order to show that~(\ref{eqn58}) is an isomorphism, it suffices to check that it induces an isomorphism on the stalks over each $b \in U$. Then the desired statement follows clearly: since $\CL_d, \CL_e$ lie in $\mathrm{Pic}^0(J_{C_b})$ for~$b\in U$, the restrictions of $\Fp'_k$ and $\Fp_k$ over $b\in U$ are cohomologically identical.
\end{proof}

\section{Twisted compactified Jacobians} \label{Sec:twist}

\subsection{Overview and setup}
The purpose of the this section is to extend the Fourier theory for compactified Jacobians established in Section \ref{comp_jac} to a family of integral locally planar curves without a section. This situation occurs in the study of Hitchin systems and the Le Potier moduli spaces of $1$-dimensional sheaves on $\BP^2$. The main results are twisted versions of Theorem \ref{thm:main} and Corollary \ref{motivic_decomp}; see Corollaries \ref{cor:main} and \ref{cor:motdectwist}. We will discuss applications in Section~\ref{applications}.

Throughout this section, we let $C \to B$ be a flat family of integral projective curves of arithmetic genus $g$ with planar singularities over an irreducible base $B$. We assume that the total space $C$ is nonsingular (hence $B$ is nonsingular), and there is a multisection 
\[
D \subset C \to B
\]
of degree $r$ which is finite and flat over $B$.
%such that the singular locus of $C \to B$ intersects with $D$ in codimension 2 loci of $D$. 
%For any integer $d$, the degree $d$ compactified Jacobian exists as a scheme \'etale locally over the base~$B$ since there is always a section of $C \to B$ \'etale locally. Globally, the degree $d$ compactified Jacobian $\overline{J}^d_C$ exists only as an algebraic space.
For convenience, we also assume that for any integer~$d$, the degree $d$ compactified Jacobian~$\overline{J}^d_C$ is a quasi-projective variety. This is satisfied when $C \to B$ is given by a linear system of a nonsingular surface; in this case, the compactified Jacobian arises as the moduli space of certain semistable torsion sheaves on the surface, so quasi-projectivity follows from the general construction of moduli spaces of semistable sheaves. We further assume that $\overline{J}^d_C$ is nonsingular.

We denote the natural projection map by 
\[
\pi_d: \overline{J}^d_C \to B
\]
to indicate its dependence on $d$.

\subsection{Universal families and gerbes}\label{Sec4.2}

For arbitrary degree $d$, two issues arise if we want to establish a Fourier theory for $\overline{J}^d_C$ governing the perverse filtration.
\begin{enumerate}
    \item[(i)] There may not exist a universal sheaf on $C \times_B \overline{J}^d_C$.
    \item[(ii)] Even if a universal sheaf exists, it is not unique. Different choices of a universal sheaf influence the Poincar\'e sheaf and the Fourier transform. In particular, for an arbitrary choice of a universal sheaf, the induced motivic filtration may not recover the perverse filtration via homological realization.\footnote{This is already the case for the Jacobian fibration associated with a family of nonsingular curves.}
\end{enumerate}
When $C\to B$ admits a section, a standard solution to (ii) is that we can force the universal family to be trivialized along the section, which eliminates the ambiguity given by a line bundle pulled back from the moduli space $\overline{J}^d_C$. Since now we only have a multisection, we first introduce the notion of \emph{trivialization along a multisection}.

For any $B$-scheme $T$, we consider a flat family of rank $1$ torsion-free sheaves of degree $d$ on the curves parameterized by $T$:
\[
\CF^d_T \rightsquigarrow C\times_B T,
\]
and we define
\begin{equation}\label{line_bundle}
\CR^d_T: = \mathrm{det}\left(p_{T*} (\CF^d_T|_{D\times_B T} )\right) \in \mathrm{Pic}(T).
\end{equation}
Note that $\CF^d_T|_{D\times_B T}$ is Tor-finite over $T$ so that the pushforward to $T$ is a perfect complex and we can take its determinant; see \cite[Tag 08IS]{Stacks}. % \cite[Lemma 76.52.10]{Stacks}. 
We say that this family over $T$ is trivialized along the multisection $D \subset C$, if there is a specified isomorphism
\[
\CR^d_T \simeq \CO_T \in \mathrm{Pic}(T).
\]
However, we note that unlike the case with a section, in general we can not find a universal sheaf on $C \times_B \overline{J}^d_C$ which is trivialized along the multisection $D$, even if a universal sheaf $\CF^d$ exists. This is because globally on $\overline{J}^d_C$ there may not exist a line bundle which is an $r$-th root of the line bundle $\CR^d$.
Nevertheless, we can always solve both issues (i, ii) above by passing to a $\mu_r$-gerbe over $\overline{J}^d_C$ as follows. 

The idea is to modify the moduli functor defining the compactified Jacobian. Let $\overline{\CJ}^d_C$ be the functor sending any $B$-scheme $T$ to a groupoid given by the data
\[
\CF^d_T \rightsquigarrow C\times_B T
\]
satisfying the same conditions as for the stack of the degree $d$ compactified Jacobian, with an extra assumption that $\CF^d_T$ is trivialized along the multisection $D$. 

\begin{prop}
The functor $\overline{\CJ}^d_C$ is represented by a Deligne--Mumford stack which is a $\mu_r$-gerbe over $\overline{J}^d_C$.
\end{prop}

\begin{proof}
Let $\overline{\mathfrak{J}}^d_C$ denote the stack of degree $d$ compactified Jacobian; it is a $\BG_m$-gerbe over $\overline{J}^d_C$. The line bundles (\ref{line_bundle}) define a morphism
\[
\overline{\mathfrak{J}}^d_C \to B\BG_m.
\]
Imposing the trivialization condition along $D$ is equivalent to taking the fiber product of this morphism with the structure map $\mathrm{pt} \to B\BG_m$:
\[
\overline{\CJ}^d_C= \overline{\mathfrak{J}}^d_C \times_ {B\BG_m} \mathrm{pt}.
\]
This is representable by a Deligne--Mumford stack. Finally it is a $\mu_r$-gerbe over $\overline{J}^d_C$ since the perfect complex $p_{T*} (\CF^d_T|_{D\times_B T} )$ has rank $r$ (which can be seen directly from nonsingular fibers).
\end{proof}

%even if there is a universal family $\CF^d$ be a universal family on $C \times_B \overline{J}^d_C$; it is a family of rank $1$ torsion-free sheaves on the fibers $C_b$ of degree $d$.  If we restrict this sheaf to the multisection and pushforward to $\overline{J}^d_C$, the result
%\begin{equation}\label{eqn59}
%p_{J*}(\CF^d|_{D\times_B \overline{J}_C^d} )
%\end{equation}
%is a perfect complex on $\overline{J}^d_C$ of rank $e$ and we can consider its determinant

%line bundle away from codimension 2 loci pulled back from $D$. Therefore, the pushforward of (\ref{eqn59}) to $\overline{J}_C^d$ is a vector bundle of rank $e$, and we can consider its determinant
%\[
%\CR(\CF^d): = \mathrm{det}\left(p_{J*} (\CF^d|_{D\times_B \overline{J}_C^d} )\right) \in \mathrm{Pic}(\overline{J}^d_C).
%\]
%In general we cannot find a universal sheaf $\CF^d$ so that the line bundle $\CR(\CF^d)$ is trivial, since there may not exist an $e$-th root of the line bundle $\CR(\CF^d)$ on $\overline{J}^d_C$. Nevertheless, we can always achieve this by passing to a $\mu_e$-gerbe 
%\[
%\overline{\CJ}^d_{C}  \to \overline{J}^d_C
%\]
%constructed as the $e$-th root stack associated with the line bundle $\CR(\CF^d)$. 

Consequently, we have constructed for any $d$ a nonsingular Deligne--Mumford stack~$\overline{\CJ}^d_C$ realized as a $\mu_r$-gerbe over $\overline{J}^d_C$, together with the universal family $\CF^d$ of rank $1$ degree~$d$ torsion-free sheaves on $C \times_B \overline{\CJ}^d_C$, trivialized along the multisection $D$:
\[
\mathrm{det}\left(p_{J*} (\CF^d|_{D\times_B \overline{\CJ}_C^d} )\right) \simeq \CO_{\overline{\CJ}^d_C} \in \mathrm{Pic}( \overline{\CJ}^d_C ).
\]

\begin{rmk}
Here we work with $\mu_r$-gerbes, instead of $\BG_m$-gerbes, for the following two reasons. First, Deligne--Mumford stacks share the same (rational) Chow groups and cohomology with the underlying course moduli spaces. Second, the morphisms from the $\mu_r$-gerbes to the base~$B$ are still proper, so that the formalisms of \cite{BBD, CH} still apply. 
\end{rmk}

To set up the Fourier--Mukai transforms, we first define the isotypic category $\mathrm{Coh}(\overline{\CJ}^d_C)_{(k)}$ (resp.~$D^b\mathrm{Coh}(\overline{\CJ}^d_C)_{(k)}$) to be the full subcategory of $\mathrm{Coh}(\overline{\CJ}^d_C)$ (resp.~$D^b\mathrm{Coh}(\overline{\CJ}^d_C)$) consisting of objects for which the action of $\mu_r$ on fibers is given by the character $\lambda \mapsto \lambda^k$ of $\mu_r$. Then we consider two integers $d,e$. As in Section~\ref{Sec3.6}, we plug $\CF^d, \CF^e$ in Arinkin's formula \cite[(1.1)]{A2} and obtain a Poincar\'e sheaf\footnote{More precisely, the universal sheaves $\CF^e, \CF^d$ define a twisted line bundle 
\[
\mathrm{det}\left(p_{13*}(p^*_{12}\CF^e \otimes p^*_{23}\CF^d)\right)\otimes \mathrm{det}\left(p_{13*}\CO_{\overline{\CJ}^e_C \times_B C\times_B\overline{\CJ}^d_C} \right)\otimes \mathrm{det}\left(p_{13*}p^*_{12}\CF^e \right)^\vee \otimes \mathrm{det}\left(p_{13*}p^*_{23}\CF^d \right)^\vee
\]
on ${\CJ}^e_C \times_B \overline{\CJ}^d_C \cup \overline{\CJ}^e_C \times_B {\CJ}^d_C$, where $p_{ij}$ are the natural projections from $\overline{\CJ}^e_C \times_BC\times_B\overline{\CJ}^d_C$ to the corresponding factors. Its Cohen--Maucalay extension is defined to be $\CP_{e,d}$.
}
\[
\CP_{e,d} : = \CP(\CF^e, \CF^d) \in \mathrm{Coh}( \overline{\CJ}^{e}_C \times_B \overline{\CJ}^d_C )_{(d, e)}.
\]
We define 
\[
\CP_{e,d}^{-1}: = \CH\mathrm{om}_{\overline{\CJ}^{e}_C \times_B \overline{\CJ}^d_C} ( \CP_{e,d}, \, p_2^* \omega_{\pi_{d}}) [g]
\]
where $\omega_{\pi_{d}}$ is the relative canonical bundle with respect to $\pi_{d}: \overline{\CJ}^{d}_C \to B$, and we view $\CP_{e,d}^{-1}$ as in~$D^b\Coh(\overline{\CJ}^{d}_C \times_B \overline{\CJ}^e_C)_{(-e, -d)}$. Because $\CP_{e,d}$ lies in the isotypic category, the Fourier--Mukai transform
\[
\mathrm{FM}_{\CP_{e,d}}: D^b\mathrm{Coh}(\overline{\CJ}^e_C) \rightarrow  D^b\mathrm{Coh}(\overline{\CJ}^{d}_C)
\]
 is only nonzero on the following isotypic components:
\[
\mathrm{FM}_{\CP_{e,d}}: D^b\mathrm{Coh}(\overline{\CJ}^e_C)_{(-d)} \rightarrow D^b\mathrm{Coh}(\overline{\CJ}^{d}_C)_{(e)}.
\]

The next proposition is parallel to \cite[Proposition 3.1(i)]{ADM} and \cite[Theorem 4.7]{GShen}.

\begin{prop}\label{prop4.1}
The objects $\CP_{e,d}, \CP^{-1}_{e,d}$ are inverse to each other as Fourier--Mukai kernels:
\[
\mathrm{FM}_{\CP_{e,d}}: D^b\mathrm{Coh}(\overline{\CJ}^e_C)_{(-d)} \xrightarrow{~~\simeq~~} D^b\mathrm{Coh}(\overline{\CJ}^{d}_C)_{(e)},\quad \mathrm{FM}_{\CP^{-1}_{e,d}}=\mathrm{FM}^{-1}_{\CP_{e,d}}.
\]
\end{prop}

\begin{proof}
Recall that for $\CP_{e,d}$ and $\CP^{-1}_{e,d}$ we always have the canonical adjunction morphisms
\[
\mathrm{id}_{D^b\mathrm{Coh}(\overline{\CJ}^e_C)_{(-d)}} \to\mathrm{FM}_{\CP^{-1}_{e,d}} \circ \mathrm{FM}_{\CP_{e,d}}, \quad 
\mathrm{FM}_{\CP_{e,d}}\circ \mathrm{FM}_{\CP^{-1}_{e,d}} \to \mathrm{id}_{D^b\mathrm{Coh}(\overline{\CJ}^{d}_C)_{(e)}}.
\]
It suffices to show that their cones in the derived categories are $0$; this is a property which can be checked \'etale locally over the base $B$. We postpone the \'etale local calculations to Corollary~\ref{cor4.3} below.
\end{proof}

\subsection{\'Etale local structures}\label{Sec4.3}

In this section, we carry out the \'etale local calculation for the Fourier--Mukai transforms given by $\CP_{e,d}, \CP^{-1}_{e,d}$. This leads to two consequences. First, the calculation completes the proof of Proposition \ref{prop4.1}. Second, we will show as in Proposition~\ref{prop3.6} that the motivic decomposition induced by the Fourier transforms $\mathfrak{F}, \mathfrak{F}^{-1}$ associated with~$\CP_{e,d}, \CP^{-1}_{e,d}$ specializes to the (standard) perverse filtration; see Section \ref{Sec4.4}.

Now we work \'etale locally over the base $B$. Let $U$ be an \'etale neighborhood of the base. We may assume that $C \to U$ admits simultaneously a section and a multisection
\[
U \subset C \to U, \quad D \subset C \to U.
\]
They are independent and do not have any nontrivial relation. Due to the existence of the section, the relative compactified Jacobian $\overline{J}^d_C$ are identified for any choices of $d$. As before, we denote by $\overline{J}_C$ the degree $0$ compactified Jacobian, and we want to compare $\CF^{d}$ on $C \times_U \overline{\CJ}^d_C$ (which is trivialized along the multisection $D$) with the normalized universal sheaf $\CF$ on $C \times_U \overline{J}_C$ (which is trivialized along the section $U$).

\begin{prop}\label{prop4.2}
There is a $U$-morphism
\[
\sigma_d: \overline{\CJ}^d_C \to \overline{J}_C
\]
satisfying the following properties.
\begin{enumerate}
    \item[(i)] We have 
    \[
    \CF^d \simeq (\mathrm{id}_C \times_U \sigma_d)^* \CF \otimes p_C^* \CO_{C}(dU) \otimes p_\CJ^* \CL_d, \quad \CL_d \in \mathrm{Pic}(\overline{\CJ}^d_C),
    \]
    where $p_C, p_\CJ$ are the natural projections from $C \times_U \overline{\CJ}^d_C$.
   \item[(ii)] For any $b\in U$ with $C_b$ a nonsingular curve, the restriction of $\CL_d$ to the fiber $\CJ^d_{C_b}$ has trivial first Chern class in $H^2({\CJ}^d_{C_b}, \BQ)$.
\end{enumerate}
\end{prop}

\begin{proof}
The desired morphism $\sigma_d$ can be constructed using the map from the Deligne--Mumford stack to its coarse moduli space, in view of the identification $\overline{J}^d_C \simeq \overline{J}_C$. For our purpose, we describe it in a more direct way as follows.

Note that $\overline{J}_C$ is a fine moduli space. Therefore, describing a morphism from $\overline{\CJ}^d_C$ is equivalent to constructing a family of rank $1$ torsion-free sheaves of degree $0$ trivialized along the section $U \subset C$, for which we may use $\CF^d$. More precisely, we define
\[
\CG^d: = \CF^d \otimes p_C^* \CO_{C}(-dU) \otimes p_\CJ^*\left( (\CF^d \otimes p_C^* \CO_{C}(-dU))|_{U \times_U \overline{\CJ}^d_C} \right)^\vee.
\]
Here the second factor adjusts the degree from $d$ to $0$, and the third factor is included to trivialize the universal sheaf along the section. This yields a morphism $\sigma_d: \overline{\CJ}^d_C \to \overline{J}_C$ satisfying
\[
(\mathrm{id}_C \times_U \sigma_d)^* \CF \simeq \CG^d.
\]
This already proves (i).

To prove (ii), it suffices to show that, for a fixed nonsingular curve $C_b$ with $s \in C_b$ the intersection with the section $U \subset C$, the line bundle $\CF^d|_{s \times \CJ_{C_b}}$ has homologically trivial first Chern class. Assume $D \cap C_b$ is given by $s_1, \dots , s_e \in C_b$ (counted with multiplicities). Then the condition that $\CF^d$ is trivialized along the multisection $D$ implies 
\[
\sum_{i=1}^e c_1\left(\CF^d|_{s_i \times \CJ_{C_b}}\right)  = 0 \in H^2({\CJ}^d_{C_b}, \BQ).
\]
This shows 
\[
c_1\left(\CF^d|_{s \times \CJ_{C_b}}\right)  = 0 \in H^2(\CJ^d_{C_b}, \BQ) 
\]
as desired, since the first Chern class of the left-hand side does not depend on the point in $C_b$. The proof of (ii) is complete.
\end{proof}

Now we consider 
\[
\sigma_e\times_U \sigma_{d}: \overline{\CJ}^e_C \times_U \overline{\CJ}^{d}_C \to \overline{J}_C \times_U \overline{J}_C.
\]
The normalized Poincar\'e sheaf $\CP$ and its inverse $\CP^{-1}$ are canonically defined on the target. The following is an immediate consequence of Proposition \ref{prop4.2} as in Section \ref{Sec3.6} (using the argument of \cite[Proposition 3.1(i)]{ADM}). It provides an \'etale local description of the Poincar\'e sheaf $\CP_{e,d}$ and its inverse $\CP^{-1}_{e,d}$.

\begin{cor}\label{cor4.3}
We have
\[
\CP_{e,d} \simeq (\sigma_e\times_U \sigma_{d})^*\CP \otimes \left(\CL_e^{\otimes d} \boxtimes \CL_{d}^{\otimes e} \right), \quad \CP^{-1}_{e,d} \simeq (\sigma_d\times_U \sigma_{e})^*\CP^{-1} \otimes \left(\CL_d^{\vee \otimes e} \boxtimes \CL_e^{\vee \otimes d} \right).
\]
Here $\CL_d, \CL_e$ are line bundles on $\overline{\CJ}^d_C, \overline{\CJ}^{e}_C$ respectively. Moreover, for any $b \in U$ with $C_b$ nonsingular, the restrictions of these line bundles to the fibers over $b$ have homologically trivial first Chern classes. 
\end{cor}

\subsection{Multiplicativity and perversity} \label{Sec4.4}

The goal of this section is to provide an extension of Theorem \ref{thm:main} and Corollary \ref{motivic_decomp} for the degree $d$ compactified Jacobian $\pi_d: \overline{J}^d_C \to B$. Note that we chose to first work out the untwisted case for purely expository reasons, so that the Fourier theory can be streamlined and separated from the language of stacks. The twisted case requires the following three additional ingredients.

%\begin{cor} \label{cor:main}
%All four statements in Theorem \ref{thm:main}, as well as Corollary \ref{motivic_decomp}, hold equally for~$\pi_d: \overline{J}^d_C \to B$.
%\end{cor}

\begin{enumerate}
\item[(i)] Since we assume $\overline{J}^d_C$ to be a quasi-projective variety, by \cite{dJ} the Brauer map
\[
\mathrm{Br}(\overline{J}^d_C) \to H^2(\overline{J}^d_C, \BG_m)_{\mathrm{tors}}
\]
is an isomorphism. Then \cite[Theorem~3.6]{EHKV} implies that the $\mu_r$-gerbe $\overline{\CJ}^d_C$ over $\overline{J}^d_C$ is a quotient stack. This also applies to fiber products of $\overline{\CJ}^d_C$ for various $d$ over the base scheme~$B$.

\item[(ii)] The intersection theory of quotient stacks has been developed in a series of papers of Edidin and Graham; we refer to \cite{E} for a summary of the theory and the references therein. In particular, the Riemann--Roch functor $\tau(-)$ for quotient stacks was constructed in \cite{EG2} and was shown to satisfy the same functorial properties as in Section~\ref{Sec2.3} except for 
the covariance under arbitrary proper morphisms. The Edidin--Graham functor $\tau(-)$ is only covariant under proper \emph{representable} morphisms of quotient stacks.

We also note that with $\BQ$-coefficients, the Edidin--Graham Chow groups of a Deligne--Mumford quotient stack are canonically isomorphic (via pushforward) to the Chow groups of the coarse moduli space; see \cite{EG, V}.

\item[(iii)] The $K$-theory of quotient stacks has been studied by Anderson, Gonzales, and Payne \cite{AP, AGP} in the style of Fulton and Edidin--Graham. For our purposes we shall need the construction of Adams operations and the scaling \eqref{eq:scale} for closed embeddings of (possibly singular) quotient stacks, both of which were established in \cite[Section 4]{AGP}.
\end{enumerate}

We are now ready to run the entire argument of Sections \ref{Sec2} and \ref{comp_jac} for $\pi_d: \overline{J}^d_C \to B$.

\medskip
\noindent {\bf Step 1.} Choosing an arbitrary integer $e$, we consider the Poincar\'e sheaf $\CP_{e, d}$ and its inverse~$\CP_{e, d}^{-1}$. We define the Fourier transforms
\[
\FF_{e, d} \in \mathrm{CH}_*(\overline{\CJ}^e_C \times_B \overline{\CJ}^d_C) = \mathrm{CH}_*(\overline{J}^e_C \times_B \overline{J}^d_C), \quad \FF_{e, d}^{-1} \in \mathrm{CH}_*(\overline{\CJ}^d_C \times_B \overline{\CJ}^e_C) = \mathrm{CH}_*(\overline{J}^d_C \times_B \overline{J}^e_C)
\]
following the recipe of \eqref{eq:defFF} and \eqref{eq:defFF-1} and using the Edidin-Graham functor $\tau(-)$.

To prove the identities parallel to Lemma \ref{lem:ff-1}, we first note that the Fourier--Mukai kernel~$\CE$ of $\mathrm{id}_{D^b\mathrm{Coh}(\overline{\CJ}^e_C)_{(-d)}}$ is given by a line bundle on a $\mu_r \times \mu_r$-gerbe over the relative diagonal $\overline{J}^e_C \subset \overline{J}^e_C \times_B \overline{J}^e_C$. The line bundle, which corresponds to the character $(\lambda, \lambda') \mapsto \lambda^d\lambda'^{-d}$ of~$\mu_r \times \mu_r$, has its $r$-th tensor power the trivial bundle. Hence by the functorialities of~$\tau(-)$, we find that $\tau(\CE)$ (when pushed forward to $\overline{J}^e_C \times \overline{J}^e_C$) is as before given by
\[
\mathrm{td}(q_1^{e*}T_{\overline{J}^e_C}) \cap [\Delta_{\overline{J}^e_C/B}] \in \mathrm{CH}_*(\overline{J}^e_C \times \overline{J}^e_C)
\]
up to a factor of $1/r^2$, where $q_1^e: \overline{J}^e_C \times_B \overline{J}^e_C \to \overline{J}^e_C$ is the first projection. To compute $\tau(\CP_{e, d}^{-1} \circ \CP_{e, d})$ we observe that there is exactly one appearance of a non-representable proper pushforward in the whole process, namely the pushforward to the first and the third factors. On the other hand, since the object
\[
p_{12}^*\CP_{e, d} \otimes p_{23}^*\CP^{-1}_{e, d} \in D^b\mathrm{Coh}(\overline{\CJ}^e_C \times_B \overline{\CJ}^d_C \times \overline{\CJ}^e_C)_{(d, 0, -d)}
\]
descends to an object in $D^b\mathrm{Coh}(\overline{\CJ}^e_C \times_B \overline{J}^d_C \times_B \overline{\CJ}^e_C)$, it suffices to apply the covariance of~$\tau(-)$ under the proper \emph{representable} morphism
\[
p_{13}: \overline{\CJ}^e_C \times_B \overline{J}^d_C \times_B \overline{\CJ}^e_C \to \overline{\CJ}^e_C \times_B \overline{\CJ}^e_C.
\]
This yields the identity
\[
\FF_{e, d}^{-1} \circ \FF_{e, d} = \frac{1}{r^4}[\Delta_{\overline{J}^e_C/B}] \in \Corr^0_B(\overline{J}^e_C, \overline{J}^e_C),
\]
where the factor of $1/r^4$ is due to the fact that each $\overline{\CJ}^d_C \to \overline{J}^d_C$ is of degree $1/r$. The other identity is parallel.

\medskip
\noindent {\bf Step 2.} The next logical step is the Fourier vanishing (FV). We first remark that the Arinkin-type dimension bounds of Sections \ref{AV} and \ref{Sec3.5} are properties of coherent sheaves which can be checked \'etale locally.

To proceed with the Adams argument, we consider a sequence of~$K$-theory classes supported in codimension $g$:
\begin{equation} \label{eq:adamstw}
\CP_{e, d}^{-1} \circ \psi^N(\CP_{e, d}) \in K_*(\overline{\CJ}^e_C \times_B \overline{\CJ}^{e}_C), \quad N \equiv 1 \pmod{r},
\end{equation}
where $\psi^N(-)$ stand for the Adams operations of Anderson--Gonzales--Payne. Here the congruence condition on $N$ guarantees that the class
\[
p_{12}^*\psi^N(\CP_{e, d}) \otimes p_{23}^*\CP^{-1}_{e, d} \in K_*(\overline{\CJ}^e_C \times_B \overline{\CJ}^d_C \times_B \overline{\CJ}^{e}_C)_{(d, 0, -d)}
\]
descends to a class in $K_*(\overline{\CJ}^e_C \times_B \overline{J}^d_C \times_B \overline{\CJ}^{e}_C)$ so that the Riemann--Roch calculations work in the same way as in Step 1. We then repeat the calculations of Section \ref{Sec3.5.3} using the infinite sequence \eqref{eq:adamstw}, and deduce the Fourier vanishing
\begin{equation}
\tag{FV} \quad \quad (\FF_{e, d}^{-1})_i \circ (\FF_{e, d})_j = 0,\quad  i+j < 2g.
\end{equation}

Consequently, we obtain pairs of orthogonal projectors $\Fp_k^{e, d}, \Fq_{k + 1}^{e,d}$ (depending on $e$) and the corresponding motives $P_k^{e}h(\overline{J}_C^d), Q_{k + 1}^{e}h(\overline{J}_C^d)$. We also deduce the Perverse $\supset$ Chern statement of Theorem~\ref{thm:main} for $\pi_d: \overline{J}^d_C \to B$, as well as the decomposition part of Corollary \ref{motivic_decomp}.
%\begin{equation} \label{eq:motdectw}
%h(\overline{J}_C^d) = \bigoplus_{i = 0}^{2g} h_i^e(\overline{J}_C^d) \in \mathrm{CHM}(B).
%\end{equation}

\medskip
\noindent {\bf Step 3.}
For the multiplicativity statement we choose two arbitrary integers $e_1, e_2$. We consider the convolution kernel
\begin{equation*} \label{eq:convtw}
\CK_{e_1, e_2, d} \in D^b\mathrm{Coh}(\overline{\CJ}^{e_1}_C \times_B \overline{\CJ}^{e_2}_C \times_B \overline{\CJ}^{e_1 + e_2}_C)_{(d, d, -d)}
\end{equation*}
obtained from $\CP_{e_1, d}, \CP_{e_2, d}$ and $\CP^{-1}_{e_1 + e_2, d}$ together with \eqref{convolution}. As in Step 2 an \'etale local verification shows that $\CK_{e_1, e_2, d}$ is supported in codimension $g$. We define the Chow class $\FC_{e_1, e_2, d}$ following \eqref{eq:chowconv}.

The identity parallel to Lemma \ref{lem:conv} is proven in the same way as in Step 1. The Fourier--Mukai kernel of
\[
\otimes: D^b\mathrm{Coh}(\overline{\CJ}^d_C)_{(e_1)} \times D^b\mathrm{Coh}(\overline{\CJ}^d_C)_{(e_2)} \to D^b\mathrm{Coh}(\overline{\CJ}^d_C)_{(e_1 + e_2)}
\]
is given by a torsion line bundle on a $\mu_r \times \mu_r \times \mu_r$-gerbe over the small relative diagonal $\overline{J}^d_C \subset \overline{J}^d_C \times_B \overline{J}^d_C \times_B \overline{J}^d_C$. On the other hand, the computation of
\[
\tau\left(\CP_{e_1 + e_2, d} \circ \CK_{e_1, e_2, d} \circ (\CP^{-1}_{e_1, d} \boxtimes \CP^{-1}_{e_2, d})\right)
\]
only involves pushforwards via proper representable morphisms.

The rest of the argument is identical to the untwisted case in Section \ref{Sec2.5.3}.

%The indices of \eqref{eq:convtw} (together with Proposition \ref{prop4.1}) guarantees that the object
%\[
%p_{13}^*\CP_{e_1, d} \otimes p_{23^*}\CP_{e_2, d} \otimes p_{34}^*\CP^{-1}_{e_1 + e_2, d} \in D^b\mathrm{Coh}(\overline{\CJ}^{e_1}_C \times_B %\overline{\CJ}^{e_2}_C \times_B \overline{\CJ}^{d}_C \times_B \overline{\CJ}^{e_1 + e_2}_C)
%\]
%descends to an object in $D^b\mathrm{Coh}(\overline{\CJ}^{e_1}_C \times_B \overline{\CJ}^{e_2}_C \times_B \overline{J}^{d}_C \times_B \overline{\CJ}^{e_1 + e_2}_C)$, so that the calculations as in Lemma \ref{lem:conv} (see also Remark \ref{altconv}) only involve one more proper \emph{representable} pushforward via
%\[
%p_{124}: \overline{\CJ}^{e_1}_C \times_B \overline{\CJ}^{e_2}_C \times_B \overline{J}^{d}_C \times_B \overline{\CJ}^{e_1 + e_2}_C \to \overline{\CJ}^{e_1}_C \times_B \overline{\CJ}^{e_2}_C \times_B \overline{\CJ}^{e_1 + e_2}_C.
%\]

\medskip
\noindent {\bf Step 4.} Finally, we verify for $\pi_d: \overline{J}^d_C \to B$ the realization statement of Theorem \ref{thm:main} which implies the realization part of Corollary \ref{motivic_decomp}. Ng\^o's support theorem holds equally in the twisted case: all perverse sheaves appearing in the decomposition theorem for $\pi_d: \overline{J}^d_C \to B$ have full support. Then, as in Section \ref{Sec2.5.2}, it suffices to check the homological realization of~$P_k^{e}h(\overline{J}_C^d)$ over an open subset of the base $B$ supporting nonsingular curves. Note that by \'etale descent of perverse sheaves, we can even replace the open subset by an \emph{\'etale} neighborhood~$U$ of $B$.

We are placed in the setup of Section \ref{Sec4.3}. By Corollary \ref{cor4.3}, over $U$ the Poincar\'e sheaf~$\CP_{e,d}$ differs from the normalized $\CP$ only by line bundles which are fiberwise homologically trivial. Now a proof identical to that of Proposition \ref{prop3.6} shows that the homological realization of~$P_k^{e}h(\overline{J}_C^d)$ is precisely ${^\mathfrak{p}}\tau_{\leq k+\dim B} \pi_{d*} \BQ_{\overline{J}^d_C}$. The homological realization of $Q_{k + 1}^{e}h(\overline{J}_C^d)$ is parallel. To conclude, we have just obtained the following version of Theorem \ref{thm:main} and Corollary~\ref{motivic_decomp} for~$\pi_d: \overline{J}^d_C \to B$.

%This finishes the proof of Corollary \ref{cor:main}. \qed

\begin{cor} \label{cor:main}
The following hold for $\pi_d: \overline{J}^d_C \to B$.
\begin{enumerate}
\item[(i)] (Decomposition) For each $e$ and $k$, there is a decomposition of motives
\[
h(\overline{J}_C^d) = P_k^eh(\overline{J}_C^d) \oplus Q_{k + 1}^eh(\overline{J}_C^d) \in \mathrm{CHM}(B)
\]
with $P_k^eh(\overline{J}_C^d) = (\overline{J}_C^d, \Fp_k^{e, d}, 0),\, Q_{k + 1}^eh(\overline{J}_C^d) = (\overline{J}_C^d, \Fq_{k + 1}^{e,d}, 0)$.

\item[(ii)] (Realization) For each $e$ and $k$, the homological realization of $P_k^eh(\overline{J}_C^d)$ together with the inclusion $P_k^eh(\overline{J}_C^d) \to h(\overline{J}_C^d)$ is the natural morphism
\[
{^\mathfrak{p}}\tau_{\leq k+\dim B} \pi_{d*} \BQ_{\overline{J}_C^d} \to \pi_{d*}\BQ_{\overline{J}_C^d}.
\]
Similarly, the homological realization of $h(\overline{J}_C^d) \to Q_{k + 1}^eh(\overline{J}_C^d)$ is the natural morphism
\[
\pi_{d*}\BQ_{\overline{J}_C^d} \to {^\mathfrak{p}}\tau_{\geq k + 1 + \dim B} \pi_{d*} \BQ_{\overline{J}_C^d}.
\]

\item[(iii)] (Multiplicativity) For each pair of $e_1, e_2$, and each pair of $k, l$, the cup-product
\[
\cup: h(\overline{J}_C^d) \times h(\overline{J}_C^d) \to h(\overline{J}_C^d)
\]
projects to zero on
\[
\cup: P_k^{e_1}h(\overline{J}_C^d) \times P_l^{e_2}h(\overline{J}_C^d) \to Q_{k + l + 1}^{e_1 + e_2}h(\overline{J}_C^d).
\]

\item[(iv)] (Perverse $\supset$ Chern) For each $e$ and $k$, the morphism $(\FF_{e, d})_k: h(\overline{J}_C^e)(g - k) \to h(\overline{J}_C^d)$ projects to zero on
\[
(\FF_{e, d})_k: h(\overline{J}_C^e)(g - k) \to Q_{k + 1}^eh(\overline{J}_C^d).
\]
\end{enumerate}
\end{cor}

\begin{cor} \label{cor:motdectwist}
There exists a decomposition of motives
\[
h(\overline{J}_C^d) = \bigoplus_{i = 0}^{2g}h_i(\overline{J}_C^d) \in \mathrm{CHM}(B)
\]
whose homological realization recovers the decomposition theorem for $\pi_d: \overline{J}^d_C \to B$.
\end{cor}

\subsection{Abel--Jacobi maps}

In addition, we treat the Abel--Jacobi map which recovers the universal family of $1$-dimensional sheaves from the Poincar\'e sheaf. We follow essentially \cite[Proposition~3.1(ii)]{ADM}.

Recall the Abel--Jacobi map to the degree $1$ compactified Jacobian associated with $C \to B$:
\[
\mathrm{AJ}: C \rightarrow \overline{J}^1_C, \quad  ( x \in C_b ) \mapsto \mathfrak{m}^\vee_{x/C_b}.
\]
Pulling back along the $\mu_r$-gerbe $\overline{\CJ}^1_C \to \overline{J}^1_C$, we obtain the Abel--Jacobi map between Deligne--Mumford stacks
\[
\mathrm{AJ}: \CC \to \overline{\CJ}^1_C.
\]
Here $\CC$ is a $\mu_r$-gerbe over $C$ with the structure map $\sigma_\CC: \CC \to C$.

For any integer $d$, we consider the universal family $\CF^d$ normalized along the multisection~$D$ (see Section \ref{Sec4.2}), and the Poincar\'e sheaf $\CP_{1,d}$ on $\overline{\CJ}_C^1\times_B \overline{\CJ}_C^d$. The Abel--Jacobi map above induces
\[
\mathrm{AJ} \times_B \mathrm{id}_{\overline{\CJ}^d_C}: \CC \times_B \overline{\CJ}^d_C \to \overline{\CJ}_C^1\times_B \overline{\CJ}_C^d.
\]

\begin{prop}[\emph{c.f.}~{\cite[Proposition 3.1(ii)]{ADM}}]\label{Prop4.6}
We have 
\[
(\mathrm{AJ} \times_B \mathrm{id}_{\overline{\CJ}^d_C})^*\CP_{1,d} \simeq (\sigma_\CC\times_B\mathrm{id}_{\overline{\CJ}^d_C})^*\CF^d\otimes p_\CC^* \CN \otimes \CB.
\]
Here $\CN$ is a line bundle on $\CC$ given by a $\BQ$-divisor proportional to $D$, and $\CB$ is the pullback of a line bundle on the base $B$.
\end{prop}

\begin{proof}
We consider 
\[
\mathrm{id}_C \times_B \mathrm{AJ}: C\times_B \CC \to C\times_B \overline{\CJ}^1_C.
\]
By the definition of the Abel--Jacobi map, we may write
\begin{equation}\label{eqn60}
(\mathrm{id}_C \times_B \mathrm{AJ})^* \CF^1 \simeq (\mathrm{id}_C \times_B \sigma_\CC)^* \CI^\vee_\Delta \otimes p^*_\CC\CN'
\end{equation}
where $\CN'$ is a line bundle on $\CC$, and $\CI_\Delta$ is the ideal sheaf of the relative diagonal on $C \times_B C$. The calculation in the proof of \cite[Proposition 3.1(ii)]{ADM} shows that the difference between
\[(\mathrm{AJ} \times_B \mathrm{id}_{\overline{\CJ}^d_C})^*\CP_{1,d}, \quad (\sigma_\CC\times_B \mathrm{id}_{\overline{\CJ}^d_C})^*\CF^d
\]
is given by $d$-th power of the line bundle $\CN'$ of (\ref{eqn60}); see (3.9) and the last equation of the proof of \cite[Proposition 3.1(ii)]{ADM}.

It suffices to show that, modulo a line bundle pulled back from $B$, $\CN'$ is given by a divisor proportional to $D$. We calculate $\CN'$ via (\ref{eqn60}). By the condition that $\CF^{1}$ is trivialized along the multisection, the right-hand side of (\ref{eqn60}) is trivialized along $D \subset C$ (of the first factor). The desired statement then follows from a direct calculation
\[
\mathrm{det}\left(p_{C*}(\CI^\vee_\Delta \big{|}_{D \times_B C} )\right) \simeq \CO_C(D) \otimes \CB 
\]
with $\CB$ a line bundle pulled back from the base $B$. This completes the proof.
\end{proof}

\begin{rmk}\label{rmk4.7}
In the case when there is a (non-canonical) universal family $F^d$ of rank $1$ degree~$d$ torsion-free sheaves on $C\times_B\overline{J}^d_C$, as in Proposition \ref{prop4.2} there is a $B$-morphism 
\[
\widetilde{\sigma}_d: \overline{\CJ}^d_C \to \overline{J}^d_C 
\]
so that
\[
(\mathrm{id}_C\times_B \widetilde{\sigma}_d)^* F^d \simeq \CF^d \otimes p_{\CJ}^* \CT, \quad \CT \in \mathrm{Pic}( \overline{\CJ}^d_C ).
\]
Combining with Proposition \ref{Prop4.6}, we may relate the Chern characters of $F^d$ and $\CP_{1,d}$. This will be applied in Section \ref{applications} connecting the Fourier transforms and the tautological classes associated with moduli of $1$-dimensional sheaves or Higgs bundles.
\end{rmk}

\section{Moduli of one-dimensional sheaves on surfaces} \label{applications}

\subsection{Overview}
This section concerns applications of our main results. After a general discussion about curves on surfaces and tautological classes, we deduce half of the $P=C$ conjecture for $\BP^2$ (Theorem \ref{thm0.4}), as well as the $P = W$ conjecture of $\mathrm{GL}_r$ (Theorem \ref{P=W!}). From now on, all Chern characters (resp.~cycle classes) take value in cohomology (resp.~Borel--Moore homology).

\subsection{Curves on a surface}\label{Sec5.2} 
Let $S$ be a nonsingular projective surface, and let $d$ be an integer. We assume that $C \to B$ is a flat family of integral curves lying in $S$, with the evaluation map
\[
\mathrm{ev}: C \to S, \quad (x\in C_b \subset C) \mapsto (x \in C_b \subset S).
\]
Assume $H\subset S$ is a divisor which does not contain any curve $C_b$ in the family. Then $H\subset S$ yields a multisection 
\[
D := \mathrm{ev}^{-1}(H) \subset C \to B
\]
which we fix from now on.

As in Sections~\ref{comp_jac} and \ref{Sec:twist}, we denote by $\pi_d: \overline{J}^d_C \to B$ the associated compactified Jacobian, and we further assume that both~$C$ and $\overline{J}_C^d$ are nonsingular. There are the (stacky) Abel--Jacobi~maps
\[
\mathrm{AJ}: C \to \overline{J}^1_C, \quad \mathrm{AJ}: \CC  \to \overline{\CJ}^1_C.
\]

We consider the closed embeddings
\begin{equation}\label{ev/AJ}
\overline{\mathrm{ev}}: = {\mathrm{ev}}\times_B \mathrm{id}_{\overline{J}^d_C}: C \times_B \overline{J}^d_C \to S \times \overline{J}^d_C, \quad \overline{\mathrm{AJ}}:= {\mathrm{AJ}}\times_B \mathrm{id}_{\overline{\CJ}^d_C}: \CC \times_B\overline{\CJ}^d_C \to \overline{\CJ}^1_C \times_B\overline{\CJ}^d_C.
\end{equation} 
Assume there is a universal sheaf
\[
F^d \rightsquigarrow C\times_B \overline{J}^d_C.
\]
It gives a family of $1$-dimensional sheaves on $S$: 
\begin{equation}\label{univ_F}
\overline{F}^d:= \overline{\mathrm{ev}}_* F^d \rightsquigarrow S \times \overline{J}^d_C.
\end{equation}
For our purpose, we want to express the Chern character $\mathrm{ch}(\overline{F}^d)$ in terms of the homological Fourier transform
\[
\mathfrak{F} = \sum_i \mathfrak{F}_i \in H_*^{\mathrm{BM}}(\overline{\CJ}^1_C \times_B \overline{\CJ}^d_C, \BQ), \quad \mathfrak{F}_i \in H_{2(\dim B+ 2g - i)}^{\mathrm{BM}}(\overline{\CJ}^1_C \times_B \overline{\CJ}^d_C, \BQ).
\]
This would allow us to understand the tautological classes associated with (\ref{univ_F}) in terms of the Fourier transform.

Combining Proposition \ref{Prop4.6} and Remark \ref{rmk4.7}, we may compare the sheaves 
\begin{equation}\label{two_sheaves}
F^d \rightsquigarrow C\times_B \overline{J}^d_C, \quad \overline{\mathrm{AJ}}^*\CP_{1,d} \rightsquigarrow \CC \times_B \overline{\CJ}^d_C
\end{equation}
via the maps 
\[
\CC \times_B \overline{\CJ}^d_C 
\xrightarrow{~~\sigma_\CC \times_B \mathrm{id}_{\overline{\CJ}^d_C}~~} C\times_B \overline{\CJ}^d_C \xrightarrow{~~\mathrm{id}_C \times_B\widetilde{\sigma}_d~~} C\times_B \overline{J}^d_C.
\]
More precisely, the pullback of $\overline{F}^d$ to $\CC \times_B \overline{\CJ}^d_C$ coincides with $(\mathrm{AJ}\times_B\mathrm{id}_\CJ)^*\CP_{1,d}$ up to a line bundle pulled back from $\overline{\CJ}^d_C$, and a line bundle pulled back from $\CC$ given by a divisor proportional to $D$. We identify the rational Borel--Moore homology groups using the maps~above:
\[
H^\mathrm{BM}_*(\CC \times_B \overline{\CJ}^d_C, \BQ)= {H}^{\mathrm{BM}}_*(C\times_B \overline{J}^d_C, \BQ), 
\]
which allows us to view 
\[
\overline{\mathrm{AJ}}^* \FF \in H^\mathrm{BM}_*(C\times_B \overline{J}^d_C, \BQ).
\]

We also need to treat the Todd class. Let $l_S \in H^2(S, \BQ)$ be the class of a curve $C_b \subset C$. The normal bundle of the closed embedding $C \hookrightarrow S \times B$ has first Chern class $p_S^* l_S+ p_B^*l_B$. Here $p_{S}, p_B$ are the natural projections from $S\times B$, and $l_B$ is a class in $H^2(B, \BQ)$. Since $\overline{\mathrm{ev}}$ is obtained as the base change of $C \hookrightarrow S \times B$, the Todd class with respect to $\overline{\mathrm{ev}}$ is given by
\begin{equation}\label{todd}
\mathrm{td}(T_{\overline{\mathrm{ev}}}) = \frac{\overline{\mathrm{ev}}^*{l_0}}{1-e^{-\overline{\mathrm{ev}}^*{l_0}}}= 1+ \frac{1}{2}\overline{\mathrm{ev}}^*{l_0}+ \cdots, \quad {l_0} := p_S^* l_S + p_J^*\pi_d^* l_B \in H^2(S\times \overline{J}_C^d, \BQ).
\end{equation}

Now we can state the connection between the Chern character of (\ref{univ_F}) and $\FF$.

\begin{prop}\label{prop3.1} Consider the morphism $\overline{\pi}_d:= \mathrm{id}_S \times \pi_d: S \times \overline{J}^d_C \to S\times B$.
\begin{enumerate}
    \item[(i)]  We have
\[
\mathrm{ch}(\overline{F}^d) = \left( \overline{\mathrm{ev}}_* \left(  e^{\lambda\cdot p^*_{C}D} \cap  \overline{\mathrm{AJ}}^* \mathfrak{F}\right)\right) \cup  \left(\frac{{l_0}}{1- e^{-{l_0}}} \right) \cup e^{p_J^*l_J} \in H^*(S\times \overline{J}^d_C, \BQ).
\]
Here $\lambda \in \BQ$ is a constant, and $l_J \in H^2(\overline{J}^d_C, \BQ)$.
\item[(ii)] Let $P_\bullet H^*(S\times \overline{J}_C^d, \BQ)$ be the perverse filtration associated with $\overline{\pi}_d$.\footnote{This perverse filtration is also multiplicative with respect to the cup-product by Theorem \ref{thm0.1}(i).} For any class $\beta \in H^*(C, \BQ)$, we have
\[
\overline{\mathrm{ev}}_*\left( 
p_C^*\beta \cap
\overline{\mathrm{AJ}}^* \mathfrak{F}_k   \right)\in P_kH^{\geq 2k+2}(S\times \overline{J}_C^d, \BQ).
\]
\end{enumerate}
\end{prop}

\begin{proof}
Statement (i) follows from a Grothendieck--Riemann--Roch calculation, where we used that the maps (\ref{ev/AJ}) are l.c.i.~morphisms, and the class $p_C^*D$ is pulled back from $S \times \overline{J}^d_C$ via $\overline{\mathrm{ev}}$. Note that the factor $e^{\lambda\cdot p^*_{C}D}$ comes from $\CN$ in Proposition \ref{Prop4.6}, the factor $\frac{{l_0}}{1- e^{-{l_0}}}$ comes from the Todd class (\ref{todd}), and the factor $e^{p_J^*l_J}$ comes from the line bundle on $\overline{\CJ}^d_C$ arising from the difference of the two sheaves (\ref{two_sheaves}).

Now we prove (ii). The bound for the cohomological degree is clear. It suffices to show that for any class $\alpha \in H^*(S, \BQ)$, we have
\[
p_{J*}\left(p_S^* \alpha \cup  \overline{\mathrm{ev}}_* \left(  
p_C^*\beta \cap
\overline{\mathrm{AJ}}^* \mathfrak{F}_k  \right)\right) \in P_kH^*(\overline{J}_C^d, \BQ).
\]
By the projection formula, this is equivalent to 
\[
\mathfrak{F}_k\left( {\mathrm{AJ}}_* \left(\mathrm{ev}^*\alpha \cup \beta \right) \right) \in P_k H^*(\overline{J}_C^d, \BQ),
\]
which is given by Corollary \ref{cor:main} (the Perverse $\supset$ Chern part).
\end{proof}

Using Proposition \ref{prop3.1}(i), the tautological classes that appear in Hitchin systems \cite{dCHM1}, moduli of $1$-dimensional sheaves on $\BP^2$ \cite{KPS,PS}, and compactified Jacobians \cite{OY} are all governed by the Fourier transform $\mathfrak{F}$; moreover, their interactions with the perverse filtrations are governed by (ii). We discuss some applications in the following sections.

\subsection{The $P=C$ conjecture for \texorpdfstring{$\BP^2$}{P^2}} \label{Sec5.3}

We prove in this section Theorem \ref{thm0.4}. Assume $r\geq 3$.

\subsubsection{The locus of integral curves} For two coprime integers $r,\chi$, recall the moduli space $M_{r,\chi}$ parameterizing stable $1$-dimensional sheaves $F$ with 
\[
[\mathrm{supp}(F)] =rH, \quad \chi(F) = \chi.
\]
The stability is with respect to the slope
\[
\mu(F) = \frac{\chi(F)}{H\cdot [\mathrm{supp}(F)]}.
\]
The moduli space $M_{r,\chi}$ admits a natural Hilbert--Chow map
\[
h: M_{r,\chi} \to \BP H^0(\BP^2, \CO_{\BP^2}(r))
\]
sending $F$ to its Fitting support. For a point in $\BP H^0(\BP^2, \CO_{\BP^2}(r))$ represented by a nonsingular degree $r$ plane curve, its preimage with respect to $h$ is isomorphic to its Jacobian. Therefore, $h$ can be viewed as a degenerating family of abelian varieties.

Let $W \subset \BP H^0(\BP^2, \CO_{\BP^2}(r))$ be the open subset parameterizing degree $r$ integral curves.

\begin{lem}\label{lem3.2}
We have
\[
\mathrm{codim}_{\BP H^0(\BP^2, \CO_{\BP^2}(r))} (\BP H^0(\BP^2, \CO_{\BP^2}(r))\backslash W) = r-1.
\]
\end{lem}

\begin{proof}
This is well-known: the component given by the closure of reducible curves with two components of bidegrees $(1,r-1)$ has the largest dimension among all the boundary components.
\end{proof}

Now we consider the restriction of $h$ to the locus of integral curves
\[
h_W: M_W: = h^{-1}(W) \to W.
\]

\begin{cor}\label{cor3.3}
The restriction map induces an isomorphism of filtered vector spaces
\[
\mathrm{res}_W: P_kH^{\leq 2(r-2)}(M_{r,\chi}, \BQ) \xrightarrow{~~\simeq~~} P_kH^{\leq 2(r-2)}(M_W, \BQ).
\]
\end{cor}

\begin{proof}
Assume $i:Z:= M_{r,\chi} \setminus M_W \hookrightarrow M_{r,\chi}$ and $j: M_W \hookrightarrow M_{r,\chi}$ are the natural closed and open embeddings. We consider the exact triangle
\[
i_*i^!\BQ_{M_{r,\chi}} \to \BQ_{M_{r,\chi}} \to j_*j^*\BQ_{M_{r,\chi}} \to i_*i^!\BQ_{M_{r,\chi}}[1]
\]
which further yields
\begin{equation}\label{3.3_1}
h_*i_*i^!\BQ_{M_{r,\chi}} \to h_*\BQ_{M_{r,\chi}} \xrightarrow{~~(*)~~} j_*h_{W*}\BQ_{M_W} \to h_*i_*i^!\BQ_{M_{r,\chi}}[1].
\end{equation}
The filtered morphism $\mathrm{res}_W$ is induced by taking the global cohomology of $(*)$.
By Lemma~\ref{lem3.2}, we have 
\[
i_*i^!\BQ_{M_{r,\chi}} \in D_c^{\geq 2(r-1)}(M_{r,\chi});
\]
in particular, both the first and the last terms of (\ref{3.3_1}) are concentrated in degrees $> 2(r-2)$. Therefore the restriction map 
\[
\mathrm{res}_W: H^{\leq 2(r-2)}(M_{r,\chi}, \BQ) \to H^{\leq 2(r-2)}(M_{r,\chi}, \BQ)
\]
is an isomorphism that preserves the perverse filtrations:
\[
\mathrm{res}_W \left(P_k H^{\leq 2(r-2)}(M_{r,\chi}, \BQ)\right) \subset P_kH^{\leq 2(r-2)}(M_{r,\chi}, \BQ).
\]
Furthermore, this inclusion has to be an isomorphism due to \cite[Lemma 3.3]{dCMS}. The corollary follows.
\end{proof}

\subsubsection{Tautological classes and $P=C$}\label{5.3.2}

We first review some structural results on the cohomology of $M_{r,\chi}$ and the $P=C$ conjecture of \cite{KPS}.

As in \cite{dCMS, KPS}, we consider twisted Chern characters. We fix $\BF$ to be a universal family over $\BP^2 \times M_{r,\chi}$. For a class 
\begin{equation*}\label{delta}
\delta \in p_{\BP^2}^*H^2(\BP^2, \BQ)\oplus p_M^*H^2(M_{r,\chi}, \BQ) \subset H^2(\BP^2\times M_{r,\chi}, \BQ), 
\end{equation*}
we define 
\[
\mathrm{ch}^\delta(\BF): = \mathrm{ch}(\BF) \cup e^{\delta} \in H^*(\BP^2 \times M_{r,\chi}, \BQ)
\]
as well as its degree $k$ part $\mathrm{ch}_k^\delta(\BF) \in H^{2k}(\BP^2 \times M_{r,\chi}, \BQ)$. This further induces the $\delta$-twisted tautological class
\[
c_k^\delta(j): = p_{M*}\left(p^*_{\BP^2} H^j \cup \mathrm{ch}_{k+1}^\delta(\BF) \right) \in H^{2(k + j - 1)}(M_{r,\chi}, \BQ).
\]
The reason for introducing the $\delta$-twisted tautological classes is the following. The universal family $\BF$ is not canonical; however, by \cite[Proposition 1.2]{KPS} there exists a unique $\delta_0$ as above such that $c^{\delta_0}_k(j)$ has perversity strictly equal to $k$ when $k+j\leq 2$. Then all the classes $c^{\delta_0}_k(j)$ are canonically defined which do not depend on the choice of $\BF$. We denote
\[
c_k(j) : = c_k^{\delta_0}(j).
\]
It was proven in \cite{PS} that the first $3r-7$ classes of cohomological degrees $\leq 2(r - 2)$, \emph{i.e.},
\[
c_k(j) \in H^{2(k+j-1)}(M_{r,\chi}, \BQ), \quad k+j \leq r-1
\]
generate $H^*(M_{r,\chi}, \BQ)$ as a $\BQ$-algebra, and there is no relation in degrees $\leq 2(r-2)$. Therefore the Chern filtration of Section \ref{0.3.2} is well-defined for $H^{\leq 2(r-2)}(M_{r,\chi}, \BQ)$. 

\begin{rmk}
    The bound $2(r-2)$ coincides with the bound obtained in Corollary \ref{cor3.3}. It is optimal in view of the free generation result \cite{PS}. However, this fact is crucial for our proof, but we do not have an easy geometric explanation of this coincidence. This bound is also optimal for the $P=C$ conjecture below; see \cite[Remark 0.5]{KPS}.
\end{rmk}

The main conjecture of \cite{KPS} connects the Chern filtration above with the perverse filtration associated with the map $h: M_{r,\chi} \to \BP H^0(\BP^2, \CO_{\BP^2}(r))$.

\begin{conj}[\cite{KPS} ``$P=C$'']\label{conj3.4}
We have
\[
P_kH^{\leq {2(r-2)}}(M_{r,\chi}, \BQ) = C_kH^{\leq 2(r-2)}(M_{r,\chi}, \BQ).
\]
\end{conj}

%\begin{rmk}
Before proceeding with the proof, we first discuss some motivations and consequences of the $P=C$ conjecture. By \cite{HST}, the dimensions of the graded pieces of the perverse filtration
\[
n_r^{i,j}: = \dim \mathrm{Gr}^P_iH^{i+j}(M_{r,\chi}, \BQ)
\]
calculate the refined BPS invariants for the local Calabi--Yau $3$-fold $\mathrm{Tot}(K_{\BP^2})$. In this case, conjecturally these invariants coincide with the refined BPS invariants calculated from refined Pandharipande--Thomas (PT) invariants. The PT calculation predicts that the invariants $n_{r}^{i,j}$ satisfy a product formula
\[
\sum_{i,j}n_r^{i,j}q^it^j = 
\prod_{i \geq 0} \frac{1}{(1-(qt)^i q^2) (1-(qt)^i q^2t^2) (1-(qt)^i t^2)}
\]
when $i+j \leq 2(r-2)$. This surprising product formula implies that, although every individual Gromov--Witten or BPS invariant for the local $\BP^2$ does not stabilize when $r \to +\infty$, their refinement does.

Conjecture \ref{conj3.4} gives a geometric explanation of this product formula; furthermore, Theorem~\ref{thm0.4} implies that the product formula and Conjecture \ref{conj3.4} are equivalent.
%\end{rmk}

\subsubsection{Proof of Theorem \ref{thm0.4}}\label{sec5.3.3}

Recall the restricted map $h: M_W \to W$ over the locus of integral curves, which can be identified with the compactified Jacobian 
\[
\pi_d: \overline{J}^d_C \to W, \quad  d= \chi -1+ \frac{(r-1)(r-2)}{2}
\]
associated with the family of degree $r$ integral planar curves $C \to W$. A hyperplane section $H \subset \BP^2$ induces a multisection $D$ of degree $r$ over the base $W$. 

By Corollary \ref{cor3.3}, it suffices to show that 
\[
\prod_{i=1}^s c_{k_i}(j_i) \in P_{\Sigma_{i=1}^sk_i} H^*(M_W, \BQ).
\]
Here the perverse filtration is defined via $\pi_d$, and the classes $c_{k_i}(j_i)$ are viewed as the restrictions of the corresponding classes to $\overline{J}^d_C$. Applying Corollary \ref{cor:main} (the multiplicativity part) to~$\pi_d: \overline{J}^d_C \to B$, the statement above is further reduced to treating each individual class
\begin{equation}\label{eqn43}
c_k(j) \in P_kH^*(\overline{J}^d_C, \BQ).
\end{equation}

The universal sheaf $\BF$ on $\BP^2 \times \overline{J}^d_C$ can be expressed as $\overline{F}^d$ of (\ref{univ_F}), so that we have
\[
\mathrm{ch}^{\delta_0}(\BF) = \mathrm{ch}(\overline{F}^d) \cup e^{\delta_0} ,\quad  
 \delta_0 : = p_{\BP^2}^* \delta_{\BP^2} + p_J^* \delta_J.
\]
Hence we may use Proposition \ref{prop3.1} to control the perversity of a tautological class.

We first note that the class $l_0$ associated with the Todd class (\ref{todd}) satisfies
\begin{equation}\label{l0}
l_0 \in P_0H^2(\BP^2 \times \overline{J}_C^d, \BQ) \subset P_1H^2(\BP^2 \times \overline{J}_C^d ,\BQ)
\end{equation}
since its component for $\overline{J}^d_C$ is pulled back from the base $B$.
By definition, we know that the K\"unneth component $H^2(\BP^2, \BQ) \otimes H^2(\overline{J}_C^d, \BQ)$ of $\mathrm{ch}^{\delta_0}_2(\BF)$ has perversity $1$ with respect to the perverse filtration associated with $\BP^2 \times \overline{J}_C^d \to \BP^2 \times B$. Therefore 
\begin{equation}\label{OmegaJ}
\omega_J:= \delta_J +l_J \in P_1H^2(\overline{J}_C^d, \BQ).
\end{equation}
Hence, if we express 
\[
\mathrm{ch}_{k+1}^{\delta_0}(\BF) \in H^{2k+2}( \BP^2 \times \overline{J}_C^d, \BQ )
\]
using Proposition \ref{prop3.1}(i) in terms of $\mathfrak{F}$, $\omega_J$, and $l_0$, we see that each term is of the form
\[
\overline{\mathrm{ev}}_*\left( 
p_C^*\beta_j \cap
\overline{\mathrm{AJ}}^* \mathfrak{F}_j   \right) \cup \gamma_j(\omega_J, l_0), \quad j \leq k.
\]
Here $\gamma_j(\omega_J, l_0)$ admits a polynomial expression in terms of $\omega_J$ and $l_0$. 

Now by Proposition \ref{prop3.1}(ii), the first term satisfies
\[
\overline{\mathrm{ev}}_*\left( 
p_C^*\beta_j \cap
\overline{\mathrm{AJ}}^* \mathfrak{F}_j   \right) \in P_{j}H^{\geq 2j+2}(\BP^2 \times \overline{J}_C^d, \BQ).
\]
By (\ref{l0}), (\ref{OmegaJ}), and the multiplicativity of the perverse filtration for $\BP^2 \times \overline{J}^d_C$, the second term satisfies
\[
\gamma_j(\omega_J, l_0) \in \bigoplus_i P_iH^{2i}( \BP^2\times \overline{J}^d_C, \BQ ).
\]
Using again the multiplicativity, we conclude that 
\[
\mathrm{ch}^{\delta_0}_{k+1}(\BF) \in P_k H^{2k+2}( \BP^2 \times \overline{J}^d_C, \BQ ),
\]
which implies (\ref{eqn43}). This completes the proof of Theorem \ref{thm0.4}.    \qed

\subsection{The $P=W$ conjecture} \label{Sec5.4}
In this section, we prove Theorem \ref{P=W!} via (\ref{P=W_taut}). Ideally, (\ref{P=W_taut}) should follow directly from Theorem \ref{thm0.1} for the Hitchin system, as in the proof of Theorem~\ref{thm0.4}. However, what we know so far is that the Hitchin system is a (twisted) dualizable abelian fibration over the locus of integral spectral curves on the base. Therefore, as in \cite{MS_PW, HMMS}, we reduce (\ref{P=W_taut}) to the parallel statement for certain parabolic moduli spaces. 

We shall use the reduction steps of Hausel--Mellit--Minets--Schiffmann \cite[Section 8]{HMMS} which we review in Section \ref{reduction}.

\subsubsection{Tautological classes and the Chern filtration}\label{sec5.4.1}

We first review briefly the tautological classes and the Chern filtration associated with the Dolbeault moduli space $M_{\mathrm{Dol}}$ of rank $r$ and degree~$n$. The discussion is parallel to Section \ref{5.3.2}.

Under the coprime condition $(r,n)=1$, the moduli space $M_{\mathrm{Dol}}$ admits a (non-canonical) universal bundle
\[
\BU \rightsquigarrow \Sigma \times M_{\mathrm{Dol}}
\]
of rank $r$. For a class 
\begin{equation*}
\delta = p_\Sigma^* \delta_\Sigma +p_M^*\delta_M \in p_{\Sigma}^*H^2(\Sigma, \BQ)\oplus p_M^*H^2(M_{\mathrm{Dol}}, \BQ) \subset H^2(\Sigma \times M_{\mathrm{Dol}}, \BQ), 
\end{equation*}
we define 
\begin{equation}\label{eqn}
\mathrm{ch}^\delta(\BU): = \mathrm{ch}(\BU) \cup e^{\delta} \in H^*(\Sigma \times M_{\mathrm{Dol}}, \BQ)
\end{equation}
as well as its degree $k$ part $\mathrm{ch}_k^\delta(\BU) \in H^{2k}(\Sigma \times M_{\mathrm{Dol}}, \BQ)$. We say that $(\BU, \delta)$ is \emph{normalized} if in the K\"unneth decomposition we have
\[
\mathrm{ch}^\delta_1(\BU) \in H^1(\Sigma, \BQ) \otimes H^1(M_{\mathrm{Dol}}, \BQ) \subset H^2(\Sigma \times M_{\mathrm{Dol}}, \BQ).
\]
The twisted Chern character (\ref{eqn}) is canonical; it does not depend on the choice of a pair $(\BU, \delta)$ as long as it is normalized. 

For any $\gamma \in H^*(\Sigma, \BQ)$ and a normalized pair $(\BU, \delta)$, we define the tautological class
\[
c_k(\gamma): = p_{M*}\left(p^*_{\Sigma} \gamma \cup \mathrm{ch}_{k}^\delta(\BU) \right) \in H^{*}(M_{\mathrm{Dol}}, \BQ).
\]
These classes generate $H^*(M_{\mathrm{Dol}}, \BQ)$ as a $\BQ$-algebra \cite{Markman}. Define the Chern filtration $C_kH^*(M_{\mathrm{Dol}}, \BQ)$ as the subspace of $H^*(M_{\mathrm{Dol}}, \BQ)$ spanned by 
\[
\prod_{i=1}^s c_{k_i}(\gamma_i) \in H^{*}(M_{\mathrm{Dol}}, \BQ), \quad \sum_{i=1}^s k_i \leq k.
\]
We call the integer $\sum_{i=1}^s k_i$ associated with the class above its Chern grading. Note that by considerations from the character variety \cite{Shende}, the ideal of the relations between the tautological classes is homogeneous with respect to the Chern grading. Therefore, unlike the case for $M_{d,\chi}$ associated with $\BP^2$, the Chern grading is well-behaved for the \emph{total} cohomology $H^*(M_{\mathrm{Dol}}, \BQ)$. As was commented in Section \ref{0.3.2}, the main result of \cite{Shende} further showed that the Chern filtration coincides with the weight filtration $W_{2k}H^*(M_B, \BQ)$ for the character variety.

\subsubsection{Moduli spaces}\label{5.4.2}
From now on we only work with the Dolbeault moduli space, therefore we write $M:= M_{\mathrm{Dol}}$ for notational convenience. We fix $r,n$, and only consider Hitchin type moduli spaces of rank $r$ and degree $n$ on the curve $\Sigma$. We also fix a point $p\in \Sigma$. The reduction techniques of \cite[Section 8]{HMMS} require several Hitchin type moduli spaces which we review here.

\begin{enumerate}
    \item[$\bullet$] $M$: the moduli space of stable Higgs bundles 
    \[
    (E, \theta), \quad  \theta: E \to E \otimes \Omega^1_\Sigma.
    \]
    \item[$\bullet$] $M^{\mathrm{mero}}$: the moduli space of stable meromorphic Higgs bundles 
    \[
    (E, \theta), \quad  \theta: E \to E\otimes \Omega^1_\Sigma(p).
    \]
    \item[$\bullet$] $M^{\mathrm{par}}$: the moduli space of stable parabolic Higgs bundles
    \[
    (E, \theta, F^\bullet),\quad \theta: E \to E \otimes \Omega^1_\Sigma(p), \quad 
    \]
    where $F^\bullet$ is a complete flag on the fiber $E_p$, and the residue $\theta_p$ preserves the flag.
    \item[$\bullet$] $M^0 \subset M^{\mathrm{par}}$: the moduli subspace given by the condition that the residue $\theta_p$ is nilpotent.
    \item[$\bullet$] $M^{\mathrm{parell}} \subset M^{\mathrm{par}}$: the moduli subspace given by the condition that the spectral curve associated with the parabolic Higgs bundle is integral and the residue $\theta_p$ has $n$ distinct eigenvalues over $p$.
    \item[$\bullet$] $\widetilde{M} \subset M^0$: the moduli subspace with trivial residue at the point $p$, \emph{i.e.}~$\theta_p=0$.
\end{enumerate}

We summarize the moduli spaces above by the diagram
\begin{equation}\label{Hitchin_moduli}
M \xleftarrow{~~f~~} \widetilde{M} \xhookrightarrow{~~\widetilde{\iota}~~} M^0 \xhookrightarrow{~~\iota_0~~} {M}^{\mathrm{par}} \xhookleftarrow{~~\iota~~} M^{\mathrm{parell}} \xrightarrow{~~q~~} \overline{J}^d_C. 
\end{equation}
Here $\widetilde{\iota}, \iota_0, \iota$ are natural inclusions. The first morphism $f: \widetilde{M} \to M$ is given by forgetting the flag. The symbol $C$ in the last object stands for 
\[
C\to W \subset \bigoplus_{i=1}^r H^0\left(\Sigma, \Omega^1_{\Sigma}(p)^{\otimes i}\right),
\]
the family of integral spectral curves in the surface $\mathrm{Tot}_\Sigma(\Omega^1_\Sigma(p))$ which intersects the fiber over~$p$ in $r$ distinct points. Then $\overline{J}^d_C$ is the open subvariety of $M^{\mathrm{mero}}$ given by the pre-image of $W$ under the Hitchin map; it is isomorphic to the compactified Jacobian fibration associated with $C \to W$ where the degree $d$ is determined by $r, n$ via a Riemann--Roch calculation. Finally, there is a natural $\mathfrak{S}_r$-action on $M^{\mathrm{parell}}$ permuting the complete flag; it is free due to the condition that the eigenvalues of $\theta_p$ are distinct. The last morphism $q: M^{\mathrm{parell}} \to \overline{J}^d_C$ is the quotient map with respect to this free $\mathfrak{S}_r$-action, which forms the Cartesian product
\begin{equation}\label{Cart}
\begin{tikzcd}
M^{\mathrm{parell}} \arrow[r, "q"] \arrow[d, "h"]
& \overline{J}^d_C \arrow[d, " "] \\
\widetilde{W} \arrow[r]
& W.
\end{tikzcd}
\end{equation}
Here $\widetilde{W}$ can be viewed as the parameter space of spectral curves lying in $W$ with a marking over $p\in \Sigma$, whose projection to $W$ is the natural $\mathfrak{S}_r$-quotient. 

%is finite and \'etale.

\begin{rmk}
For the reader's convenience, we list the notation used in \cite{HMMS} for the moduli spaces above. In the following, the first is the notation of this section, and the second is the notation of \cite[Section 8]{HMMS}: $M$ is $X$, $M^{\mathrm{par}}$ is either $\overline{M}_{r,d}$ or $X^\mathrm{par}$, $M^0$ is $M^0_{r,d}$, and $M^{\mathrm{parell}}$ is either $M^\mathrm{parell}_{r,d,D}$ or $M_{r,d}$.
\end{rmk}

Now for each Hitchin type moduli space above, we can use a universal bundle $\BU$ to define the normalized tautological classes 
\[
c_k(\gamma), \quad k \in \BN, \quad \gamma \in H^*(\Sigma, \BQ)
\]
identical to the definition in Section \ref{sec5.4.1}. Each moduli space above also admits a proper Hitchin map defined in the obvious way by calculating the characteristic polynomial of the Higgs field, from which we may define the corresponding perverse filtrations. 

\begin{cor}\label{cor} Let $c_k(\gamma)$ be the tautological classes on $M^{\mathrm{parell}}$. The operator on the cohomology given by the cup-product with respect to $c_k(\gamma)$ satisfies
\[
c_k(\gamma) \cup {} : P_iH^*(M^{\mathrm{parell}}, \BQ) \to P_{i+k}H^*(M^{\mathrm{parell}}, \BQ).
\]
\end{cor}

\begin{proof}
Since the left vertical arrow $h$ of (\ref{Cart}) is the pullback of the compactified Jacobian fibration $\overline{J}^d_C$ along a finite \'etale map $\widetilde{W} \to W$, Corollary \ref{cor:main} applies to it. In particular, the perverse filtration $P_\bullet H^*( M^{\mathrm{parell}}, \BQ)$ is multiplicative. So we only need to show that 
\[
c_k(\gamma) \in P_k H^*(M^{\mathrm{parell}}, \BQ).
\]
Moreover, since this class is pulled back from $\overline{J}^d_C$, in view of the diagram (\ref{Cart}) it suffices to prove the corresponding statement for $\overline{J}^d_C$:
\[
c_k(\gamma) \in P_k H^*(\overline{J}^d_C, \BQ).
\]
This is a consequence of Corollary \ref{cor:main} which is completely parallel to the proof of Theorem~\ref{thm0.4} in Section \ref{sec5.3.3}.

The only minor difference is that, in Section \ref{sec5.3.3} we considered a universal $1$-dimensional sheaf over a surface, and here we consider a universal vector bundle over a curve. The latter can be reduced to the surface case as we explain in the following.

Let $S$ be the projective bundle over the curve $\Sigma$ given by
\[
\mathrm{pr}: S:= \BP_\Sigma\left( \Omega^1_\Sigma(p) \oplus \CO_\Sigma\right) \to \Sigma.
\]
Then $C \to W$ can be viewed as a family of curves in the linear system $|r\Sigma|$ with $\Sigma \subset S$ the~$0$-section. A universal sheaf $\overline{F}^d$ on $S\times \overline{J}^d_C$ of Section \ref{Sec5.2} provides a universal bundle
\[
\BU: = (\mathrm{pr} \times \mathrm{id}_J)_* \overline{F}^d \rightsquigarrow \Sigma \times \overline{J}^d_C.
\]
In particular, the Grothendieck--Riemann--Roch formula with respect to 
\[
\overline{\mathrm{pr}}:=\mathrm{pr} \times \mathrm{id}_J: S\times \overline{J}^d_C \to \Sigma \times \overline{J}^d_C
\]
allows us to express the class $c_k(\gamma)$ in terms of the tautological classes associated with $S$ defined via $\overline{F}^d$. More precisely, by the formula
\[
\mathrm{ch}(\BU) = \overline{\mathrm{pr}}_*\left( \mathrm{ch}(\overline{F}^d) \cup p_S^* \mathrm{td}_{\mathrm{pr}}\right),
\]
every tautological class $c_k(\gamma)$ can be written in terms of K\"unneth components of 
\[
\mathrm{ch}^{\overline{\mathrm{pr}}^*\delta}_{j}(\overline{F}^d), \quad j \leq {k+1}.
\]
Hence it suffices to show that 
\[
\mathrm{ch}^{\overline{\mathrm{pr}}^*\delta}_{k + 1}(\overline{F}^d) \in P_{k}H^*(S\times \overline{J}^d_C, \BQ).
\]
This follows from an identical argument as in Section \ref{sec5.3.3} where we note that the normalization condition for $(\BU, \delta)$ ensures that the class (\ref{OmegaJ}) associated with $\overline{F}^d$ vanishes.
\end{proof}

\subsubsection{The reduction steps of Hausel--Mellit--Minets--Schiffmann}\label{reduction}
Finally, we complete the proof of Theorem \ref{P=W!} by reducing it to Corollary \ref{cor}. Clearly (\ref{P=W_taut}) can be deduced from the following analogue of Corollary \ref{cor} for the moduli space $M$:
\begin{equation}\label{P=W_final}
c_k(\gamma) \cup {} : P_iH^*(M, \BQ) \to P_{i+k}H^*(M, \BQ).
\end{equation}
If such a statement holds for a Hitchin type moduli space in Section \ref{5.4.2}, we say that this moduli space satisfies \emph{stronger $P \supset C$}. For the moduli space $M$, the stronger $P \supset C$ condition is equivalent to the weaker version (\ref{P=W_taut}) due to Markman's generation result \cite{Markman}. But for other spaces, these two conditions are not equivalent.

We now explain how the reduction steps of \cite[Section~8]{HMMS} reduces (\ref{P=W_final}) to Corollary \ref{cor}. The strategy is to keep track of the stronger $P \supset C$ condition through (\ref{Hitchin_moduli}) from the right end to the left end.

\medskip
\noindent {\bf Step 1.} Since all the moduli spaces of (\ref{Hitchin_moduli}) admit natural maps to $M^{\mathrm{mero}}$, we fix once and for all a normalized pair 
\[
(\BU, \delta) \rightsquigarrow \Sigma\times M^{\mathrm{mero}},
\]
whose pullback yields a normalized pair and the tautological classes $c_k(\gamma)$ for each moduli space.

By Corollary \ref{cor}, the stronger $P \supset C$ condition holds for $M^{\mathrm{parell}}$.

\medskip
\noindent {\bf Step 2.} The stronger $P \supset C$ condition holds for $M^0$.

This follows from Step 1 combined with the following facts:
\begin{enumerate}
    \item[(i)] the restriction map 
    \[
    \iota_0^*: H^*(M^{\mathrm{par}}, \BQ) \to H^*(M^0, \BQ)
    \]
    is a filtered isomorphism with respect to the perverse filteations and preserves tautological classes; see \cite[Proposition 8.18]{HMMS};
    \item[(ii)] the restriction map 
    \[
    \iota^*: H^*(M^{\mathrm{par}}, \BQ) \to H^*(M^{\mathrm{parell}}, \BQ)
    \]
    is a filtered injection with respect to the perverse filtrations and preserves tautological classes; see \cite[Proposition 8.16]{HMMS} and the paragraph after that.
\end{enumerate}

\medskip
\noindent {\bf Step 3.} The stronger $P\supset C$ condition holds for $M$, which completes the proof of (\ref{P=W_final}).

To see this, we consider the morphism
\[
\Gamma:= \widetilde{\iota}_* f^* : H^*(M, \BQ) \to H^*(M^0, \BQ).
\]
We note that this is only a correspondence which does not preserve cup-products. By \cite[Section 8.6, before Theorem 8.20]{HMMS} $\Gamma$ is injective with left inverse given by
\[
\Gamma' := f_*\widetilde{\iota}^*  : H^*(M^0, \BQ) \to H^*(M, \BQ)
\]
up to a nonzero constant.\footnote{In \cite[Section 8.6]{HMMS} the correspondences $\Gamma$ and $\Gamma'$ are called $B$ and $A$ respectively.}

Since the pullback of the tautological classes $c_k(\gamma)$ on $M$ and $M^0$ to $\widetilde{M}$ coincide, by the projection formula we have
\[
\Gamma\left(c_k(\gamma) \cup \alpha\right) = c_k(\gamma) \cup \Gamma(\alpha) \in H^*(M^0, \BQ)
\]
for any cohomology class $\alpha \in H^*(M, \BQ)$. Then the desired statement follows from Step 2 and the compatibility of $\Gamma, \Gamma'$ with the perverse filtrations \cite[Proposition 8.7]{HMMS}.

We have completed the proof of the $P=W$ conjecture. \qed

\end{document}